\documentclass[a4paper, 11pt, dvipsnames]{extarticle}

\usepackage[utf8]{inputenc}    
\usepackage[T1]{fontenc}
\usepackage[english]{babel}
\usepackage[normalem]{ulem}
\usepackage{lmodern, dsfont, bm}

\usepackage[margin=30pt]{subcaption}
\usepackage{float, graphicx, caption, wrapfig, tikz}
\usetikzlibrary{backgrounds}
\usepackage{amsthm, amsmath, amsfonts, amssymb, mathrsfs, stackrel, mathtools}
\usepackage[alphabetic]{amsrefs}

\usepackage{hyperref, xcolor, titlesec} 
\hypersetup{colorlinks = true, urlcolor = blue,linkcolor=BrickRed, citecolor=RoyalBlue, bookmarksopen = true}
\usepackage[nomarginpar]{geometry}
\numberwithin{equation}{section}

\newcommand{\D}{\mathrm{d}}
\newcommand{\I}{\mathrm{i}}
\newcommand{\scp}[2]{\langle #1,#2\rangle}
\DeclareMathOperator{\Tr}{Tr}

\theoremstyle{plain}
\newtheorem{theorem}{Theorem}[section]
\newtheorem{lemma}[theorem]{Lemma}

\theoremstyle{definition}
\newtheorem{remark}[theorem]{Remark}

\newtheorem{definition}[theorem]{Definition}

\titleformat{\paragraph}[runin]{\itshape\normalsize}{\theparagraph}{}{}
\titleformat{\subparagraph}[runin]{\itshape\normalsize}{\theparagraph}{0em}{}
\titleformat{\section}[block]{\normalfont\filcenter}{\Large\bf\thesection .}{.5em}{\Large\bf}
\titleformat{\subsection}[block]{\normalfont}{\large\bf\thesubsection .}{.5em}{\large \bf}
\titleformat{\subsubsection}[block]{\normalfont}{\bf\thesubsubsection .}{.5em}{\bf}

\begin{document}

\title{\textbf{Eigenvalue distribution of some nonlinear models of random matrices}}
\author{L. Benigni\\\vspace{-0.15cm}\footnotesize{\it{LPSM, Université Paris Diderot}}\\\footnotesize{\it{lbenigni@lpsm.paris}}\and S. Péché\thanks{S.P. is supported by the Institut Universitaire de France.}\\\vspace{-0.15cm}\footnotesize{\it{LPSM, Université Paris Diderot}}\\\footnotesize{\it{peche@lpsm.paris}}}
\date{}
\maketitle

\tikzset{every loop/.style={min distance=10mm,in=60,out=120,looseness=10}}
\tikzset{every node/.style={draw,circle, scale=1}}
\tikzset{every label/.append style={font=\fontsize{9}{10.8}\selectfont}}
\begin{abstract}
\small{This paper is concerned with the asymptotic empirical eigenvalue distribution of some non linear random matrix ensemble.
More precisely we consider $M= \frac{1}{m} YY^*$ with $Y=f(WX)$ where $W$ and $X$ are random rectangular matrices with i.i.d. centered entries.
The function $f$ is applied pointwise and can be seen as an activation function in (random) neural networks.
We compute the asymptotic empirical distribution of this ensemble in the case where $W$ and $X$ have sub-Gaussian tails and $f$ is real analytic. This extends a result of \cite{pennington2017nonlinear} where the case of Gaussian matrices $W$ and $X$ is considered.
We also investigate the same questions in the multi-layer case, regarding neural network and machine learning applications.}
\end{abstract}
\section{Introduction}
Machine learning has shined through a large list of succesful applications over the past five years or so (see for instance applications in image or speech recognition \cites{krizhevsky2012imagenet, hinton2012deep} or translation \cite{wu2016google}) but is also used now in video, style transfer, dialogues, games and countless other topics. The interested reader can go to \cite{schmidhuber2015deep} for an overview of the subject. However, a complete theoretical and mathematical understanding of learning is still missing. The main difficulty comes from the complexity of studying highly non-convex functions of a very large number of parameters \cites{choromanska2015loss, pennington2017geometry}.  We also refer to \cite{Zdeborova} for a comprehensive exposition of the problem we are interested in.

An artificial neural network can be modeled as follows: some input column vector $x\in \mathbb{R}^{n_0}$ goes through a multistage architecture of alternated layers with both
linear and non linear functionals: let $g_i: \mathbb{R} \to \mathbb{R}$, $i=1, \ldots, L$ be some given \emph{activation  functions} and $W_i,i=1\ldots L$ be $n_i\times n_{i-1}$ matrices. The output vector after layer $L$ is 
\begin{equation}s_1=g_1(W_1x), \quad s_i=g_i(W_is_{i-1}), i=2, \ldots,L .\label{FFNN}\end{equation}
The functions $g_i$ are here applied componentwise. 
The matrices $W_i$ are the (synaptic) weights in the  layer $i$ and the activation function $g_i$ models the impact of the neurons in the architecture. 
There are different possible choices for the activation functions: some notable examples are $g(x)=\max(0,x)$ (known as the ReLU activation function for Rectified Linear Unit) or the sigmoid function $g(x)=(1+e^{-x})^{-1}$. The parameter $L$ is called the depth of the neural network. This depth is important with respect to the question of machine learning in artificial neural networks: we hereafter introduce the problem of learning in this context.  One may also refer the interested reader to \cite{LeCun} for more information on the development of deep learning, i.e. when $L>1$. 

Generally in supervised machine learning, one is given a $n_0\times m$ matrix dataset $X$ coinjointly with a target dataset $Z$ of size $d\times m$. The parameter $m$ is here the sample size. For instance the $n_0-$dimensional column vectors of $X$ encode the (pixels of) photographs of cats and $Z$ is the $m$ sample of $d$ possible breeds of cats. The aim of supervised learning is to determine a function $h$ so that, given a \emph{new} photo $x$, the output of the function $h(x)$ yields an acceptable approximation of the target (true) object (that is the breed in our example). The parameters to be learned are here the weight matrices.  The error when performing such an approximation is measured through a loss function. 
In the context of Feed Forward Neural Networks as in (\ref{FFNN}), when the input vector is high dimensional and the sample size is comparably large, one of the commonly used learning method is ridge regression (and $h$ is linear in $Y$).   More precisely, in the one layer case ($L=1$) the loss function is 
$$B\in \mathbb{R}^{d\times n_1}\mapsto \mathcal{L}(B):= \frac{1}{2dm}||Z - B^* (g_1(W_1 X))||_F^2 +\gamma || B||_F^2,$$ where $\gamma$ is the \emph{learning rate} or penalizing parameter.  The optimal matrix $B$ can then be proved to be proportional to $Y QZ^*$ where $Y=(g_1(W_1 X))$ and \begin{equation}\label{defQ}Q=\left( \frac{1}{m}Y^*Y +\gamma I\right)^{-1}.\end{equation} As a consequence, the performance of this learning procedure can be measured thanks to the asymptotic spectral properties of the matrix $\frac{1}{m}Y^*Y .$ Indeed, for the one layer case, the expected training loss
can be proved to be related to the asymptotic e.e.d. (and Stieltjes transform). It is given by 
\[
\mathbb{E} (\mathcal{L}(B))= -\frac{\gamma^2}{m}\frac{\partial}{\partial \gamma} \mathbb{E} (\text{Tr } Q),
\]
where $Q$ is given by (\ref{defQ}) and $\text{Tr}$ denotes the unnormalized trace.

\vspace{\baselineskip}
A possible idea to understand better such large complex systems is to approximate the elements of the system by random variables as it is done in statistical physics and thermodynamics. This is the place where random matrix theory can bring its techniques in principle.
Random matrix theory has already been proved to be useful in machine learning. 
In \cite{giryes2016deep} for instance, neural networks with random Gaussian weights have been studied for practical interest while eigenvalues of non-Hermitian matrices were used to understand neural networks in \cite{rajan2006eigenvalue}. See also \cite{zhang2012nonlinear} who study echo state networks used to model nonlinear dynamical systems.
In \cites{benaych2016spectral, couillet2016kernel}, a random matrix approach has been used to do a theoretical study of spectral clustering by looking at the Gram matrix $WW^*$ where the columns of $W$ are given by random vectors. They compute the asymptotic deterministic empirical distribution of this matrix which allows the analysis of the spectral clustering algorithm in large dimensions. Nonlinear  random matrix models have also been studied in \cite{elkaroui2010spectrum} e.g.

We are here interested in random neural networks where both the number of samples $m$ and the number of parameters $n_0$ are large. We consider rectangular matrices of size $n_0\times m$ in the regime where $n_0/m$ goes to some constant $\phi$ as the dimension grows to infinity.
The study of such matrix models for random neural networks was first accomplished in \cites{louart2018random, pennington2017nonlinear}, where they consider
\[
M=\frac{1}{m}Y^*Y\in\mathbb{R}^{n_1\times n_1}\quad\text{with}\quad Y_{ij}=f\left(
	\frac{1}{\sqrt{n_0}}(WX)_{ij}
\right)
\quad\text{for}\quad
1\leqslant i\leqslant n_1,\quad
1\leqslant j\leqslant m.
\]
In the above equation $f$ is a nonlinear activation function, $W$ is the  $n_1\times n_0$ matrix corresponding to the weights and $X$ the $n_0\times m$ matrix of the data. There are several possibilities to incorporate randomness in this model. In \cite{louart2018random}, the authors consider random weights with \emph{deterministic data} $X$. The weights are given by functions of Gaussian random variables and the asymptotic eigenvalue distribution of $M$ is studied thanks to concentration inequalities in the case where the function $f$ is Lipschitz continuous. 
They prove that the eigenvalue distribution corresponds to that of a (usual) sample covariance matrix $\frac{1}{m}T^*X^*XT$ with population covariance $T^*T=\overline{M}$  as studied in \cite{silverstein1995empirical}. Thus the nonlinearity coming from applying the function $f$ entrywise is rather mysteriously hidden in the asymptotic empirical eigenvalue distribution. However, there is a major difference from a usual sample covariance matrix ensemble, which is
the non universality of the eigenvalue distribution (as $\overline{M}$ depends on the distribution of $W$ beyond its first two moments). The authors \cite{louart2018random} use this equation to study the effect of the fourth moments of the distribution for the efficiency of the neural networks. The general approach based on concentration arguments that they develop is detailed in the recent preprint \cite{LouartCouilletBis}.

On the other side, and this is the scope of this article,  \cite{pennington2017nonlinear} consider the case where both the matrices $W$ and $X$ are random as both matrices are chosen to be independent random matrices with normalized Gaussian entries. Thus, interestingly, they derive (using Gaussian integration and a saddle point argument) a fixed point equation for the Stieltjes transform of the asymptotic e.e.d., which is a \emph{quartic equation}. This equation will be recalled in our main Theorem \ref{theo:result2} below.\\
Before discussing our result, one may note that the quartic equation specializes in some special cases of the parameters to the Mar\v{c}enko-Pastur equation for the Stieltjes transform:
\[
zm(z)^2
+
\left(
	\left(
		1-\frac{\psi}{\phi}
	\right)z-1
\right)m(z)
+
\frac{\psi}{\phi}=0.
\]
Thus there exists a class of functions such that the nonlinear matrix model has the same limiting e.e.d. as that of Wishart matrices. It was then conjectured in \cite{pennington2017nonlinear} that choosing such an activation function could speed up training through the network.
The equation also becomes cubic when the function $f$ is linear and corresponds to the product Wishart matrix. The limiting e.e.d. of such matrices, known as the Fuss--Catalan or Raney distribution, has been computed in \cites{Penson2011,dupic2014spectral,Forrester}.\\
  We refer the reader to Sections 4 and 5 of \cite{pennington2017nonlinear} for a more detailed discussion on machine learning applications of such a result. In particular \cite{pennington2017nonlinear} use this equation to facilitate the choice of activation function, a problem which has a crucial impact on the training procedure. In \cite{hayou2019selection}, the choice of function was studied for random neural networks after going through a large number of layers. 
This is of particular interest to consider the multi-layer case, due to potential application to Feed Forward Neural Networks. It is achieved in \cite{pennington2017nonlinear}, where they conjecture that the Mar\v{c}enko-Pastur is invariant through multiple layers for some appropriate activation function $f$.
Interestingly for practical applications, one can then measure some performance of the network (and its depth) using the shape of the Marcenko-Pastur distribution.
For linear models, one may note the multilayer case corresponds to the maybe simpler setting of products of random matrices. One refers the reader to \cites{kuijlaars2014singular, claeys2015correlation, akemann1, akemann2, Penson2011} for products of complex Ginibre matrices and to \cite{hanin2018products} where a large product of large random matrices is considered.

\paragraph{}The scope of this paper is both theoretical and practical: one aims to study the asymptotic e.e.d. of such nonlinear models of random matrices $f(WX)$ where $f$ is applied entrywise and to extend the result established by \cite{pennington2017nonlinear} to non-Gaussian matrices. In particular, the question of universality of the limiting e.e.d. is of interest here as initial weights can be chosen to be non-Gaussian (a typical example is the uniform distribution as in \cite{glorot}). In this setting, it has to be compared to the result of \cite{elkaroui2010spectrum} where some kernel matrices are investigated (with another universal limiting empirical eigenvalue distribution).
 There is also some practical interest as it provides an easy way to compare different possible activation functions for a certain class of distribution for both weights and data. For practical purpose, we also investigate the multilayer case $Y^{(\ell)}=f(W^{(\ell-1)}Y^{(\ell-1)})$ for $\ell=1\dots L$ with $L$ fixed and study again the asymptotic empirical eigenvalue distribution for a class of activation functions. This gives one of the few theoretical results on multilayer nonlinear random neural networks and confirm the prediction made in \cite{pennington2017nonlinear}.
From a theoretical point of view, such a study is also of interest in random matrix theory itself as it introduces a new class of ensembles of random matrices as well as a new class for universality.

\paragraph{Acknowledgments.} The authors would like to thank D. Schr\"oder and Z. Fan for pointing out errors in a previous version of the article as well as anonymous referees for helpful suggestions on how to improve the present paper.

\section{Model and results}
Consider a random matrix $X\in\mathbb{R}^{n_0\times m}$ with $i.i.d.$ elements with distribution $\nu_1$. Let also $W\in\mathbb{R}^{n_1\times n_0}$ be a random matrix with $i.i.d.$ entries with distribution $\nu_2$. $W$ is called the weight matrix. Both distributions are centered and we denote the variance of each distribution by
\begin{equation}\label{eq:variance}
\mathds{E}\left[X_{ij}^2\right]=\sigma_x^2\quad\text{and}\quad\mathds{E}\left[W_{ij}^2\right]=\sigma_w^2.
\end{equation}
We also need the following assumption on the tails of $W$ and $X$: there exist constants $\vartheta_w,\,\vartheta_x>0$ and $\alpha>1$ such that for any $t>0$ we have

\begin{equation}\label{eq:assumwx}
\mathds{P}\left(
	\left\vert W_{11}\right\vert > t
\right)
\leqslant
e^{-\vartheta_w t^\alpha}
\quad\text{and}\quad
\mathds{P}\left(
	\left\vert X_{11}\right\vert > t
\right)
\leqslant
e^{-\vartheta_x t^\alpha}.
\end{equation}
Note that the above implies  that there exists a constant $C>0$ such that
\begin{equation}
\mathds{P}\left(
	\left\vert
		\frac{1}{\sqrt{n_0}}
		\sum_{k=1}^{n_0}
		W_{1k}X_{k1}
	\right\vert
	> t
\right)
\leqslant 
Ce^{-t^2/2}.
\end{equation}
We now consider a smooth function $f:\mathbb{R}\to \mathbb{R}$ with zero Gaussian mean in the sense that 
\begin{equation}\label{eq:assumfgauss}
\int f(\sigma_w\sigma_x x)\frac{e^{-x^2/2}}{\sqrt{2\pi}}\D x=0.
\end{equation}
As an additional assumption, we also suppose that there exist positive constants $C_f$ and $c_f$ and $A_0>0$ such that for any $A\geqslant A_0$ and any $n\in\mathbb{N}$ we have,
\begin{equation}\label{eq:assum2f}
\sup_{x\in[-A,A]}
\vert f^{(n)}(x)\vert
\leqslant
C_fA^{c_fn}.
\end{equation}
\begin{remark}
\eqref{eq:assum2f} guarantees that the function is real analytic which may be seen as a strong restriction. However, commonly used activation functions fall within the scope of this paper such as the sigmoid $f(x)=(1+e^{-x})^{-1}$, $f(x)=\tanh x$ or the softplus $f(x)=\beta^{-1}\log(1+e^{\beta x})$, i.e. a smooth variant of the ReLU. Extensions to more general (non analytic) functions $f$ is the object of current research.
\end{remark}

We consider the following random matrix, 
\begin{equation}\label{eq:defM}
M=\frac{1}{m}YY^*\in\mathbb{R}^{n_1\times n_1}\quad\text{with}\quad Y=f\left(\frac{WX}{\sqrt{n_0}}\right)
\end{equation}
where $f$ is applied entrywise. We suppose that the dimensions of both the columns and the rows of each matrix grow together in the following sense: there exist positive constants $\phi$ and $\psi$ such that 
\[
\frac{n_0}{m}\xrightarrow[m\rightarrow\infty]{}\phi,\quad\frac{n_0}{n_1}\xrightarrow[m\rightarrow\infty]{}\psi
\]
 Denote by $(\lambda_1,\dots,\lambda_{n_1})$ the eigenvalues of $M$ given by \eqref{eq:defM} and define its e.e.d. by
\begin{equation}\label{eq:empirical}
\mu_{n_1}=\frac{1}{n_1}\sum_{i=1}^{n_1}\delta_{\lambda_i}.
\end{equation} 
\begin{theorem}\label{theo:result1}
There exists a deterministic compactly supported measure $\mu$ such that we have
\[
\mu_{n_1}^{(f)}\xrightarrow[n_1\rightarrow\infty]{} \mu \quad\text{weakly almost surely}.
\]
\end{theorem}
Similarly we denote by $(\tilde \lambda_1,\dots,\tilde \lambda_{n_1}, 0\ldots, 0)$ the eigenvalues of $\frac{1}{m}Y^*Y$ (note that $m-n_1$ such eigenvalues are necessarily null). We set $\tilde{\mu}_m$ its e.e.d. and by $\tilde \mu$ its limit.

The moments of the asymptotic empirical eigenvalue distribution depend on the two following parameters of the function $f$: we set
\begin{equation}\label{eq:theta}
\theta_1(f) = \int f^2(\sigma_w\sigma_x x)\frac{e^{-x^2/2}}{\sqrt{2\pi}}\D x
\quad\text{and}\quad
\theta_2(f)=\left(\sigma_w\sigma_x\int f^\prime(\sigma_w\sigma_x x)\frac{e^{-x^2/2}}{\sqrt{2\pi}}\D x\right)^2.
\end{equation}

We also define the following Stieltjes transforms: let $z\in\mathbb{C}\setminus\mathbb{R}$, we set
\[
\begin{gathered}
G(z):=\int \frac{\D \mu(x)}{x-z}
,\quad\text{ }
\tilde G(z):=\int \frac{\D \tilde \mu(x)}{x-z}\quad \text{ and }
H(z):=\frac{\psi-1}{\psi}-\frac{z}{\psi}G(z).
\end{gathered}
\]

\begin{theorem}\label{theo:result2}
The measure $\mu$ satisfies the following fixed point equation for its Stieljes transform $G$: 
\[
\frac{H(z)}{z}
=
\frac{1}{z}+\frac{G(z)\tilde G(z)(\theta_1(f)-\theta_2(f))}{\psi }+\frac{G(z)\tilde G(z)\theta_2(f)}{\psi -zG(z)\tilde G(z)\theta_2(f)},
\]
with  $\theta_1(f)$ and $\theta_2(f)$ are defined in \eqref{eq:theta}.
\end{theorem}

\begin{remark} Assumption (\ref{eq:assumfgauss}) is not really needed for this result to hold true: the asymptotic e.e.d. is unchanged if $Y$ is switched by a rank one matrix but this is the correct centering to avoid a very large eigenvalue.
 \end{remark}
In the case where $\theta_2(f)=0$, $\theta_1(f)=\phi=\psi=1$, the limiting measure $\mu$ is the Marcenko--Pastur distribution (with shape parameter 1).
In the general case, the above fourth-order equation admits two pairs of conjugated solutions. In the companion article \cite{PecheECP}  (see Theorem 1.4), it is shown that  $\mu$ is actually the limiting e.e.d. of an information plus noise sample covariance matrix.


\begin{figure}[!ht]
	\centering
	\begin{subfigure}{.16\linewidth}
	\end{subfigure}
	\begin{subfigure}{.32\linewidth}
		\includegraphics[width=\linewidth]{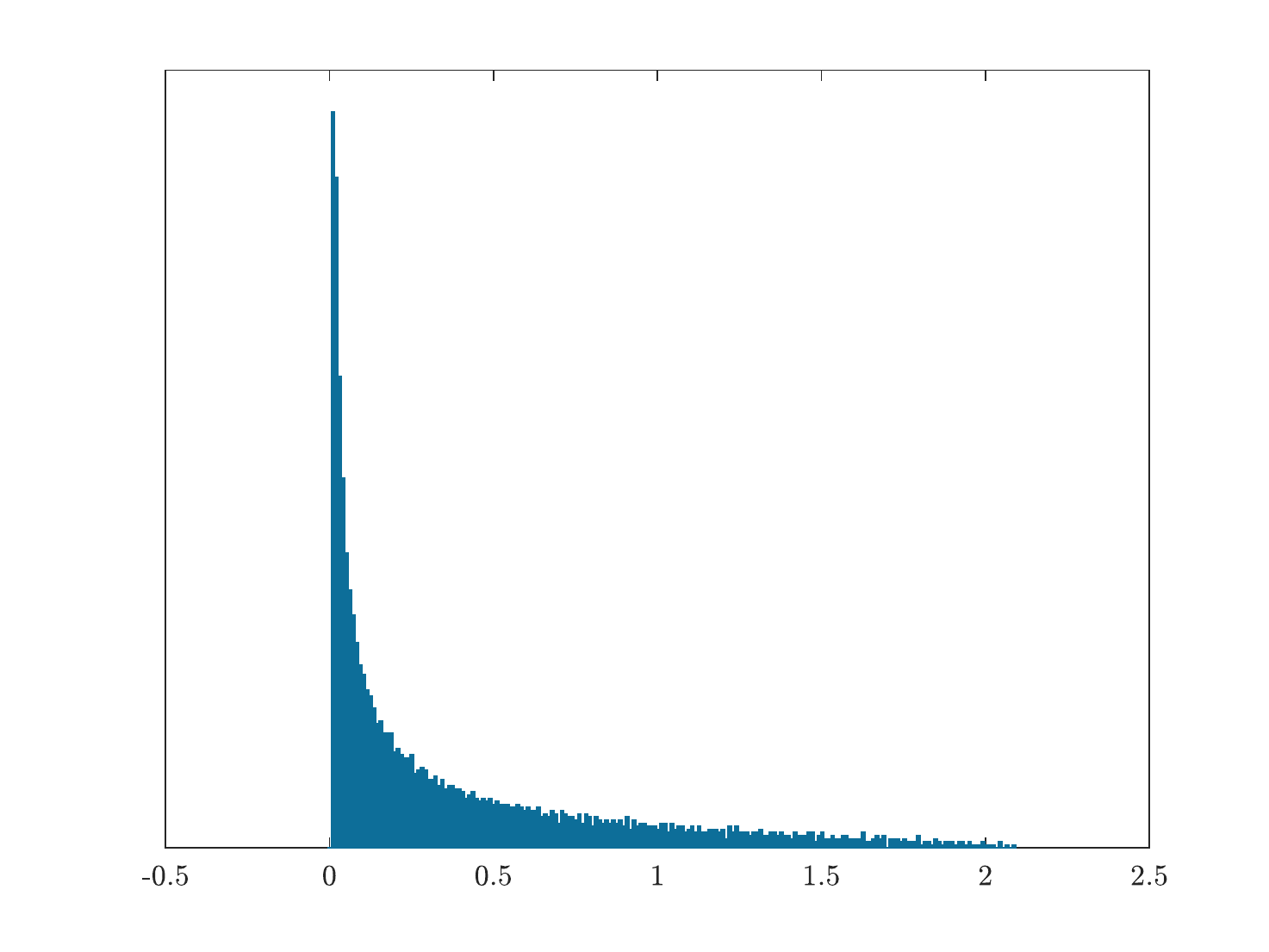}
		\caption{$f(x)\hspace{-.2em}=\hspace{-.2em}\tanh(x)$}
	\end{subfigure}
	\begin{subfigure}{.32\linewidth}
		\includegraphics[width=\linewidth]{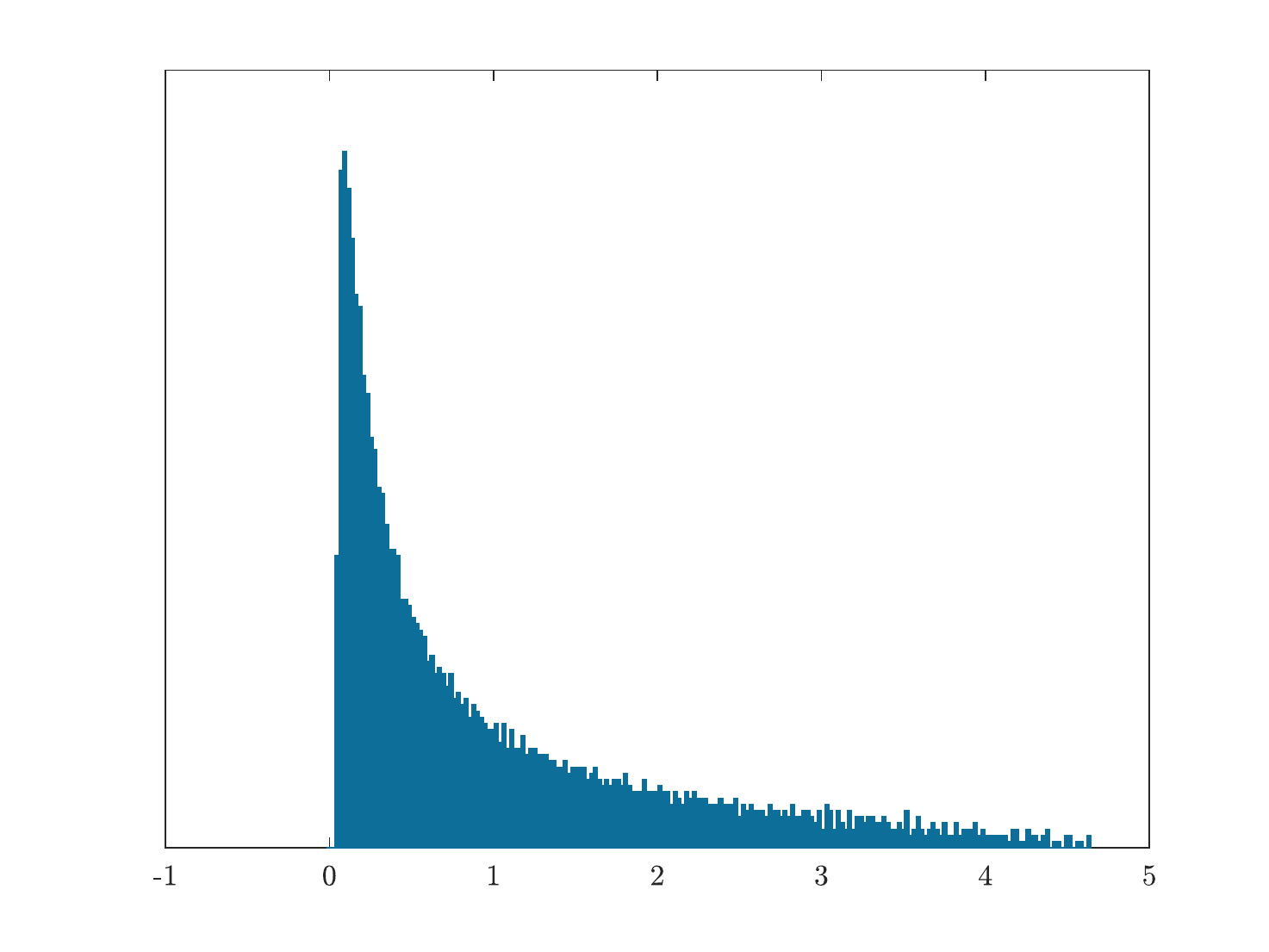}
		\caption{\small{$f(x)\hspace{-.2em}=\hspace{-.2em}\max(x,0)$}}
	\end{subfigure}
	\begin{subfigure}{.16\linewidth}
	\end{subfigure}\\
	\begin{subfigure}{.16\linewidth}
	\end{subfigure}
	\begin{subfigure}{.32\linewidth}
		\includegraphics[width=\linewidth]{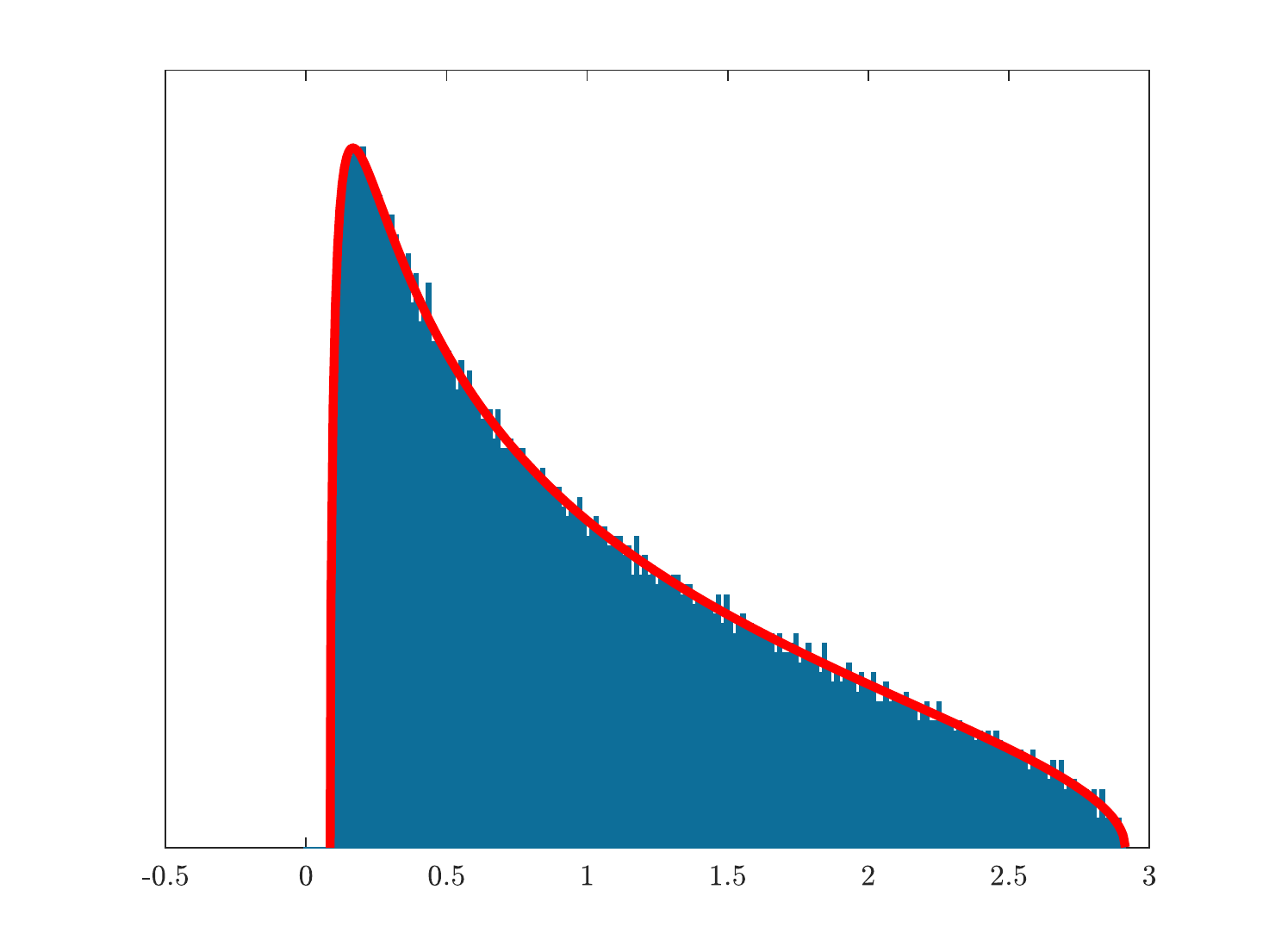}
		\caption{$f(x)\hspace{-.2em}=\hspace{-.2em}\cos(x)$}
	\end{subfigure}
	\begin{subfigure}{.32\linewidth}
		\includegraphics[width=\linewidth]{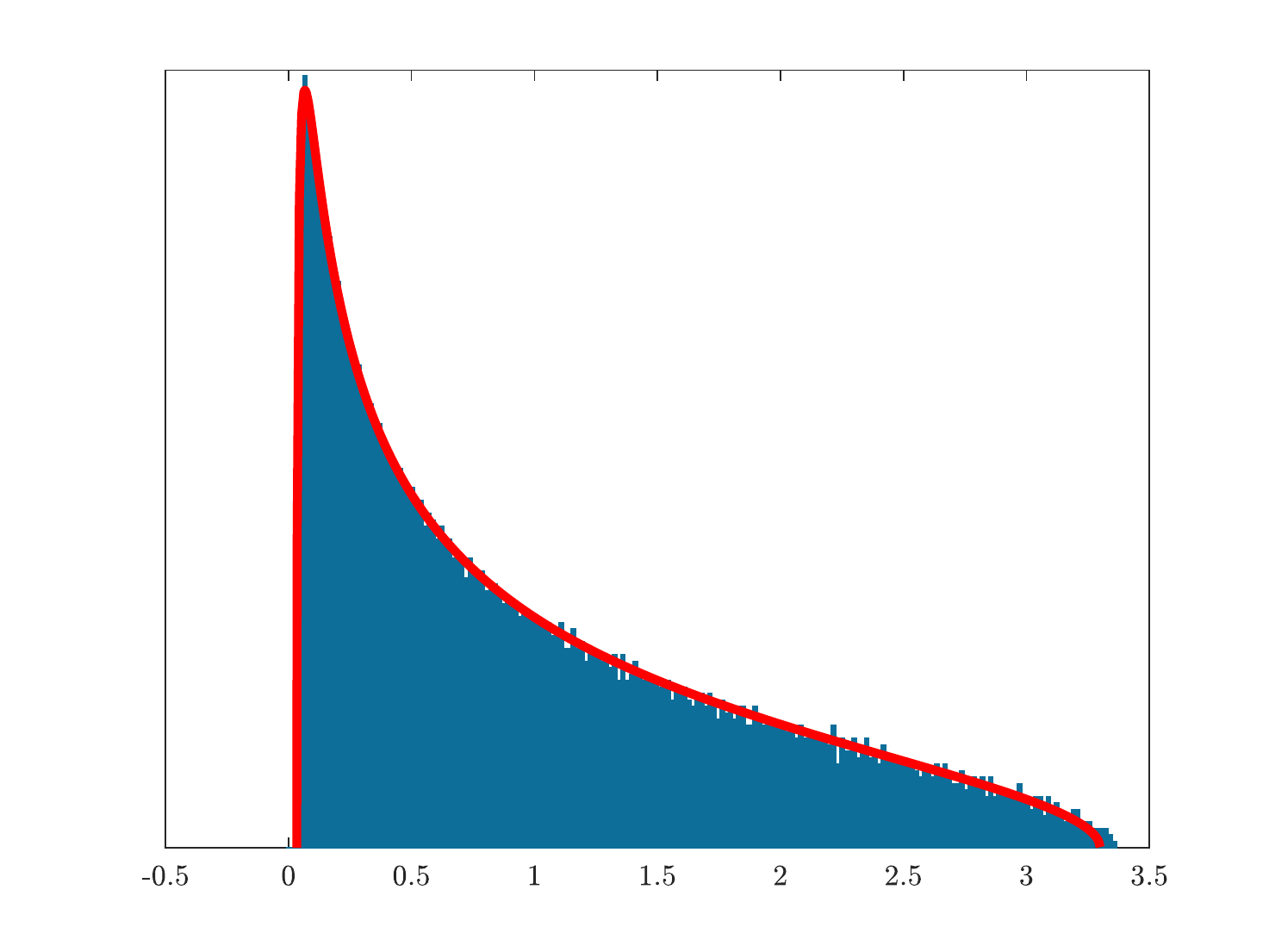}
		\caption{$f(x)\hspace{-.2em}=\hspace{-.2em}x^3-3x$}
	\end{subfigure}
	\begin{subfigure}{.16\linewidth}
	\end{subfigure}
	\caption{\small{Eigenvalues of $M$ for different activation functions. Note that every function displayed here is actually scaled so that $\theta_1(f)=1$ and centered so that there is no very large eigenvalue. For the two bottom figures, we have $\theta_2(f)=0$ and the Marcenko--Pastur of shape parameter $\phi/\psi$ density is plotted in red.}}
\end{figure}



\begin{remark}
Our proof is based on a method of moments as in \cite{Montanari}  to recover the self-consistent equation for the Stieltjes transform. Our analysis is actually strong enough to obtain convergence of the largest eigenvalue to the edge of the support when the function $f$ is odd by considering moments of order larger than $\log n_1$. The behavior of the largest eigenvalue in the general case is the object of a forthcoming article. 
\end{remark}

The model given by \eqref{eq:defM} consists in passing the input data through one layer of a neural network as we apply the function $f$ a single time. However, we could reinsert the output data through the network again, thus multiplying layers. It was conjectured in \cite{pennington2017nonlinear} that for activation functions such that $\theta_2(f)=0$ the limiting e.e.d. is invariant and given by the Mar\v{c}enko--Pastur distribution at each layer. We prove this statement in Theorem \ref{theo:marclayer} below. 
We denote by $L$ the number of layers  and consider, for $p\in[\![0,L-1]\!]$ a family of independent matrices $W^{(p)}\in \mathbb{R}^{n_{p+1}\times n_{p}}$ where $(n_p)_p$ is a family of growing sequences of integers such that there exists $(\phi_p)_p$ and $(\psi_p)_p$ such that 
\[
\frac{n_{0}}{m}\xrightarrow[m\rightarrow\infty]{}\phi\quad \frac{n_{p}}{n_{p+1}}\xrightarrow[m\rightarrow\infty]{}\psi_p.
\]
We suppose that all the matrix entries $(W_{ij}^{(p)})_{ij}, 1\leq i\leq n_{p+1}, 1\leq j \leq n_{p}$, $p=0, \ldots, L-1$ are $i.i.d$ with variance $\sigma_{w}^2$. Consider also $X\in\mathbb{R}^{n_0\times m}$ with $i.i.d$ entries of variance $\sigma_x^2$ and define the sequence of random matrices
\begin{equation}\label{def:yell}
Y^{(p+1)}=f\left(
	\frac{\sigma_x}{\sqrt{\theta_1(f)}}
	\frac{W^{(p)}Y^{(p)}}{\sqrt{n_p}}
\right)\in\mathbb{R}^{n_{p+1}\times m}
\quad\text{with}\quad
Y^{(0)}=X.
\end{equation} 
The scaling is here chosen to normalize the variance of the entries of $Y^{(p)}$ at every layer. This normalization is known (adding centering) as batch normalization and is proved to improve the training speed \cite{ioffe2015batch}. The centering  \eqref{eq:assumfgauss} is only important here. Now, one can define 
\[
M^{(L)}=\frac{1}{m}Y^{(L)}Y^{(L)*}\quad\text{and}\quad \mu_{n_L}^{(L)}=\frac{1}{n_L}\sum_{i=1}^{n_L}\delta_{\lambda_i^{(L)}},
\]
where $(\lambda_k^{(L)})$ are the eigenvalues of $M^{(L)}.$ We then prove the following theorem under the additional assumption that the function $f$ is bounded.

\begin{theorem}\label{theo:marclayer} Let $L$ be a given integer.
	Suppose that $f$ is a bounded analytic function such that \eqref{eq:assumfgauss} and \eqref{eq:assum2f} hold. In the case where $\theta_2(f)=0,$ then the asymptotic e.e.d. $\mu_{n_L }^{(L)}$ is given almost surely by the Mar\v{c}enko-Pastur distribution of shape parameter $\frac{\phi}{\psi_0\psi_1\cdots\psi_{L-1}}$.
\end{theorem}

In particular the above result implies that the range of the spectrum of the matrix $\tilde{ M^{(L)}}=\frac{1}{n_L}{Y^{(L)}}^*Y^{(L)}$ is less and less spread as $L$ grows. This may be of importance when one has to determine the parameter $\gamma$ so as to minimize the loss in the testing phase.
 Unfortunately our result does not encompass the case of a number of layers also growing to infinity.

\begin{remark}
	The model we present here for several layers can be thought as a theoretical toy model since in practicality, weights would be updated along the neural network (using gradient descent for instance) and the independence assumption on the weights would not be true. However, the correlation induced by updating the weights does not seem to be tractable yet on a random matrix theory point of view. 
\end{remark}

\paragraph{}The next section is dedicated to proving Theorem \ref{theo:result1} for polynomial activation functions using the moment method. Our choice is motivated by the fact that \cite{pennington2017nonlinear} introduce a family of graphs to describe the asymptotic e.e.d. of Gaussian non linear random matrices, which we want to understand in greater generality. Thus the main part of the article has some combinatorial aspects and we believe that this point of view can give some insights in the study of these matrix models since it relates analytic objects such as $\theta_1$ or $\theta_2$ to combinatorial ones. Our approach is very similar to that of \cite{Montanari} and the results quite similar.
In Section \ref{sec:polyapprox}, we generalize the result  to other functions by using a polynomial approximation . 
 Finally, in Section \ref{sec:multilayer} we first give a combinatorial description of the multilayer case for polynomials and then prove Theorem \ref{theo:marclayer}. 
\section{Limiting e.e.d. when \texorpdfstring{$\bm{f}$}{} is a polynomial}\label{sec:momentmethod}

The point of this section is to compute the moments of the empirical eigenvalue distribution of the matrix $M$ when the activation is a polynomial. The following statement gives the expected moment of the distribution in this case using a graph enumeration.
Before stating the result, we need the following definition. 
\begin{definition}\label{def:graph} Let $q\geq 1$ be a given integer.
A coincidence graph  is a connected graph built up from the simple (bipartite) cycle  of  vertices labeled $i_1, j_1, i_2, \ldots, i_q, j_q$ (in order) by  identifying some $i$-indices respectively and $j$-indices respectively. Such a graph is admissible if the formed cycles are joined to another by at most a common vertex and each edge belongs to a unique cycle.
\end{definition}

\begin{remark}In the following the edges and vertices of such an admissible graph are colored red.
\end{remark}
\begin{remark}
An admissible graph has $2q$ edges. It can also be seen as a tree of cycles (simply replacing cycles by edges) also called a cactus graph. These graphs appear also in random matrix theory in the so-called theory of traffics when expanding injective traces (see \cite{cebron2016universal} e.g.).
\end{remark}

The basic admissible graph is given by the simple cycle (left figure on Figure \ref{fig:tree}) whose associated tree is a simple edge. The two right figures show a tree and one admissible graph that is associated to the tree: note that the points $i_1$ and $j_1$ where cycles are glued to each other are not determined by the tree, neither the lengths of the cycles.
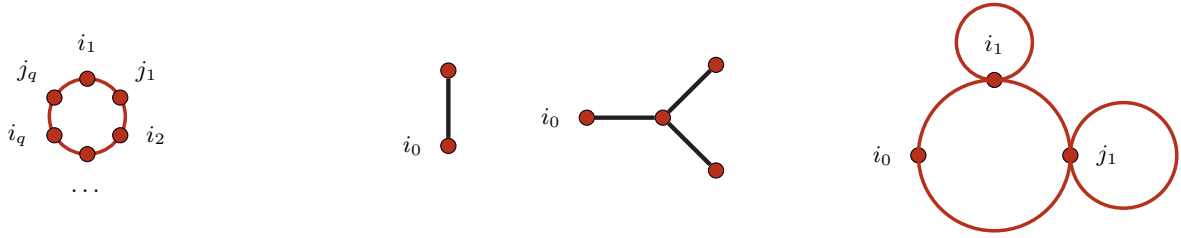
\begin{figure}[ht!]{\centering
\begin{minipage}[c]{.3\linewidth}
\begin{tikzpicture}
	\draw[BrickRed, line width = .12em] (0,0) circle (.5cm) ;
	\node[fill=BrickRed, label=$i_{1}$, inner sep = 0pt, minimum size=.2cm] (1) at (360/6+30: .5cm) {};
	\node[fill=BrickRed, label=135:$j_{q}$, inner sep = 0pt, minimum size=.2cm] (2) at (2*360/6+30: .5cm) {};
	\node[fill=BrickRed, label=left:$i_q$, inner sep = 0pt, minimum size=.2cm] (3) at (3*360/6+30: .5cm)  {};
	\node[fill=BrickRed, label=below:$\dots$, inner sep = 0pt, minimum size=.2cm] (4) at (4*360/6+30: .5cm) {};
	\node[fill=BrickRed, label=right:$i_{2}$, inner sep = 0pt, minimum size=.2cm] (5) at (5*360/6+30: .5cm) {};
	\node[fill=BrickRed, label=45:$j_{1}$, inner sep = 0pt, minimum size=.2cm] (6) at (6*360/6+30: .5cm) {};
\end{tikzpicture}
\end{minipage}
\begin{minipage}[c]{.1\linewidth}
\begin{tikzpicture}
\node[fill=BrickRed,label=left:$i_0$, inner sep = 0pt, minimum size = .2cm] (01) at (0,0) {};
\node[fill=BrickRed, inner sep = 0pt, minimum size = .2cm] (02) at (0,1) {};
\draw[-,Black, line width=.15em] (01) edge (02);
\end{tikzpicture}
\end{minipage}
\begin{minipage}[c]{.25\linewidth}
	\begin{tikzpicture}

			\node[fill=BrickRed,label=left:$i_0$, inner sep = 0pt, minimum size = .2cm] (01) at (0,0) {};
			\node[fill=BrickRed, inner sep = 0pt, minimum size = .2cm] (02) at (1,0) {};
			\node[fill=BrickRed, inner sep = 0pt, minimum size = .2cm] (03) at (1.7,.7) {};
			\node[fill=BrickRed, inner sep = 0pt, minimum size = .2cm] (04) at (1.7,-.7) {};

		\draw[-,Black, line width=.15em] (01) edge (02);
		\draw[-,Black, line width=.15em] (02) edge (03);
		\draw[-,Black, line width=.15em] (02) edge (04);
		\end{tikzpicture}
		\end{minipage}
\begin{minipage}[c]{.3\linewidth}
\begin{tikzpicture}
			\node[fill=BrickRed, label=left:$i_0$, inner sep = 0pt, minimum size = .2cm] (001) at (4,0) {};
			\node[fill=BrickRed, label=right:$j_1$, inner sep = 0pt, minimum size = .2cm] (002) at (6,0) {};
			\node[fill=BrickRed,label=above:$i_1$, inner sep = 0pt, minimum size = .2cm] (003) at (5,1) {};
		\draw[-,BrickRed, line width=.12em] (5,0) circle (1);
		\draw[-,BrickRed, line width=.12em] (6.7,0) circle (0.7);
		\draw[-,BrickRed, line width=.12em]  (5,1.5) circle (0.5);

	\end{tikzpicture}
	
\end{minipage}
}\caption{Two admissible graphs and their associated trees.}\label{fig:tree}\vspace*{-0.5cm}
\end{figure}

\begin{definition} \label{defA}
$\mathcal{A}(q,I_i,I_j,b)$ is the number of admissible graphs with $2q$ edges, $I_i$ $i$-identifications, $I_j$ $j$-identifications and  with exactly $b$ cycles of size 2.
\end{definition}

We can now state the following Theorem. Let $\theta_1$ and $\theta_2$ are defined in \eqref{eq:theta}.
\begin{theorem}\label{theo:pol}
Let $f=\sum_{k=1}^K \frac{a_k}{k!}(x^k-k!!\mathds{1}_{k\text{ even}})$ be a polynomial such that \eqref{eq:assumfgauss} holds. The degree of $f$, $K$, can grow with $n_1$ but we suppose that 
\begin{equation}\label{eq:degreepoly}
K=\mathcal{O}\left(
	\frac{\log n_1}{\log \log n_1}
\right).
\end{equation}
Let $\mu_{n_1}^{(f)}$ be defined in \eqref{eq:empirical} and its expected moments
$\overline{m}_q:=\mathds{E}\left[
	\scp{\mu_{n_1}^{(f)}}{x^q}
\right]
=
\mathds{E}\left[\frac{1}{n_1}\sum_{i=1}^{n_1}\lambda_i^{q}\right].$
We then have the following asymptotics
\begin{equation}\label{eq:resultmoment}
\overline{m}_q=
\sum_{I_i,I_j=0}^q\sum_{b=0}^{I_i+I_j+1}{\mathcal{A}}(q,I_i,I_j,b)\theta_1(f)^b\theta_2(f)^{q-b}\psi^{I_i+1-q}\phi^{I_j}
\left(
1+o(1)
\right).
\end{equation}
\end{theorem}
Note that in this theorem we allow the degree $K$ of the polynomial to grow with $n_1$ as in \eqref{eq:degreepoly} but the theorem holds true for any fixed integer $q$ (independent of $n$).
 It is possible to improve the assumption \eqref{eq:degreepoly} in the sense that $K$ could grow faster with $n_1$. However, this bound is enough for the polynomial approximation we need later (using a Taylor approximation of the function $f$). 
 The proof of the above Theorem relies on combinatorial arguments we now develop.
\subsection{Proof of Theorem \ref{theo:pol} when \texorpdfstring{$\bm{f}$}{} is a monomial of odd degree:}
We first consider the case where $f(x)=\frac{x^k}{k!}$ for an odd integer $k$. 
We first assume that the entries of $W$ and $X$ are bounded in the following sense: there exists a $A>0$ such that 
\[
\max_{ij}\vert W_{ij}\vert+\vert X_{ij}\vert\leqslant A\quad\text{almost surely}.
\]
\subsubsection{Basic definitions}
For this activation function, the entries of $Y=f(WX/\sqrt{n_0})$ are of the form
\begin{equation}\label{eq:entryy}
Y_{ij}=\frac{1}{k!}\left(\frac{WX}{\sqrt{n_0}}\right)^k_{ij}=\frac{1}{n_0^{k/2}k!}\left(
	\sum_{\ell=1}^{n_0}W_{ik}X_{kj}
\right)^k
=
\frac{1}{n_0^{k/2}k!}\sum_{\ell_1,\dots\ell_k=1}^{n_0}\prod_{p=1}^kW_{i\ell_p}X_{\ell_p j}.
\end{equation} 
We want to study the normalized tracial moments of the matrix $M$. Thus we want to consider, for a positive integer  $q$, 
\begin{equation}\label{eq:moment0}
\frac{1}{n_1}\mathds{E}\left[
	\Tr M^q
\right]
=
\frac{1}{n_1m^q}\mathds{E}\left[
	\Tr \left(YY^*\right)^q
\right]
=
\frac{1}{n_1m^q}\mathds{E}\sum_{i_1,\dots,i_q=1}^{n_1}\sum_{j_1,\dots,j_q=1}^mY_{i_1j_1}Y_{i_2j_1}Y_{i_2j_2}Y_{i_3j_2}\dots Y_{i_qj_q}Y_{i_1j_q}.
\end{equation}
We first encode each of the summand in (\ref{eq:moment0}) as a coincidence graph (not necessarily admissible) by simply marking the coinciding indices in the summand.
Then injecting \eqref{eq:entryy} in the previous equation we obtain the following development 
\begin{equation}\label{eq:moment1}
\frac{1}{n_1}\mathds{E}\left[\Tr M^q\right]
=
\frac{1}{n_1m^qn_0^{kq}(k!)^{2q}}
\mathds{E}
\sum_{i_1,\dots,i_q}^{n_1}
\sum_{j_1,\dots,j_q}^{m}
\sum_{\substack{\ell_1^1,\dots\ell^1_k\\\dots\\\ell^{2q}_1\dots\ell^{2q}_k}}^{n_0}
\prod_{p=1}^kW_{i_1\ell^1_p}X_{\ell_p^1j_1}
\prod_{p=1}^kW_{i_2\ell^2_p}X_{\ell_p^2j_1}
\dots
\prod_{p=1}^kW_{i_1\ell^{2q}_p}X_{\ell_p^{2q}j_q}
\end{equation}
To take the $l-$indices into account, we now add to the red graph $2kq$ blue vertices. We can represent the vertices in a graph such as in Figure \ref{fig:idistinct}.  
We call a red edge a \emph{niche}. Each \emph{niche} is decorated by $k$ \emph{blue} vertices from which leave blue edges corresponding to a term $W_{i\ell}X_{\ell j}$ in \eqref{eq:moment1}. Since the $W_{i\ell}$ and $X_{\ell j}$ are centered and independent, each such entry has to arise at least twice in the summand in equation \eqref{eq:moment1}. 
Thus, to compute the spectral moment, one needs to match the blue edges so that each entry arises with multiplicity greater than 2. The matching of $\ell$ indices in \eqref{eq:moment1} corresponds to a matching of the \emph{blue} vertices.
Then, the main contribution shall come from those summands maximizing the number of pairwise distinct indices.

\subsubsection{The simplest admissible graph: a cycle of length \texorpdfstring{$2q$}{}}\label{subsec:oddcycle}
In this subsection, we assume that the $i$ and $j$ indices are pairwise distinct and consider the associated contribution to the spectral moment, which we denote by 
$\mathbf{E_q(k)}$.
We show the following Lemma: 
\begin{lemma}\label{lemma:cycle}One has that 
\[
\mathbf{E_q(k)}=\begin{cases} \theta_2^q(f)\psi^{1-q}+
\mathcal{O}\left(
	\theta_2^q(f)\frac{q+k}{n_0}
\right)& \text{if $q>1$}\cr
\theta_1(f)
+
\mathcal{O}\left(
	\frac{k^2(2k-2)!!}{n_0(k!)^2}
\right) &\text{ if $q=1$.}
\end{cases}\]
\end{lemma}
\paragraph{Proof of Lemma \ref{lemma:cycle}: }
Because the $i$- and $j-$indices are pairwise distinct, the associated red graph is the simple cycle of length $2q$. Thus we can really encode the products in the summand as in the left case of Figure \ref{fig:cy}. Since each matrix entry has to arise twice, say for instance $W_{i_1,\ell^1_1}$, it needs to occur at least an other time in the product. There are then two different ways it can happen:
\begin{itemize}
\item[$(i)$] There exists $p\in\{2,\dotsc,k\}$ such that $\ell^1_p=\ell^1_1$.
\item[$(ii)$] There exists $p\in\{1,\dots,k\}$ such that $\ell_p^{2q}=\ell^1_1.$ Applying the same reasoning for $X_{\ell^1_1, j_1}$, there exists $p^\prime\in\{1,\dots, k\}$ such that $\ell_{p^\prime}^{2}=\ell^1_1$.
\end{itemize}
The same reasoning applies for each niche. Now, in order to maximize the number of pairwise distinct indices, one has to perform the most perfect matchings inside each niche. Note that, as $k$ is odd, case $(ii)$ necessarily occurs. 

\paragraph{}We first consider the contribution of those decorated graphs maximizing the number of pairwise distinct indices.  \\
\textit{The case where $q>1$:}
In this case,  there is a blue cycle of size $2q$ as in Figure \ref{fig:idistinct}. 
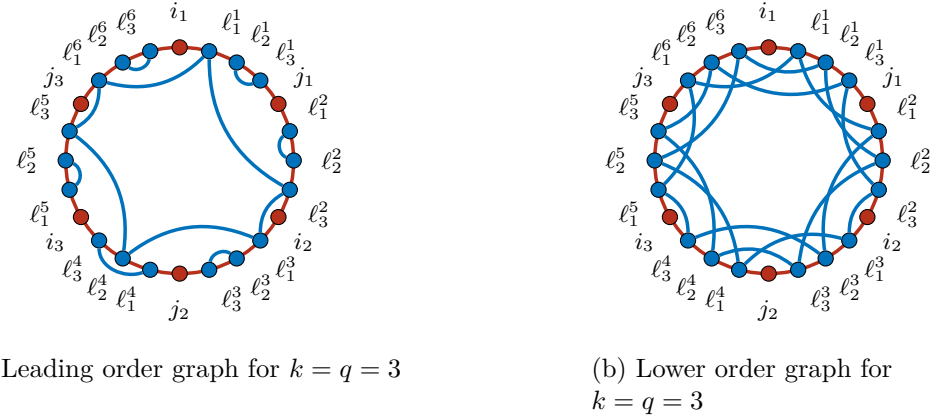
\begin{figure}[ht!]
	\centering
	\begin{subfigure}[t]{0.5\textwidth}
	\centering
	\begin{tikzpicture}
	  \draw[BrickRed, line width = .12em] (0,0) circle (1.5cm) ;
	  \node[fill=BrickRed, label=$i_{1}$, inner sep = 0pt, minimum size=.2cm] (1) at (360/6+30: 1.5cm) {};
	  \node[fill=BrickRed, label=135:$j_{3}$, inner sep = 0pt, minimum size=.2cm] (2) at (2*360/6+30: 1.5cm) {};
	  \node[fill=BrickRed, label=225:$i_3$, inner sep = 0pt, minimum size=.2cm] (3) at (3*360/6+30: 1.5cm)  {};
	  \node[fill=BrickRed, label=below:$j_{2}$, inner sep = 0pt, minimum size=.2cm] (4) at (4*360/6+30: 1.5cm) {};
	  \node[fill=BrickRed, label=315:$i_{2}$, inner sep = 0pt, minimum size=.2cm] (5) at (5*360/6+30: 1.5cm) {};
	  \node[fill=BrickRed, label=45:$j_{1}$, inner sep = 0pt, minimum size=.2cm] (6) at (6*360/6+30: 1.5cm) {};
	\foreach \a in {1,2,3}{
		\pgfmathsetmacro\result{4-\a}
		\node[fill=RoyalBlue, label=(90+\a*11):$\ell_{\pgfmathprintnumber{\result}}^{6}$, inner sep = 0pt, minimum size=.2cm] (6+\a) at (360/6+30+\a*360/24: 1.5cm) {};
		}
	\foreach \a in {1,2,3}{
		\pgfmathsetmacro\result{4-\a}
		\node[fill=RoyalBlue, label=(135+\a*22):$\ell_{\pgfmathprintnumber{\result}}^{5}$, inner sep = 0pt, minimum size=.2cm] (9+\a) at (2*360/6+30+\a*360/24: 1.5cm) {};
		}
	\foreach \a in {1,2,3}{
		\pgfmathsetmacro\result{4-\a}
		\node[fill=RoyalBlue, label=(225+\a*11):$\ell_{\pgfmathprintnumber{\result}}^{4}$, inner sep = 0pt, minimum size=.2cm] (12+\a) at (3*360/6+30+\a*360/24: 1.5cm) {};
		}
	\foreach \a in {1,2,3}{
		\pgfmathsetmacro\result{4-\a}
		\node[fill=RoyalBlue, label=(270+\a*11):$\ell_{\pgfmathprintnumber{\result}}^{3}$, inner sep = 0pt, minimum size=.2cm] (15+\a) at (4*360/6+30+\a*360/24: 1.5cm) {};
		}
	\foreach \a in {1,2,3}{
		\pgfmathsetmacro\result{4-\a}
		\node[fill=RoyalBlue, label=(315+\a*22):$\ell_{\pgfmathprintnumber{\result}}^{2}$, inner sep = 0pt, minimum size=.2cm] (18+\a) at (5*360/6+30+\a*360/24: 1.5cm) {};
		}
	\foreach \a in {1,2,3}{
		\pgfmathsetmacro\result{4-\a}
		\node[fill=RoyalBlue, label=(45+\a*11):$\ell_{\pgfmathprintnumber{\result}}^{1}$, inner sep = 0pt, minimum size=.2cm] (21+\a) at (6*360/6+30+\a*360/24: 1.5cm) {};
		}
	\draw[-, RoyalBlue] (6+1) edge[bend left=60, line width=.12em] (6+2);
	\draw[-,RoyalBlue] (9+2) edge[bend left=60, line width=.12em] (9+3);
	\draw[-,RoyalBlue] (12+1) edge[out=270, in=200, line width=.12em] (12+3);
	\draw[-,RoyalBlue] (15+1) edge[bend left=60, line width=.12em] (15+2);
	\draw[-,RoyalBlue] (18+2) edge[bend left=60, line width=.12em] (18+3);
	\draw[-,RoyalBlue] (21+1) edge[bend left=60, line width=.12em] (21+2);
	\draw[-,RoyalBlue] (6+3) edge[bend left, line width=.12em] (9+1);
	\draw[-,RoyalBlue] (9+1) edge[bend left, line width=.12em] (12+2);
	\draw[-,RoyalBlue] (12+2) edge[bend left, line width=.12em] (15+3);
	\draw[-,RoyalBlue] (15+3) edge[bend left, line width=.12em] (18+1);
	\draw[-,RoyalBlue] (18+1) edge[bend left, line width=.12em] (21+3);
	\draw[-,RoyalBlue] (21+3) edge[bend left, line width=.12em] (6+3);
	\end{tikzpicture}	
	\caption{Leading order graph for $k=q=3$}
	\label{fig:idistinct}
	\end{subfigure}
	\begin{subfigure}[t]{0.4\textwidth}
	\centering
	\begin{tikzpicture}
	  \draw[BrickRed, line width = .12em] (0,0) circle (1.5cm) ;
	  \node[fill=BrickRed, label=$i_{1}$, inner sep = 0pt, minimum size=.2cm] (1) at (360/6+30: 1.5cm) {};
	  \node[fill=BrickRed, label=135:$j_{3}$, inner sep = 0pt, minimum size=.2cm] (2) at (2*360/6+30: 1.5cm) {};
	  \node[fill=BrickRed, label=225:$i_3$, inner sep = 0pt, minimum size=.2cm] (3) at (3*360/6+30: 1.5cm)  {};
	  \node[fill=BrickRed, label=below:$j_{2}$, inner sep = 0pt, minimum size=.2cm] (4) at (4*360/6+30: 1.5cm) {};
	  \node[fill=BrickRed, label=315:$i_{2}$, inner sep = 0pt, minimum size=.2cm] (5) at (5*360/6+30: 1.5cm) {};
	  \node[fill=BrickRed, label=45:$j_{1}$, inner sep = 0pt, minimum size=.2cm] (6) at (6*360/6+30: 1.5cm) {};
	\foreach \a in {1,2,3}{
		\pgfmathsetmacro\result{4-\a}
		\node[fill=RoyalBlue, label=(90+\a*11):$\ell_{\pgfmathprintnumber{\result}}^{6}$, inner sep = 0pt, minimum size=.2cm] (6+\a) at (360/6+30+\a*360/24: 1.5cm) {};
		}
	\foreach \a in {1,2,3}{
		\pgfmathsetmacro\result{4-\a}
		\node[fill=RoyalBlue, label=(135+\a*22):$\ell_{\pgfmathprintnumber{\result}}^{5}$, inner sep = 0pt, minimum size=.2cm] (9+\a) at (2*360/6+30+\a*360/24: 1.5cm) {};
		}
	\foreach \a in {1,2,3}{
		\pgfmathsetmacro\result{4-\a}
		\node[fill=RoyalBlue, label=(225+\a*11):$\ell_{\pgfmathprintnumber{\result}}^{4}$, inner sep = 0pt, minimum size=.2cm] (12+\a) at (3*360/6+30+\a*360/24: 1.5cm) {};
		}
	\foreach \a in {1,2,3}{
		\pgfmathsetmacro\result{4-\a}
		\node[fill=RoyalBlue, label=(270+\a*11):$\ell_{\pgfmathprintnumber{\result}}^{3}$, inner sep = 0pt, minimum size=.2cm] (15+\a) at (4*360/6+30+\a*360/24: 1.5cm) {};
		}
	\foreach \a in {1,2,3}{
		\pgfmathsetmacro\result{4-\a}
		\node[fill=RoyalBlue, label=(315+\a*22):$\ell_{\pgfmathprintnumber{\result}}^{2}$, inner sep = 0pt, minimum size=.2cm] (18+\a) at (5*360/6+30+\a*360/24: 1.5cm) {};
		}
	\foreach \a in {1,2,3}{
		\pgfmathsetmacro\result{4-\a}
		\node[fill=RoyalBlue, label=(45+\a*11):$\ell_{\pgfmathprintnumber{\result}}^{1}$, inner sep = 0pt, minimum size=.2cm] (21+\a) at (6*360/6+30+\a*360/24: 1.5cm) {};
		}
	\draw[-,RoyalBlue] (6+1) edge[bend left, line width=.12em] (9+2);
	\draw[-,RoyalBlue] (6+2) edge[bend left, line width=.12em] (9+1);
	\draw[-,RoyalBlue] (6+3) edge[bend left, line width=.12em] (9+3);
	
	\draw[-,RoyalBlue] (9+1) edge[bend left, line width=.12em] (12+2);
	\draw[-,RoyalBlue] (9+2) edge[bend left, line width=.12em] (12+3);
	\draw[-,RoyalBlue] (9+3) edge[bend left, line width=.12em] (12+1);
	
	\draw[-,RoyalBlue] (12+1) edge[bend left, line width=.12em] (15+2);
	\draw[-,RoyalBlue] (12+2) edge[bend left, line width=.12em] (15+1);
	\draw[-,RoyalBlue] (12+3) edge[bend left, line width=.12em] (15+3);

	\draw[-,RoyalBlue] (15+1) edge[bend left, line width=.12em] (18+2);
	\draw[-,RoyalBlue] (15+2) edge[bend left, line width=.12em] (18+3);
	\draw[-,RoyalBlue] (15+3) edge[bend left, line width=.12em] (18+1);

	\draw[-,RoyalBlue] (18+1) edge[bend left, line width=.12em] (21+1);
	\draw[-,RoyalBlue] (18+2) edge[bend left, line width=.12em] (21+2);
	\draw[-,RoyalBlue] (18+3) edge[bend left, line width=.12em] (21+3);

	\draw[-,RoyalBlue] (21+1) edge[bend left, line width=.12em] (6+2);
	\draw[-,RoyalBlue] (21+2) edge[bend left, line width=.12em] (6+1);	
	\draw[-,RoyalBlue] (21+3) edge[bend left, line width=.12em] (6+3);
	\end{tikzpicture}
	\caption{\raggedright Lower order graph for $k=q=3$}
	\label{fig:idistinctnonad}
	\end{subfigure}
\caption{The contribution of the simple cycle}\label{fig:cy} \vspace*{-0.4cm}
\end{figure}
Thus we can construct the decorated  graphs maximizing the number of pairwise distinct indices  in the following way : One chooses an index $\ell_p$ in each niche which is in the only blue cycle of the graph and then the remaining blue edges are perfectly matched inside niches. 
The corresponding contribution from the basic cycle to the moment is, as every entry exactly occurs twice in the products, using \eqref{eq:variance}, 
\[
{E}_q(k)=\frac{((\sigma_w\sigma_x)^{k}k(k-1)!!)^{2q}n_0}{n_1m^qn_0^{kq}(k!)^{2q}}\frac{m!}{(m-q)!}\frac{n_1!}{(n_1-q)!}\frac{n_0!}{(n_0-(k-1)q)!}
\]  
To obtain this formula, note that we choose the $i$-labels over $n_1$ possible indices and the $j$-labels over $m$ indices. Now, we also choose the $\ell$-labels over $n_0$: the one for the blue cycle and those vertices corresponding to matched edges. Finally, we have to fix the \emph{blue} vertices belonging to the blue cycle: there are $k^{2q}$ possible choices. The number of perfect matchings on the rest of the vertices in each niche is then equal to $((k-1)!!)^{2q}$. We then obtain that
\begin{equation}\label{eq:leading}
E_q(k)=\left(
	\frac{(\sigma_w\sigma_x)^kk(k-1)!!}{k!}
\right)^{2q}\psi^{1-q}
+\mathcal{O}\left(
	\left(
		\frac{(\sigma_w\sigma_x)^kk(k-1)!!}{k!}
	\right)^{2q}\frac{q+k}{n_0}
\right).
\end{equation}
Note that, by \eqref{eq:theta}, one has that $\theta_2(f)=\left(
	\dfrac{(\sigma_w\sigma_x)^kk(k-1)!!}{k!}
\right)^{2}$ and 
 we can write 
\[
{E}_q(k)=\theta_2^q(f)\psi^{1-q}+
\mathcal{O}\left(
	\theta_2^q(f)\frac{q+k}{n_0}
\right)=\mathbf{{E}_q(k)}(1+o(1)).
\]

\subparagraph{Case where $q=1$.} The behavior in the case where $k=1$ is slightly different. 
Indeed in this case, we can do any perfect matching between the $2k$ \emph{blue} vertices since there is no difference between any factor $W$ or $X$ in the summand in $\eqref{eq:moment1}$. The graph can bee seen in Figure \ref{fig:q1}. Thus, the contribution of the moments in this case is the following
\[
{E}_1(k)=\frac{(\sigma_w\sigma_x)^{2k}(2k)!!}{(k!)^2}
+
\mathcal{O}\left(
	\frac{k^2(2k-2)!!}{n_0(k!)^2}
\right)
=\theta_1(f)
+
\mathcal{O}\left(
	\frac{k^2(2k-2)!!}{n_0(k!)^2}
\right)=\mathbf{{E}_1(k)}(1+o(1)),
\]
where the error comes from performing a matching which is not a perfect one.
	\begin{figure}[!ht]
		\begin{subfigure}{.49\linewidth}
		\centering
		\begin{tikzpicture}
			\node[fill=BrickRed, label=$i_{1}$, inner sep = 0pt, minimum size=.2cm] (1) at (0,0) {};
			\node[fill=BrickRed, label=$j_{1}$, inner sep = 0pt, minimum size=.2cm] (2) at (2,0) {};
			\draw[-, BrickRed, line width=.12em] (1) edge[bend left]  node[draw=black,line width=.05em,pos=.20, fill=RoyalBlue, inner sep = 0pt, minimum size = .2cm] (3) {} node[draw=black,line width=.05em,midway, fill=RoyalBlue, inner sep = 0pt, minimum size = .2cm] (4) {} node[draw=black,line width=.05em,pos=.80, fill=RoyalBlue, inner sep = 0pt, minimum size = .2cm] (5) {} (2) ;
			\draw[-, BrickRed, line width=.12em] (1) edge[bend right]  node[draw=black,line width=.05em,pos=.20, fill=RoyalBlue, inner sep = 0pt, minimum size = .2cm] (6) {} node[draw=black,line width=.05em,midway, fill=RoyalBlue, inner sep = 0pt, minimum size = .2cm] (7) {} node[draw=black,line width=.05em,pos=.80, fill=RoyalBlue, inner sep = 0pt, minimum size = .2cm] (8) {} (2) ;
			\draw[-, RoyalBlue, line width=.12em] (3) edge  (7);
			\draw[-, RoyalBlue, line width=.12em] (4) edge  (6);
			\draw[-, RoyalBlue, line width=.12em] (5) edge  (8);
		\end{tikzpicture}
		\caption{Contribution in the case where $q=1$}
		\label{fig:q1}
		\end{subfigure}
		\begin{subfigure}{.49\linewidth}
			\centering
			\begin{tikzpicture}
			\node[fill=BrickRed, label=$i_{1}$, inner sep = 0pt, minimum size=.2cm] (1) at (0,0) {};
			\node[fill=BrickRed, label=$j_{1}$, inner sep = 0pt, minimum size=.2cm] (2) at (2,0) {};
			\draw[-, BrickRed, line width=.12em] (1) edge[bend left]  node[draw=black,line width=.05em,pos=.20, fill=RoyalBlue, inner sep = 0pt, minimum size = .2cm] (3) {} node[draw=black,line width=.05em,midway, fill=RoyalBlue, inner sep = 0pt, minimum size = .2cm] (4) {} node[draw=black,line width=.05em,pos=.80, fill=RoyalBlue, inner sep = 0pt, minimum size = .2cm] (5) {} (2) ;
			\draw[-, BrickRed, line width=.12em] (1) edge[bend right]  node[draw=black,line width=.05em,pos=.20, fill=RoyalBlue, inner sep = 0pt, minimum size = .2cm] (6) {} node[draw=black,line width=.05em,midway, fill=RoyalBlue, inner sep = 0pt, minimum size = .2cm] (7) {} node[draw=black,line width=.05em,pos=.80, fill=RoyalBlue, inner sep = 0pt, minimum size = .2cm] (8) {} (2) ;
			\draw[-, RoyalBlue, line width=.12em] (3) edge  (7);
			\draw[-, RoyalBlue, line width=.12em] (4) edge  (6);
			\draw[-, RoyalBlue, line width=.12em] (5) edge[bend right=60]  (3);
			\draw[-, RoyalBlue, line width=.12em] (8) edge  (4);
		\end{tikzpicture}
		\caption{Subleading term in the case $q=1$.}
		\end{subfigure}	
		\vspace*{-0.4cm}
	\end{figure}
\paragraph{}We now consider the contribution of other matchings, that is those not maximizing the number of pairwise distinct indices. 
We will show that ${E_q}$ is indeed the typical contribution from the basic cycle, that is all other matchings lead to a negligible contribution with respect to $E_q$. There are four different phenomena that can give a (lower order) contribution. First, there may be more than one cycle linking every niche as in Figure \ref{fig:idistinctnonad}. Also, in at least one niche there could be more identifications between $\ell$-indices, which raises  moments of entries of $W$ and $X$. There could be an identification between the index of the cycle and an index from a perfect matching inside a niche. Finally, there could also exist identifications between two distinct niches; note we can only get higher moments in the case where the two niches are adjacent. While these four behaviors can happen simultaneously, we see the contribution separately since it would induce an even smaller order if counted together.
\subparagraph{a) There is more than one cycle between niches.} We call $E^{(1)}_q$ the contribution to the moments of such decorated graphs. Suppose there are $c$ cycles. Note that necessarily $c$ is odd since $k$ is odd and entries are centered, then we can write, if we suppose that indices $\ell$ not in cycles are being perfectly matched,
\begin{multline*}
{E}_q^{(1)}
=
\frac{(k^c(k-c)!!)^{2q}}{n_1m^qn_o^{kq}(k!)^{2q}}
\sum_{\substack{i_1,\dots,i_q\\\text{pairwise}\\\text{distinct}}}^{n_1}
\sum_{\substack{j_1,\dots,j_q\\\text{pairwise}\\\text{distinct}}}^{m}
\sum_{\ell_0,\dots,\ell_c}^{n_0}
\sum_{\substack{\ell_1^1,\dots\ell^1_{\frac{k-c}{2}}\\\dots\\\ell^{2q}_1\dots\ell^{2q}_{\frac{k-c}{2}}}}^{n_0}(\sigma_w\sigma_x)^{2kq}\\
=
\frac{((\sigma_w\sigma_x)^k k^c(k-c)!!)^{2q}}{n_1m^qn_0^{kq-c}(k!)^{2q}}
\frac{m!}{(m-q)!}
\frac{n_1!}{(n_1-q)!}
\frac{n_0!}{(n_0-(k-c)q)!}.
\end{multline*}
In order to understand the very first term, note that one has to select in each niche $c$ \emph{blue} vertices to create the cycles and then do a perfect matching for the rest of the vertices.  Thus one has that
\begin{equation}\label{eq:cycles}
{E}_q^{(1)}=\frac{((\sigma_w\sigma_x)^kk(k-c)!!)^{2q}\psi^{1-q}}{n_0^{(c-1)(q-1)}(k!)^{2q}}
\left(1+o\left(1\right)\right).
\end{equation}
Thus this is of smaller order than \eqref{eq:leading} when the number of cycles is strictly greater than 1 as in Figure \ref{fig:idistinctnonad} for instance. Indeed, one obtains that
\[
\frac{E_q^{(1)}}{E_q}
=
\mathcal{O}\left(
	\frac{1}{n_0^{(c-1)(q-1)}}
	\left(
		\frac{(k-c)!!}{(k-1)!!}
	\right)^{2q}
\right).
\]

\subparagraph{b) The matching in each niche is not a perfect matching- apart from the vertex in the cycle.} If the matching is more complicated than a perfect matching, the associated moments could be of higher order than the variance. Consider a matching inside a niche which is not a perfect one: there exists then  an identification between $a_1,\dotsc,a_b$ entries such that $a_1+\dotsb+a_b=k-1$ and such that at least one of the $a_i$'s is greater than 2. For ease we suppose that $a_1=\dots=a_{b_1}=2$ and $a_{b_1+1},\dots,a_b>2$ for some $b_1\in[\![1,b-1]\!]$. See e.g. Figure \ref{fig:fourth}. 
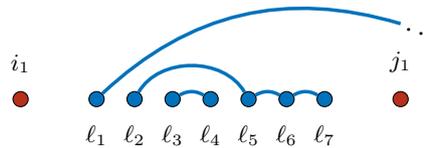
\begin{figure}[ht]
	\centering
	\begin{tikzpicture}
		\node[fill=BrickRed, label=$i_{1}$, inner sep = 0pt, minimum size=.2cm] (0) at (0,0) {};
		\foreach \a in {1,2,...,7}{
			\node[fill=RoyalBlue, label=below:$\ell_{\a}$, inner sep = 0pt, minimum size=.2cm] (\a+1) at (1+\a*0.5-0.5,0) {};
		}
		\node[fill=BrickRed, label=$j_1$, inner sep = 0pt, minimum size =.2cm] (1) at (5,0) {};		
		\draw[-, RoyalBlue] (1+1) edge[bend left, line width = .12em] (5,1);
		\draw[-, RoyalBlue] (2+1) edge[out=60, in=120, line width = .12em] (5+1);
		\draw[-, RoyalBlue] (3+1) edge[bend left, line width = .12em] (4+1);
		\draw[-, RoyalBlue] (5+1) edge[bend left, line width = .12em] (6+1);
		\draw[-, RoyalBlue] (6+1) edge[bend left, line width = .12em] (7+1);
		\node[draw=none, rotate=-20] at (5.3,.857) {\dots};
	\end{tikzpicture}
	\caption{Niche where the induced graph is not a perfect matching which raises a fourth moment in the case where $k=7$.}
	\label{fig:fourth}
\end{figure}

We call ${E}_q^{(2)}$ the contribution of all such matchings where a single niche breaks the perfect matching condition. Then we obtain that:

\begin{multline*}
\frac{E_q^{(2)}}{E_q}
=
\sum_{b=1}^{\frac{k-1}{2}-1}
\sum_{b_1=1}^{b-1}
\sum_{\substack{a_{b_1+1}\dots a_b>2\\\sum a_j=k-1-2b_1}}
\frac{(k-1)!}{(k-1)!!\prod_{i=1}^b a_i!}
\frac{n_0!}{(n_0-(1+b+\frac{k-1}{2}(2q-1)))!}
\times\\\times
\frac{(n_0-(1+(k-1)q)!}{n_0!}
\frac{\prod_{b_1+1}^b\mathds{E}\vert W_{11}\vert^{a_p}\mathds{E}\vert X_{11}\vert^{a_p}}{(\sigma_w\sigma_x)^{2(\frac{k-1}{2}-b_1)}}.
\end{multline*}
The first term in the summand corresponds to assigning the $k-1$ remaining \emph{blue} vertices (after the choice of the cycle) into $b$ classes of size $a_1\dots a_b$. We can bound it in the following way 
\[
\frac{(k-1)!}{(k-1)!!\prod_{i=1}^b a_i}
\leqslant C\frac{2^{\frac{k-1}{2}-b_1}(\frac{k-1}{2})!}{\prod_{i=b_1+1}^{b}a_i}
\leqslant C\left(\frac{k-1}{2}\right)^{\frac{k-1}{2}-b}\frac{2^{\frac{k-1}{2}-b_1}}{\prod_{i=b_1+1}^b a_i}\leqslant C\left({k-1}\right)^{(\frac{k-1}{2}-b)}.
\]
In the first inequality we use the fact that $a_1=\dots=a_{b_1}=2$ and the definition of the double factorial. Then we expand the factorial and in the last inequality we use the fact that $a_i\geqslant 3$ for $i>b_1$.
Now, for the second term, we compare the number of possible choices for $\ell$ indices, yielding that
\[
\frac{(n_0-(1+(k-1)q)!}{(n_0-(1+b+\frac{k-1}{2}(2q-1)))!}
\leqslant
\frac{1}{n_0^{\frac{k-1}{2}-b}}e^{-\frac{C(kq)^2}{N}}.
\]
Finally, the last term in the summand corresponds to the different possible moments, as only variances intervene in the leading contribution, while higher moments can appear inside the niche $\{i_1,j_1\}$. We use the fact that
\begin{equation}\label{eq:highmom1}
\frac{\prod_{b_1+1}^b\mathds{E}\vert W_{11}\vert^{a_p}\mathds{E}\vert X_{11}\vert^{a_p}}{(\sigma_w\sigma_x)^{2(\frac{k-1}{2}-b_1)}}
\leqslant
\frac{A^{2\sum_{i\geqslant b_1+1}a_i}}{\prod_{i\geqslant b_1+1}\sigma_w^{a_i}\sigma_x^{a_i}}
=
\left(
	\frac{A^4}{\sigma_w^2\sigma_x^2}
\right)^{\frac{k-1}{2}-b_1}.
\end{equation}
Now we need to bound the combinatorial factor coming from the sums: 
\[
\sum_{b_1=1}^{b-1}
\sum_{\substack{a_{b_1+1},\dots,a_{b}\geqslant 3\\\sum a_j=k-1-2b_1}}
\leqslant \sum_{b_1=1}^{b-1}\binom{k-1-3b-b_1+b-b_1-1}{b-b_1-1}
\leqslant \sum_{b_1=1}^{b-1}(k-1)^{k-1-3b+b_1}\leqslant (k-1)^{2(\frac{k-1}{2}-b)},
\]
where we use in the first inequality that $\sum_{j}(a_j-3)=k-1-2b_1-3(b-b_1).$ Finally, putting all these contributions together, we obtain the following comparison between $E^{(2)}_q$ and $E_q$,
\begin{equation}\label{eq:contribmatching}
\frac{E_q^{(2)}}{E_q}\leqslant 
\sum_{b=1}^{\frac{k-1}{2}-1}
\left(
	\frac{CA^4}{\sigma_w^2\sigma_x^2}
	\frac{(k-1)^3}{n_0}
\right)^{3(\frac{k-1}{2}-b)}
=
\mathcal{O}\left(
	\frac{Ck^3}{n_0}
\right).
\end{equation}

Note that $k^3=o(n_0)$. Here we suppose that in all other niches a perfect matching and a single cycle is used to match the blue vertices. The other cases are just negligible.
\subparagraph{c) There are identifications between matchings from different niches} If these niches are not adjacent, then such matchings would not increase the moments of the entries of $W$ or $X$. On the contrary, matchings between adjacent niches may result into moments of higher order than the variance. We can then perform the same analysis as the previous one where we replace $k-1$ (the remaining indices after the choice of the cycle in one niche) to $2k-2$ corresponding to the number of vertices of two adjacent niches. This yields a contribution in the order of \eqref{eq:contribmatching} with respect to $E_q$.
\subparagraph{d) There are identifications between the cycle and perfect matchings inside niches.} Suppose that these identifications happen in $d$ niches, and for $p\in\{1,\dots,d\}$, we identify the index from the cycle with $2b_p$ \emph{blue} vertices from the niche. Indeed if the number of identifications was odd, in order to obtain a non-vanishing term, we would need to either create another cycle or perform more identifications inside the niches. Thus, we obtain the following upper bound
\begin{multline*}
\frac{E_q^{(3)}}{E_q}
=
\sum_{d=1}^{2q}
\sum_{b_1,\dots,b_d=1}^{\frac{k-1}{2}}
\binom{2q}{d}
\left[
	\prod_{p=1}^d\binom{k-1}{b_p}
\right]
\prod_{i=1}^d\mathds{E}\vert W_{11}\vert^{2+2b_p}\mathds{E}\vert X_{11}\vert^{2+2b_p}\times \\
\times\frac{((k-1)!!)^{2q-d}\prod_{p=1}^d((k-2b_p-1)!!)}{n_0^{\sum_{p=1}^db_p}((k-1)!!)^{2q}(\sigma_w\sigma_x)^{2d+\sum_{p=1}^d 2b_p}}.
\end{multline*} 
This comes from the choices of the niches, the identifications we make in each niche, and the perfect matchings we perform in the other niches. Finally, we suppose that we perform perfect matchings in the rest of the $d$ niches. Then, we can use the bounds
\begin{equation}\label{eq:highmoment2}
\prod_{p=1}^d \frac{1}{b_p!}\leqslant 1,
\quad
\prod_{i=1}^d\mathds{E}\vert W_{11}\vert^{2+2b_p}\mathds{E}\vert X_{11}\vert^{2+2b_p}
\leqslant A^{4d+4\sum b_p}
\quad\text{and}\quad
\frac{(k-1)!!^{2q-d}\prod_{p=1}^d(k-1-2b_p)!!}{(k-1)!!^{2q}}\leqslant 1.
\end{equation}
From the above we obtain that
\[
\frac{E_q^{(3)}}{E_q}
\leqslant 
\sum_{d=1}^{2q}
\binom{2q}{d}
\left(
	\frac{A^{4}}{\sigma_w^2\sigma_x^2}
\right)^{d}
\sum_{b_1,\dots,b_p=1}^{\frac{k-1}{2}}
\left(
	\frac{A^4(k-1)}{2\sigma_w^2\sigma_x^2 n_0}
\right)^{\sum b_p}
=
\sum_{d=1}^{2q}
\binom{2q}{d}
\left(
	\frac{A^4}{\sigma_w\sigma_x}
	\sum_{b=1}^{\frac{k-1}{2}}
	\left(
		\frac{A^4 (k-1)}{2\sigma_w^2\sigma_x^2n_0}
	\right)^b
\right)^d.
\]
Now since $k\ll n_0$, we obtain that
$\dfrac{E_q^{(3)}}{E_q}
=
\mathcal{O}\left(
	\dfrac{Ck}{n_0}
\right).$
This finishes the proof of Lemma \ref{lemma:cycle}. \hfill $\square$
\subsubsection{Contribution of general admissible graphs}
\begin{lemma}\label{lemma:GA}
The total contribution from admissible graphs to the spectral moment is 
\begin{equation}\label{GA}
E'_q(k)=\sum_{I_i,I_j=0}^q\sum_{b=0}^{I_i+I_j+1}{\mathcal{A}}(q,I_i,I_j,b)\theta_1(f)^b\theta_2(f)^{q-b}\psi^{I_i+1-q}\phi^{I_j}\left(1+o(1)\right).
\end{equation}
\end{lemma}
\begin{remark}
Lemma \ref{lemma:GA} is almost the statement of Theorem \ref{theo:pol}.
\end{remark}
 
\paragraph{Proof of Lemma \ref{lemma:GA}:}
We now suppose that there are $I_i$ identifications between the vertices indexed by $i$ labels and $I_j$ identifications between the vertices indexed by $j$ labels. Note that by our definition, such a graph is admissible if and only if it consists of $I_i+I_j+1$ cycles. See for example Figures \ref{fig:admi1} and \ref{fig:admi2}. As seen earlier in the case of a simple cycle, the case of a cycle of size 2 has to be considered separately. Thus we denote by $b$ the number of cycles of size 2.

\begin{figure}[ht!]
	\begin{subfigure}[t]{.33\linewidth}
	\captionsetup{width=.9\linewidth}

	\centering
	\begin{tikzpicture}
		\draw[BrickRed, line width=.12em] (0,0) circle (1.5cm);
		\node[fill=BrickRed, label=$j_{1}$, inner sep = 0pt, minimum size=.2cm] (40) at (360/4: 1.5cm) {};
		\node[fill=BrickRed, label=left:$i_{1}$, inner sep = 0pt, minimum size=.2cm] (41) at (2*360/4: 1.5cm) {};
		\node[fill=BrickRed, label=below:$j_{3}$, inner sep = 0pt, minimum size=.2cm] at (3*360/4: 1.5cm) {};
		\node[fill=BrickRed, label=left:$i_{2}{=}i_3$, inner sep = 0pt, minimum size=.2cm] at (4*360/4: 1.5cm) {};
		\draw[BrickRed, line width = .12em] (2.25,0) circle (.75cm);
		\node[fill=BrickRed, label=left:$j_2$, inner sep = 0pt, minimum size = .2cm] at (3,0) {};
		\foreach \b in {1,2,...,4}{
			\foreach \a in {1,2,3}{
			\node[fill=RoyalBlue, inner sep = 0pt, minimum size=.2cm] (\b+\a) at (\b*360/4+\a*360/16: 1.5cm) {};
			}
		}
		\begin{scope}[xshift=2.25cm]
		\foreach \a in {1,2,3}{
			\node[fill=RoyalBlue, inner sep=0pt, minimum size=.2cm] (5+\a) at (45*\a: .75cm) {};
		}
		\foreach \a in {1,2,3}{
			\node[fill=RoyalBlue, inner sep=0pt, minimum size=.2cm] (6+\a) at (45*\a+180: .75cm) {};
		}
		\end{scope}
		
		\draw[-, RoyalBlue] (1+1) edge[out=175, in=95, line width=.12em] (1+3);
		\draw[-, RoyalBlue] (2+1) edge[bend left, line width=.12em] (2+2);
		\draw[-, RoyalBlue] (3+2) edge[bend left, line width=.12em] (3+3);
		\draw[-, RoyalBlue] (4+1) edge[bend left, line width=.12em] (4+2);
		\draw[-, RoyalBlue] (5+1) edge[bend left, line width=.12em] (5+2);
		\draw[-, RoyalBlue] (6+1) edge[out=270, in=270, line width=.12em] (6+3);

		\draw[-, RoyalBlue] (1+2) edge[bend left, line width=.12em] (2+3);
		\draw[-, RoyalBlue] (2+3) edge[bend left, line width=.12em] (3+1);
		\draw[-, RoyalBlue] (3+1) edge[bend left, line width=.12em] (4+3);
		\draw[-, RoyalBlue] (4+3) edge[bend left, line width=.12em] (1+2);

		\draw[-, RoyalBlue] (5+3) edge[bend left, line width=.12em] (6+2);
	\end{tikzpicture}
	\caption{\raggedright Admissible graph with the $i$-identification $i_2=i_3$ for $k=q=3$.}
	\label{fig:admi1}
	\end{subfigure}
	\begin{subfigure}[t]{.33\linewidth}
	\captionsetup{width=.9\linewidth}
	\centering
	\begin{tikzpicture}
		\node[fill=BrickRed, inner sep = 0pt, minimum size = .2cm, label=below:$j_1$] (1) at (0,0) {};
		\node[fill=BrickRed, inner sep = 0pt, minimum size = .2cm, label=above:$i_1$] (2) at (360/3-30: 2cm) {};
		\node[fill=BrickRed, inner sep = 0pt, minimum size = .2cm, label=below:$i_2$] (3) at (2*360/3-30: 2cm) {};
		\node[fill=BrickRed, inner sep = 0pt, minimum size = .2cm, label=below:$i_3$] (4) at (3*360/3-30: 2cm) {};
		\draw[-, BrickRed, line width=.12em] (1) edge[bend left]  node[draw=black,line width=.05em,pos=.20, fill=RoyalBlue, inner sep = 0pt, minimum size = .2cm] (5) {} node[draw=black,line width=.05em,midway, fill=RoyalBlue, inner sep = 0pt, minimum size = .2cm] (6) {} node[draw=black,line width=.05em,pos=.80, fill=RoyalBlue, inner sep = 0pt, minimum size = .2cm] (7) {} (2) ;
		\draw[-, BrickRed, line width=.12em] (1) edge[bend right]  node[draw=black,line width=.05em,pos=.20, fill=RoyalBlue, inner sep = 0pt, minimum size = .2cm] (8) {} node[draw=black,line width=.05em,midway, fill=RoyalBlue, inner sep = 0pt, minimum size = .2cm] (9) {} node[draw=black,line width=.05em,pos=.80, fill=RoyalBlue, inner sep = 0pt, minimum size = .2cm] (10) {} (2) ;
	
		\draw[-, BrickRed, line width=.12em] (1) edge[bend left]  node[draw=black,line width=.05em,pos=.20, fill=RoyalBlue, inner sep = 0pt, minimum size = .2cm] (11) {} node[draw=black,line width=.05em,midway, fill=RoyalBlue, inner sep = 0pt, minimum size = .2cm] (12) {} node[draw=black,line width=.05em,pos=.80, fill=RoyalBlue, inner sep = 0pt, minimum size = .2cm] (13) {} (3) ;
			\draw[-, BrickRed, line width=.12em] (1) edge[bend right]  node[draw=black,line width=.05em,pos=.20, fill=RoyalBlue, inner sep = 0pt, minimum size = .2cm] (14) {} node[draw=black,line width=.05em,midway, fill=RoyalBlue, inner sep = 0pt, minimum size = .2cm] (15) {} node[draw=black,line width=.05em,pos=.80, fill=RoyalBlue, inner sep = 0pt, minimum size = .2cm] (16) {} (3) ;
		
		\draw[-, BrickRed, line width=.12em] (1) edge[bend left]  node[draw=black,line width=.05em,pos=.20, fill=RoyalBlue, inner sep = 0pt, minimum size = .2cm] (17) {} node[draw=black,line width=.05em,midway, fill=RoyalBlue, inner sep = 0pt, minimum size = .2cm] (18) {} node[draw=black,line width=.05em,pos=.80, fill=RoyalBlue, inner sep = 0pt, minimum size = .2cm] (19) {} (4) ;
		\draw[-, BrickRed, line width=.12em] (1) edge[bend right]  node[draw=black,line width=.05em,pos=.20, fill=RoyalBlue, inner sep = 0pt, minimum size = .2cm] (20) {} node[draw=black,line width=.05em,midway, fill=RoyalBlue, inner sep = 0pt, minimum size = .2cm] (21) {} node[draw=black,line width=.05em,pos=.80, fill=RoyalBlue, inner sep = 0pt, minimum size = .2cm] (22) {} (4) ;
		
		\draw[-,RoyalBlue, line width=.12em] (5) edge[bend left=60,line width=.12em] (6);
		\draw[-,RoyalBlue, line width=.12em] (7) edge[bend left=20, line width=.12em] (8);
		\draw[-,RoyalBlue, line width=.12em] (9) edge[bend right=60, line width=.12em] (10);
		
		\draw[-,RoyalBlue, line width=.12em] (11) edge[bend left=60,line width=.12em] (13);
		\draw[-,RoyalBlue, line width=.12em] (12) edge[bend right=20, line width=.12em] (14);
		\draw[-,RoyalBlue, line width=.12em] (15) edge[bend right=60, line width=.12em] (16);
		
		\draw[-,RoyalBlue, line width=.12em] (17) edge[bend left=60,line width=.12em] (18);
		\draw[-,RoyalBlue, line width=.12em] (22) edge[bend left=20, line width=.12em] (19);
		\draw[-,RoyalBlue, line width=.12em] (20) edge[bend right=60, line width=.12em] (21);
	\end{tikzpicture}
	\caption{\raggedright Admissible graph with the $j$-identification $j_1=j_2=j_3$ for $k=q=3$.}
	\label{fig:admi2}
	\end{subfigure}
	\begin{subfigure}[t]{.33\linewidth}
	\captionsetup{width=.9\linewidth}
	\centering
		\begin{tikzpicture}
		\node[fill=BrickRed, label=$i_{1}{=}i_3$, inner sep = 0pt, minimum size=.2cm] (40) at (360/4: 1.5cm) {};
		\node[fill=BrickRed, label=left:$j_{2}{=}j_3$, inner sep = 0pt, minimum size=.2cm] (41) at (2*360/4: 1.5cm) {};
		\node[fill=BrickRed, label=below:$j_2$, inner sep = 0pt, minimum size=.2cm] (42) at (3*360/4: 1.5cm) {};
		\node[fill=BrickRed, label=left:$i_{2}$, inner sep = 0pt, minimum size=.2cm] (43) at (4*360/4: 1.5cm) {};
		
		\draw[-, BrickRed] (40) edge[bend right=40, line width=.12em] node[draw=black,line width=.05em,pos=.20, fill=RoyalBlue, inner sep = 0pt, minimum size = .2cm] (7) {} node[draw=black,line width=.05em,midway, fill=RoyalBlue, inner sep = 0pt, minimum size = .2cm] (8) {} node[draw=black,line width=.05em,pos=.80, fill=RoyalBlue, inner sep = 0pt, minimum size = .2cm] (9) {} (41);
		
		\draw[-, BrickRed] (40) edge[line width=.12em] node[draw=black,line width=.05em,pos=.20, fill=RoyalBlue, inner sep = 0pt, minimum size = .2cm] (1) {} node[draw=black,line width=.05em,midway, fill=RoyalBlue, inner sep = 0pt, minimum size = .2cm] (2) {} node[draw=black,line width=.05em,pos=.80, fill=RoyalBlue, inner sep = 0pt, minimum size = .2cm] (3) {} (41);
		
		\draw[-, BrickRed] (40) edge[bend left=40, line width=.12em] node[draw=black,line width=.05em,pos=.20, fill=RoyalBlue, inner sep = 0pt, minimum size = .2cm] (4) {} node[draw=black,line width=.05em,midway, fill=RoyalBlue, inner sep = 0pt, minimum size = .2cm] (5) {} node[draw=black,line width=.05em,pos=.80, fill=RoyalBlue, inner sep = 0pt, minimum size = .2cm] (6) {} (41);
		
		\draw[-, BrickRed] (41) edge[bend right=40, line width=.12em] node[draw=black,line width=.05em,pos=.20, fill=RoyalBlue, inner sep = 0pt, minimum size = .2cm] (10) {} node[draw=black,line width=.05em,midway, fill=RoyalBlue, inner sep = 0pt, minimum size = .2cm] (11) {} node[draw=black,line width=.05em,pos=.80, fill=RoyalBlue, inner sep = 0pt, minimum size = .2cm] (12) {} (42);
		
		\draw[-, BrickRed] (42) edge[bend right=40, line width=.12em] node[draw=black,line width=.05em,pos=.20, fill=RoyalBlue, inner sep = 0pt, minimum size = .2cm] (13) {} node[draw=black,line width=.05em,midway, fill=RoyalBlue, inner sep = 0pt, minimum size = .2cm] (14) {} node[draw=black,line width=.05em,pos=.80, fill=RoyalBlue, inner sep = 0pt, minimum size = .2cm] (15) {} (43);
		
		\draw[-, BrickRed] (43) edge[bend right=40, line width=.12em] node[draw=black,line width=.05em,pos=.20, fill=RoyalBlue, inner sep = 0pt, minimum size = .2cm] (16) {} node[draw=black,line width=.05em,midway, fill=RoyalBlue, inner sep = 0pt, minimum size = .2cm] (17) {} node[draw=black,line width=.05em,pos=.80, fill=RoyalBlue, inner sep = 0pt, minimum size = .2cm] (18) {} (40);
		
		\draw[-,RoyalBlue, line width=.12em] (18) edge[bend left=40,line width=.12em] (6);
		\draw[-,RoyalBlue, line width=.12em] (6) edge[bend left,line width=.12em] (12);
		\draw[-,RoyalBlue, line width=.12em] (12) edge[bend left,line width=.12em] (15);
		\draw[-,RoyalBlue, line width=.12em] (15) edge[bend left,line width=.12em] (18);
		\draw[-,RoyalBlue, line width=.12em] (10) edge[bend left=60,line width=.12em] (11);
		\draw[-,RoyalBlue, line width=.12em] (13) edge[bend left=60,line width=.12em] (14);
		\draw[-,RoyalBlue, line width=.12em] (16) edge[bend left=60,line width=.12em] (17);
		\draw[-,RoyalBlue, line width=.12em] (4) edge[bend left=60,line width=.12em] (5);
		\draw[-,RoyalBlue, line width=.12em] (3) edge[line width=.12em] (9);
		\draw[-,RoyalBlue, line width=.12em] (1) edge[line width=.12em] (7);
		\draw[-,RoyalBlue, line width=.12em] (2) edge (8);
		\end{tikzpicture}
		\caption{\raggedright Non-admissible graph for $k=q=3$.}
		\label{fig:nonadmi}
	\end{subfigure}
	\caption{Examples of admissible and non-admissible graphs}
\end{figure}
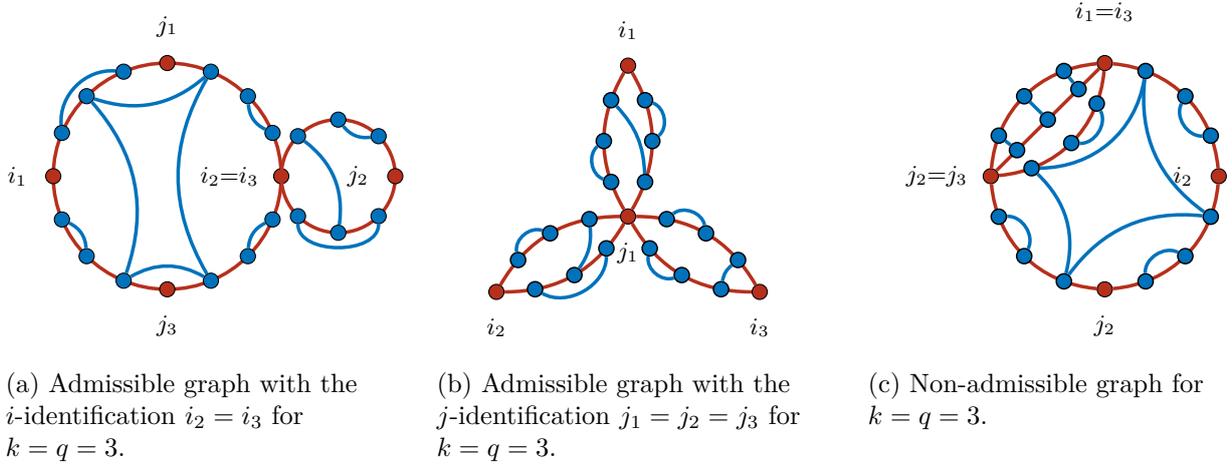
We can do a similar analysis in the case of general admissible graphs because we can realize \emph{blue} identifications inside each red cycle as they are well defined.  
Thus, recalling
$\mathcal{A}(q,I_i,I_j,b)$ from Definition \ref{defA}, we can write the contribution from all admissible graphs as
\begin{multline*}
{E}_q'(k)=\frac{1+o(1)}{n_1m^qn_0^{kq}}\sum_{I_i,I_j=0}^q\sum_{b=0}^{I_i+I_j+1}\frac{n_1!}{(n_1-q+I_i)!}\frac{m!}{(m-q+I_j)!}\times\\
\times\mathcal{A}(k,q,I_i,I_j,b)\theta_1^b(f)n_0^{kb}\theta_2^{q-b}(f)n_0^{(k-1)(q-b)+I_i+I_j+1-b}.
\end{multline*}
Thus we obtain \eqref{GA} provided we show that the error terms are negligible. 

Note that the same error terms arise as in cases a), b), c) or d) for each red cycle: their contribution is then negligible as before as soon as matchings are still performed inside each cycle. \\
Another possible contribution may come from cross-cycle \emph{blue} identifications: we now show this contribution is subleading. 
Consider the first case where such a cross-cycle identification arises around an $i$-identification or a $j$-identification: see e.g. Figure \ref{fig:crosscycle}. These \emph{blue} edges match entries of $W$ to get a non-vanishing moments. However, in order to match the corresponding $X$ entries, some new identifications are needed. This either implies that inside a niche, the matching is not a perfect matching. In this case the total final contribution is in the order of \eqref{eq:contribmatching}. Either this implies that two \emph{blue} cycles going through two cycles bear the same vertex. Thus the total contribution of such cases is in the order of $n_0^{-1}E_q$ due to the fact that one loses a possible choice for the index of a \emph{blue} cycle.
There may also exist cross-cycle \emph{blue} identifications which do not arise around an $i-$ or $j-$ identification. It is not difficult to check in this case that the total contribution is again at most of the order of $n_0^{-1}E_q$. This finishes the proof of Lemma \ref{lemma:GA}. \hfill $\square$

	\begin{figure}[ht!]
		\centering
		\begin{tikzpicture}
			\node[fill=BrickRed, label=left:$i_1$, inner sep = 0pt, minimum size=.2cm] (1) at (0,0) {};
			\node[fill=BrickRed, label=right:$j_1$, inner sep = 0pt, minimum size=.2cm] (2) at (.75,.75) {};
			\node[fill=BrickRed, label=right:$j_{p_1}$, inner sep = 0pt, minimum size=.2cm] (3) at (.75,-.75) {};
			\node[fill=BrickRed, label=left:$j_{p_3}$, inner sep = 0pt, minimum size=.2cm] (4) at (-.75,.75) {};
			\node[fill=BrickRed, label=left:$j_{p_2}$, inner sep = 0pt, minimum size=.2cm] (5) at (-.75,-.75) {};
			
			\draw[-, BrickRed] (1) edge[bend left, line width=.12em] node[draw=black,line width=.05em, midway, fill=RoyalBlue, inner sep = 0pt, minimum size = .2cm] (6) {} (2);
			\draw[-, BrickRed] (1) edge[bend right, line width=.12em] node[draw=black,line width=.05em, midway, fill=RoyalBlue, inner sep = 0pt, minimum size = .2cm] (7) {} (3);
			\draw[-, BrickRed] (1) edge[bend right, line width=.12em] node[draw=black,line width=.05em, midway, fill=RoyalBlue, inner sep = 0pt, minimum size = .2cm] (8) {} (4);
			\draw[-, BrickRed] (1) edge[bend left, line width=.12em] node[draw=black,line width=.05em, midway, fill=RoyalBlue, inner sep = 0pt, minimum size = .2cm] (9) {} (5);
			
			\draw[-,RoyalBlue, line width=.12em] (6) edge[bend right=60, line width=.12em] (8);
			\draw[-,RoyalBlue, line width=.12em] (7) edge[bend left=60,line width=.12em] (9);
		\end{tikzpicture}
		\caption{Subleading \emph{blue} identifications around an $i$-identification}\label{fig:crosscycle}
	\end{figure}
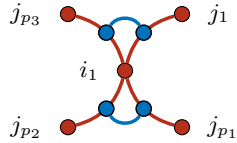

\subsubsection{Contribution from non-admissible graphs}\label{subsec:nonadmi}
Now we estimate the contribution of non-admissible graphs which we denote $E_q^{(\mathrm{NA})}$ . 
Our aim is to show the following Lemma.

\begin{lemma}\label{NAneg}
One has that \begin{equation}\label{eq:constraintqk2}\frac{E_q^{(\mathrm{NA})}}{E'_q(k)}=\mathcal{O}\left(\frac{q^3(1+q^{2k})}{n_0}\right).\end{equation}
\end{lemma}
Once Lemma \ref{NAneg} is proved, this finishes the proof of Theorem \ref{theo:pol} in the case where $f$ is an odd monomial.

\paragraph{Proof of Lemma \ref{NAneg}:}
Let us first come back to admissible graphs. Starting from the origin $i_1$ of an admissible graph $G$, there is a single way to run through the different cycles and return to the origin. Note that all the cycles are oriented e.g. counter clockwise. They are called the \emph{fundamental} cycles: they correspond to the cycles where we perform a matching on the \emph{blue} vertices.
 An admissible graph $G$ can then be (partially) encoded into a rooted tree $T=(V,E)$ as follows. The number of edges of $T$ is the number of fundamental cycles of $G$.
Given the tree $T$, one replaces each edge $e\in E$ with a cycle of length $2L(e)$ with $L(e) \geqslant1$ though one may have to choose the vertices where cycles are glued to each other.


A non-admissible graph is a multigraph $G=(V,E_1,E_2)$ where $E_1$ denotes the set of single edges and $E_2$ is the set of multiple edges. There are first multiple ways to determine the fundamental cycles. Thus, we have to count the number of non-admissible graphs labeled by their fundamental cycles (see Figure \ref{fig:glue} for an illustration).
There may be also multiple ways to go through the whole graph: we explain later how this can be counted thanks to the associated admissible graph.

Our aim is to obtain all the non admissible graphs from the set of admissible ones by adding $i-$ and $j$- identifications. This will determine the fundamental cycles and there just remains to count the number of ways to run through the graph.
Consider the tree $T$ encoding an admissible graph, we can then choose two edges and \emph{glue} them together in the sense of identifying one vertex of one edge to one of the other. This adds an identification in the initial graph and encodes a non-admissible graph. Now, while these two cycles (corresponding to these two edges) are identified at an additional vertex, there could be more identifications for the same two cycles by choosing additional vertices in each cycle to be again identified. So finally doing this step a single time, the number of possible ways to choose two cycles which are then identified at $r$ pairs of vertices is at most:
\begin{equation}\label{eq:choicefund}
C\binom{E}{2}\left(
	2q
\right)^r\leqslant {q^{r+2}},
\end{equation}   
since we need to choose two edges and then two vertices. And the number of possible ways to go through the whole graph is then multiplied by a factor at most $3^r$. 
Indeed one needs to see the fundamental cycles in the order they are initially numbered on the admissible graph. The moments of time where one can make a choice when running through the graph correspond to the vertices of degree greater than $2$. Suppose one lies on cycle $i$ (for some integer $i$) and meets a vertex of degree at least $3$. Then there are multiple possible ways to continue the walk on the graph iff this vertex belongs to cycle $i+1$ or $i-1$. Indeed because the fundamental cycles are oriented the first moment one jumps from a cycle to discover a new one is determined.\\
However, one loses a power of $n_1$ (or $m$) for each additional identification as we lose a choice of index without gaining a cycle. Assume that $q^3\ll n_0$, 
and denote by $E_q^{(\mathrm{NA},r)}$ the contribution of non admissible graphs with no cross matchings (apart from the cycle inside each red cycle).
Then one has that 
\begin{equation}\label{311}E_q^{(\mathrm{NA},r)}= \sum_r \left (\frac{q^3 3}{n_0}\right)^r E'_q(k)<<E'_q(k).\end{equation}
Hereabove $r$ denotes the total number of additionnal identifications (that is the surplus $s(G)$ of the non admissible graph).


\begin{figure}[!ht]
	\centering
	\begin{tikzpicture}
		\node[fill=BrickRed, inner sep = 0pt, minimum size = .2cm] (1) at (0,0) {};
		\node[fill=BrickRed, inner sep = 0pt, minimum size = .2cm] (2) at (1,0) {};
		\node[fill=BrickRed, inner sep = 0pt, minimum size = .2cm] (3) at (2,0) {};
		\node[fill=BrickRed, inner sep = 0pt, minimum size = .2cm] (4) at (2.7,.7) {};
		\node[fill=BrickRed, inner sep = 0pt, minimum size = .2cm] (5) at (2.7,-.7) {};
		
		\draw[-,BrickRed, line width=.12em] (1) edge[bend left=30] (2);
		\draw[-,BrickRed, line width=.12em] (1) edge[bend right=30] (2);
		\draw[-,BrickRed, line width=.12em] (2) edge[bend left=30] (3);
		\draw[-,BrickRed, line width=.12em] (2) edge[bend right=30] (3);
		\draw[-,BrickRed, line width=.12em] (3) edge[bend left=30] (4);
		\draw[-,BrickRed, line width=.12em] (3) edge[bend right=30] (4);
		\draw[-,BrickRed, line width=.12em] (3) edge[bend left=30] (5);
		\draw[-,BrickRed, line width=.12em] (3) edge[bend right=30] (5);
		
		\draw[-,dashed,Black, line width=.15em] (1) edge (2);
		\draw[-,Black, line width=.15em] (2) edge (3);
		\draw[-,Black, line width=.15em] (3) edge (4);
		\draw[-,dashed, line width=.15em] (3) edge (5);

		\draw[->,>=stealth, line width=.12em] (3,0) -- node[draw=none, pos=.5, above=-.5cm] {Gluing} (4.5,0);
		\begin{scope}[xshift=5cm]
			\node[fill=BrickRed, inner sep = 0pt, minimum size = .2cm] (01) at (0,0) {};
			\node[fill=BrickRed, inner sep = 0pt, minimum size = .2cm] (02) at (1,0) {};
			\node[fill=BrickRed, inner sep = 0pt, minimum size = .2cm] (03) at (1.7,.7) {};
			\node[fill=BrickRed, inner sep = 0pt, minimum size = .2cm] (04) at (1.7,-.7) {};
		\draw[->,>=stealth, line width=.12em] (2,0) -- node[draw=none, pos=.5, above=-.4cm] {$r=3$} (3.5,0);

		\end{scope}
		
		\draw[-,dashed,Black, line width=.15em] (01) edge (04);
		\draw[-,Black, line width=.15em] (01) edge (02);
		\draw[-,Black, line width=.15em] (02) edge (03);
		\draw[-,dashed, line width=.15em] (02) edge (04);
		
		\begin{scope}[xshift=9cm]
			\node[fill=BrickRed, inner sep = 0pt, minimum size = .2cm] (001) at (0,0) {};
			\node[fill=BrickRed, inner sep = 0pt, minimum size = .2cm] (002) at (1,0) {};
			\node[fill=BrickRed, inner sep = 0pt, minimum size = .2cm] (003) at (1.7,.7) {};
			\node[fill=BrickRed, inner sep = 0pt, minimum size = .2cm] (004) at (1.7,-1) {};
		\end{scope}
		
		\draw[-,BrickRed, line width=.12em] (001) edge[bend left=30] (002);
		\draw[-,BrickRed, line width=.12em] (001) edge[bend right=30] (002);
		\draw[-,BrickRed, line width=.12em] (002) edge[bend left=30] (003);
		\draw[-,BrickRed, line width=.12em] (002) edge[bend right=30] (003);
		
		\draw[-,BrickRed, line width=.12em] (002) edge[bend left=60] (004);
		\draw[-, BrickRed, line width=.12em] (002)  edge[bend right=60] node[draw=none, pos=.25] (005) {} node[draw=none, pos=.75] (006) {} (004);
		
		\coordinate[shift={(-2.2mm,-2.2mm)}] (test) at (006);

		\draw[-,BrickRed, line width=.12em] (001) edge[bend right=30] (005.center);
		\draw[-,BrickRed, line width=.12em] (005.center) edge[bend right=80] (006.center);
		\draw[-,BrickRed, line width=.12em] (006.center) edge[bend right=80] (004);
		\draw[-,BrickRed, line width=.12em] (004) edge[out = 270, in = 270] (001);
		
		\draw[-,dashed,Black, line width=.15em] (001) edge[out = 290, in = 150] (test.center);
		\draw[-,dashed,Black, line width=.15em] (test.center) edge[out = 330, in = 250] (004);
		\draw[-,Black, line width=.15em] (001) edge (002);
		\draw[-,Black, line width=.15em] (002) edge (003);
		\draw[-,dashed, line width=.15em] (002) edge (004);
	\end{tikzpicture}
	\caption{The left picture is an admissible graph with its encoding tree. The two dashed lines correspond to the two edges we glue together. The second graph correspond to the glued tree where there is now a cycle. Now the last step consists in choosing the number of identifications:  here we have three total identifications between the two cycles.}
	\label{fig:glue}
\end{figure}
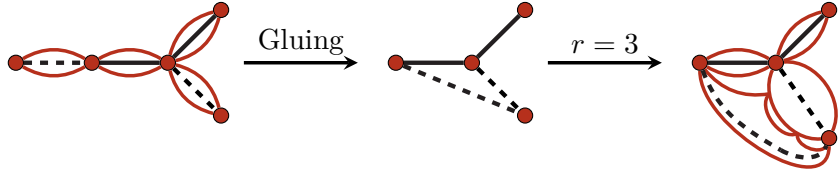

Now, once the fundamental cycles are identified, cross identifications between \emph{blue} edges from distinct niches (or fundamental cycles) are subleading unless in the following case: there are multiple cycles of length 2. Consider  a cycle of length 2, with multiplicity $p$. Then $pk$ \emph{blue} vertices have to be matched. While the leading order is given by performing a perfect matching between these vertices such as in Figure \ref{fig:multsingedge}, we can do any kind of matching and use the similar analysis we did for \eqref{eq:contribmatching}. Suppose that we have an identification between $a_1,\dotsc,a_b$ entries such that $a_1+\dotsb+a_b=pk$. For ease we suppose that $a_1=\dots=a_{b_1}=2$ and $a_{b_1+1},\dots,a_b>2$ for some $b_1\in[\![1,b-1]\!],$ then we can compare their contribution to that of the admissible graph (used to encode it) by
\begin{equation}\label{eq:highmom3}
\sum_{p=2}^q\sum_{b=1}^{\frac{m(e)k}{2}}
\frac{1}{n_0^{p-1}}
\sum_{b_1=1}^b
\sum_{\substack{a_{b_1+1},\dots,a_b>2\\\sum a_i=pk-2b_i}}
\frac{(pk)!}{((2k)!!)^{p/2}b!\prod_{i=b_1+1}^ba_i}
\frac{n_0^{b}\prod_{i=b_1+1}^b \mathds{E}\vert W_{11}\vert^{a_i}\mathds{E}\vert X_{11}\vert^{a_i}}{n_0^{kp/2}(\sigma_w^2\sigma_x^2)^{kp/2-b}}
\end{equation}

The factor of $n_0^{1-p}$ comes from the additional identifications between $i$'s and $j$'s in order to obtain a multiple edge. For instance in Figure \ref{fig:multsingedge} there are less identifications in the admissible graph than in the corresponding non-admissible graph. The first term in the summand compares the number of possible matchings of the $pk$ edges to that of a perfect matching in every single cycle. There exists a constant $C>0$ such that
\[
\frac{(pk)!}{((2k)!!)^{p/2}b!\prod_{i=b_1+1}^ba_i}\leqslant (Cp)^{kp}.
\]
The second term now comes from the number of $\ell$ indices chosen and the ratio of moments and we bound it in the same way as in \eqref{eq:contribmatching},
\[
\frac{\prod_{b_1+1}^b\mathds{E}\vert W_{11}\vert^{a_p}\mathds{E}\vert X_{11}\vert^{a_p}}{n_0^{kp/2-b}(\sigma_w\sigma_x)^{2(kp/2-b)}}
\leqslant
\frac{A^{2\sum_{i\geqslant b_1+1}a_i}}{n_0^{kp/2-b}\prod_{i\geqslant b_1+1}\sigma_w^{a_i}\sigma_x^{a_i}}
=
\left(
	\frac{A^4}{n_0\sigma_w^2\sigma_x^2}
\right)^{\frac{kp}{2}-b}.
\] 
Also in the same way as in \eqref{eq:contribmatching}, we can bound the combinatorial factor coming from the sums as 
\[
\sum_{b_1=1}^b
\sum_{\substack{a_{b_1+1},\dots,a_b>2\\\sum a_i=pk-2b_i}}
\leqslant 
(pk)^{2(\frac{pk}{2}-b)}.
\]
Finally, putting all the contribution together we have
\begin{equation}\label{eq:constraintqk}
n_0
\sum_{p=2}^q
\left(
	\frac{Cp^k}{n_0}
\right)^{p}
\sum_{b=1}^{pk/2}
\left(
	\frac{A^4p^2k^2}{n_0\sigma_w^2\sigma_x^2}
\right)^{\frac{kp}{2}-b}
=
\mathcal{O}\left(
	\frac{q^{2k}}{n_0}
\right),
\end{equation}
where we used the fact that the leading order comes from the case where $b=\frac{kp}{2}$.
Actually \eqref{eq:constraintqk} can be improved to $O (\frac{q^{kj}}{n_0^{j-1}})$ for any integer $1<j<q$.
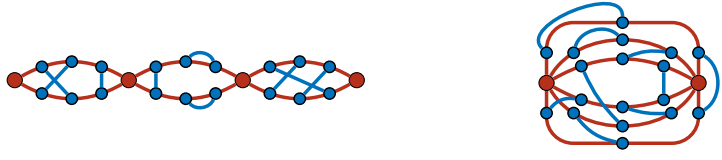
\begin{figure}[!ht]
\centering
	\begin{subfigure}{.33\linewidth}
	\centering
		\begin{tikzpicture}
		\node[fill=BrickRed, inner sep = 0pt, minimum size=.2cm] (1) at (0,0) {};
			\node[fill=BrickRed, inner sep = 0pt, minimum size=.2cm] (2) at (1.5,0) {};
			\draw[-, BrickRed, line width=.12em] (1) edge[bend left]  node[draw=black,line width=.05em,pos=.20, fill=RoyalBlue, inner sep = 0pt, minimum size = .15cm] (3) {} node[draw=black,line width=.05em,midway, fill=RoyalBlue, inner sep = 0pt, minimum size = .15cm] (4) {} node[draw=black,line width=.05em,pos=.80, fill=RoyalBlue, inner sep = 0pt, minimum size = .15cm] (5) {} (2) ;
			\draw[-, BrickRed, line width=.12em] (1) edge[bend right]  node[draw=black,line width=.05em,pos=.20, fill=RoyalBlue, inner sep = 0pt, minimum size = .15cm] (6) {} node[draw=black,line width=.05em,midway, fill=RoyalBlue, inner sep = 0pt, minimum size = .15cm] (7) {} node[draw=black,line width=.05em,pos=.80, fill=RoyalBlue, inner sep = 0pt, minimum size = .15cm] (8) {} (2) ;
			\draw[-, RoyalBlue, line width=.12em] (3) edge  (7);
			\draw[-, RoyalBlue, line width=.12em] (4) edge  (6);
			\draw[-, RoyalBlue, line width=.12em] (5) edge  (8);

		\node[fill=BrickRed, inner sep = 0pt, minimum size=.2cm] (22) at (3,0) {};
		\draw[-, BrickRed, line width=.12em] (2) edge[bend left]  node[draw=black,line width=.05em,pos=.20, fill=RoyalBlue, inner sep = 0pt, minimum size = .15cm] (30) {} node[draw=black,line width=.05em,midway, fill=RoyalBlue, inner sep = 0pt, minimum size = .15cm] (40) {} node[draw=black,line width=.05em,pos=.80, fill=RoyalBlue, inner sep = 0pt, minimum size = .15cm] (50) {} (22) ;
		\draw[-, BrickRed, line width=.12em] (2) edge[bend right]  node[draw=black,line width=.05em,pos=.20, fill=RoyalBlue, inner sep = 0pt, minimum size = .15cm] (60) {} node[draw=black,line width=.05em,midway, fill=RoyalBlue, inner sep = 0pt, minimum size = .15cm] (70) {} node[draw=black,line width=.05em,pos=.80, fill=RoyalBlue, inner sep = 0pt, minimum size = .15cm] (80) {} (22) ;
		\draw[-, RoyalBlue, line width=.12em] (30) edge  (60);
		\draw[-, RoyalBlue, line width=.12em] (40) edge[bend left=60]  (50);
		\draw[-, RoyalBlue, line width=.12em] (70) edge[bend right=60]  (80);
		
		\node[fill=BrickRed, inner sep = 0pt, minimum size=.2cm] (222) at (4.5,0) {};
\draw[-, BrickRed, line width=.12em] (22) edge[bend left]  node[draw=black,line width=.05em,pos=.20, fill=RoyalBlue, inner sep = 0pt, minimum size = .15cm] (300) {} node[draw=black,line width=.05em,midway, fill=RoyalBlue, inner sep = 0pt, minimum size = .15cm] (400) {} node[draw=black,line width=.05em,pos=.80, fill=RoyalBlue, inner sep = 0pt, minimum size = .15cm] (500) {} (222) ;
		\draw[-, BrickRed, line width=.12em] (22) edge[bend right]  node[draw=black,line width=.05em,pos=.20, fill=RoyalBlue, inner sep = 0pt, minimum size = .15cm] (600) {} node[draw=black,line width=.05em,midway, fill=RoyalBlue, inner sep = 0pt, minimum size = .15cm] (700) {} node[draw=black,line width=.05em,pos=.80, fill=RoyalBlue, inner sep = 0pt, minimum size = .15cm] (800) {} (222) ;
		\draw[-, RoyalBlue, line width=.12em] (300) edge  (800);
		\draw[-, RoyalBlue, line width=.12em] (400) edge  (600);
		\draw[-, RoyalBlue, line width=.12em] (500) edge  (700);
		
		\node[draw=none] (105) at (1,-1) {};
		\node[draw=none] (105) at (1,1) {};
		\end{tikzpicture}
	\end{subfigure}
	\begin{subfigure}{.33\linewidth}
	\centering
	\begin{tikzpicture}
		\node[fill=BrickRed, inner sep = 0pt, minimum size=.2cm] (1) at (0,0) {};
		\node[fill=BrickRed, inner sep = 0pt, minimum size=.2cm] (2) at (2,0) {};
		
		\draw[-, BrickRed] (1) edge[bend left, line width=.12em] node[draw=black,line width=.05em, pos=.20, fill=RoyalBlue, inner sep = 0pt, minimum size = .15cm] (30) {} node[draw=black,line width=.05em, midway, fill=RoyalBlue, inner sep = 0pt, minimum size = .15cm] (40) {} node[draw=black,line width=.05em,pos= .80, fill=RoyalBlue, inner sep = 0pt, minimum size = .15cm] (50) {} (2);
		
		\draw[-, BrickRed] (1) edge[bend left=60, line width=.12em] node[draw=black,line width=.05em, pos=.20, fill=RoyalBlue, inner sep = 0pt, minimum size = .15cm] (60) {} node[draw=black,line width=.05em, midway, fill=RoyalBlue, inner sep = 0pt, minimum size = .15cm] (70) {} node[draw=black,line width=.05em,pos= .80, fill=RoyalBlue, inner sep = 0pt, minimum size = .15cm] (80) {} (2);
		\node[draw=none] (3) at (1,.5) {};
		\coordinate[shift={(0mm,3mm)}] (n) at (3);
		\draw[-, BrickRed, rounded corners= 10pt, line width=.12em] (1) |- (n) -| (2);

		\node[draw=black,line width=.05em, fill=RoyalBlue, inner sep = 0pt, minimum size = .15cm] (90) at (0,.4) {};
		\node[draw=black,line width=.05em, fill=RoyalBlue, inner sep = 0pt, minimum size = .15cm] (110) at (2,.4) {};
		
		\node[draw=black,line width=.05em, fill=RoyalBlue, inner sep = 0pt, minimum size = .15cm] (100) at (1,.8) {};

		\draw[-, BrickRed] (1) edge[bend right, line width=.12em] node[draw=black,line width=.05em, pos=.20, fill=RoyalBlue, inner sep = 0pt, minimum size = .15cm] (35) {} node[draw=black,line width=.05em, midway, fill=RoyalBlue, inner sep = 0pt, minimum size = .15cm] (45) {} node[draw=black,line width=.05em,pos= .80, fill=RoyalBlue, inner sep = 0pt, minimum size = .15cm] (55) {} (2);
		\draw[-, BrickRed] (1) edge[bend right=60, line width=.12em] node[draw=black,line width=.05em, pos=.20, fill=RoyalBlue, inner sep = 0pt, minimum size = .15cm] (65) {} node[draw=black,line width=.05em, midway, fill=RoyalBlue, inner sep = 0pt, minimum size = .15cm] (75) {} node[draw=black,line width=.05em,pos= .80, fill=RoyalBlue, inner sep = 0pt, minimum size = .15cm] (85) {} (2);
		\node[draw=none] (4) at (1,-.5) {};
		\coordinate[shift={(0mm,-3mm)}] (m) at (4);
		\draw[-, BrickRed, rounded corners= 10pt, line width=.12em] (1) |- (m) -| (2);
		
		\node[draw=black,line width=.05em, fill=RoyalBlue, inner sep = 0pt, minimum size = .15cm] (95) at (0,-.4) {};
		\node[draw=black,line width=.05em, fill=RoyalBlue, inner sep = 0pt, minimum size = .15cm] (115) at (2,-.4) {};
		\node[draw=black,line width=.05em, fill=RoyalBlue, inner sep = 0pt, minimum size = .15cm] (105) at (1,-.8) {};

		\draw[-,RoyalBlue, line width=.08em] (95) edge[bend left, line width=.12em] (35);
		\draw[-,RoyalBlue, line width=.08em] (65) edge[bend right=20, line width=.12em] (105);
		\draw[-,RoyalBlue, line width=.08em] (75) edge[bend left=10, line width=.12em] (30);
		\draw[-,RoyalBlue, line width=.08em] (45) edge[bend right=10, line width=.12em] (85);
		\draw[-,RoyalBlue, line width=.08em] (55) edge[line width=.12em] (50);
		\draw[-,RoyalBlue, line width=.08em] (115) edge[bend right=60, line width=.12em] (110);
		\draw[-,RoyalBlue, line width=.08em] (90) edge[out=135, in=90, line width=.12em] (100);
		\draw[-,RoyalBlue, line width=.08em] (60) edge[bend left=60, line width=.12em] (70);
		\draw[-,RoyalBlue, line width=.08em] (40) edge[bend left=15, line width=.12em] (80);

\node[draw=none] (105) at (1,-1.1) {};
		\node[draw=none] (105) at (1,1.1) {};
	\end{tikzpicture}
	\end{subfigure}
	\caption{Different behavior between an admissible graph and a multiple edge.}
	\label{fig:multsingedge}
\end{figure}

Thus, combining \eqref{311}, \eqref{eq:choicefund} and \eqref{eq:constraintqk} finishes the proof of Lemma \ref{NAneg}.  \hfill $\square$

\subsection{Proof of Theorem \ref{theo:pol} when \texorpdfstring{$\bm{f}$}{} is a monomial of even degree: }\label{subsec:even}
In the case of an even monomial we center the function $f$, to do so we substract a constant given by the corresponding expectation. We then consider centered monomial of the form
\[
f(x)=\frac{x^k-k!!}{k!},\: k=2p \quad\text{ so that}\quad \theta_1(f)=\frac{(\sigma_w\sigma_x)^{2k}}{(k!)^2}\left(
	(2k)!!-(k!!)^2
\right)\quad\text{and}\quad \theta_2(f)=0.
\]
Here, the fact that $\theta_2(f)$ vanishes means that all admissible graphs which have at least one cycle of size greater than 2 are subleading so that we see admissible graphs consisting only in cycles of size 2 such as Figure \ref{fig:admi2} for instance. 
Note that we have seen earlier that we can write
\[
\mathds{E}\left[
	\frac{1}{k!}\left(
		\frac{(WX)_{ij}}{\sqrt{n_0}}
	\right)^k
\right]
=
\frac{1}{n_0^{k/2}k!}\mathds{E}\sum_{\ell_1,\dots,\ell_k=1}^{n_0}\prod_{p=1}^kW_{i\ell_p}X_{\ell_pj}=\frac{k!!}{k!}\left(
	\sigma_w\sigma_x
\right)^k\left(1+\mathcal{O}\left(\frac{1}{n_0}\right)\right).
\]
Thus, by developing the tracial moments of $M$ we obtain the following formula,
\begin{multline}\label{eq:moment2}
\frac{1}{n_1}\mathds{E}\left[\Tr M^q\right]
=
\left(1+\mathcal{O}\left(\frac{1}{n_0}\right)\right)
\frac{1}{n_1m^q}
\mathds{E}
\sum_{i_1,\dots,i_q}^{n_1}
\sum_{j_1,\dots,j_q}^{m}
\left[
	\frac{1}{n_0^{kq}(k!)^{2q}}\hspace{-.8em}
	\sum_{\substack{\ell_1^1,\dots\ell^1_k\\\dots\\\ell^{2q}_1\dots\ell^{2q}_k}}^{n_0}
	\prod_{p=1}^kW_{i_1\ell^1_p}X_{\ell_p^1j_1}\times\right.\\
	\times\left.\prod_{p=1}^kW_{i_2\ell^2_p}X_{\ell_p^2j_1}
	\dots
	\prod_{p=1}^kW_{i_1\ell^{2q}_p}X_{\ell_p^{2q}j_q}
	-c_0^{2q}
\right].
\end{multline}
Now it is not difficult to check that
\[
c_0^{2q}=
\left(1+\mathcal{O}\left(\frac{1}{n_0}\right)\right)\mathds{E}
\sum_{\ell_1^1,\dots\ell^1_k\\\dots\\\ell^{2q}_1\dots\ell^{2q}_k **}^{n_0}
\prod_{p=1}^kW_{i_1\ell^1_p}X_{\ell_p^1j_1}
\prod_{p=1}^kW_{i_2\ell^2_p}X_{\ell_p^2j_1}
\dots
\prod_{p=1}^kW_{i_1\ell^{2q}_p}X_{\ell_p^{2q}j_q},
\]
where the $**$ means that the $\ell$indices are matched according to a perfect matching inside each niche. Thus the centering by $c_0$ corresponds to the contribution of the admissible graphs where \emph{blue} vertices make a perfect matching inside each niche.

	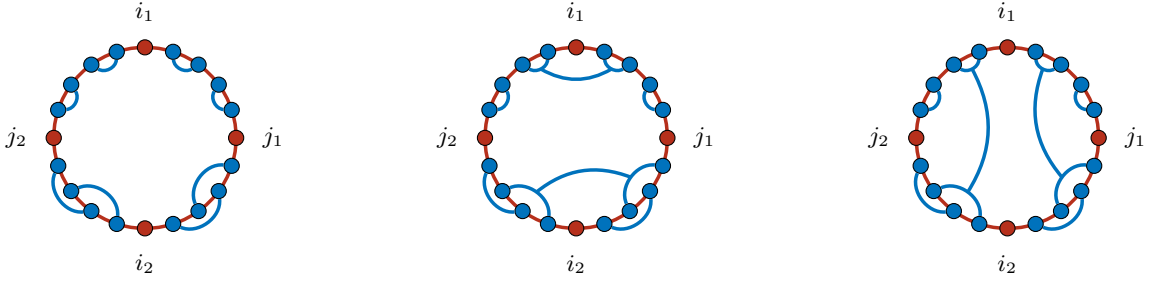
\begin{figure}[!ht]
	\begin{subfigure}{.33\linewidth}
		\centering
		\begin{tikzpicture}
		\draw[BrickRed, line width=.12em] (0,0) circle (1.2cm);
		\node[fill=BrickRed, label=$i_{1}$, inner sep = 0pt, minimum size=.2cm] (40) at (360/4: 1.2cm) {};
		\node[fill=BrickRed, label=left:$j_{2}$, inner sep = 0pt, minimum size=.2cm] (41) at (2*360/4: 1.2cm) {};
		\node[fill=BrickRed, label=below:$i_{2}$, inner sep = 0pt, minimum size=.2cm] at (3*360/4: 1.2cm) {};
		\node[fill=BrickRed, label=right:$j_{1}$, inner sep = 0pt, minimum size=.2cm] at (4*360/4: 1.2cm) {};
		\foreach \b in {1,2,...,4}{
			\foreach \a in {1,2,...,4}{
			\node[fill=RoyalBlue, inner sep = 0pt, minimum size=.2cm] (\b+\a) at (\b*360/4+\a*360/20: 1.2cm) {};
			}
		}
		
		\draw[-, RoyalBlue] (1+3) edge[bend left=60, line width=.12em] (1+4);
		\draw[-, RoyalBlue] (4+3) edge[bend left=60, line width=.12em] (4+4);
		\draw[-, RoyalBlue] (3+2) edge[bend left=60, line width=.12em] (3+4);
		\draw[-, RoyalBlue] (2+3) edge[bend left=60, line width=.12em] (2+1);

		\draw[-, RoyalBlue] (1+1) edge[bend left=60, line width=.12em] (1+2);
		\draw[-, RoyalBlue] (4+1) edge[bend left=60, line width=.12em] (4+2);
		\draw[-, RoyalBlue] (3+1) edge[bend right=60, line width=.12em] (3+3);
		\draw[-, RoyalBlue] (2+4) edge[bend right=60, line width=.12em] (2+2);
	\end{tikzpicture}
	\label{fig:leadevenbef}
	\end{subfigure} 
	\begin{subfigure}{.33\linewidth}
		\centering
		\begin{tikzpicture}
		\draw[BrickRed, line width=.12em] (0,0) circle (1.2cm);
		\node[fill=BrickRed, label=$i_{1}$, inner sep = 0pt, minimum size=.2cm] (40) at (360/4: 1.2cm) {};
		\node[fill=BrickRed, label=left:$j_{2}$, inner sep = 0pt, minimum size=.2cm] (41) at (2*360/4: 1.2cm) {};
		\node[fill=BrickRed, label=below:$i_{2}$, inner sep = 0pt, minimum size=.2cm] at (3*360/4: 1.2cm) {};
		\node[fill=BrickRed, label=right:$j_{1}$, inner sep = 0pt, minimum size=.2cm] at (4*360/4: 1.2cm) {};
		\foreach \b in {1,2,...,4}{
			\foreach \a in {1,2,...,4}{
			\node[fill=RoyalBlue, inner sep = 0pt, minimum size=.2cm] (\b+\a) at (\b*360/4+\a*360/20: 1.2cm) {};
			}
		}
		
		\draw[-, RoyalBlue] (1+3) edge[bend left=60, line width=.12em] (1+4);
		\draw[-, RoyalBlue] (4+3) edge[bend left=60, line width=.12em] node[draw=none, midway, inner sep = 0pt, minimum size = .2cm] (b) {} (4+4);
		\draw[-, RoyalBlue] (3+2) edge[bend left=60, line width=.12em] node[draw=none, midway, inner sep = 0pt, minimum size = .2cm] (d) {} (3+4);
		\draw[-, RoyalBlue] (2+3) edge[bend left=60, line width=.12em] (2+1);

		\draw[-, RoyalBlue] (1+1) edge[bend left=60, line width=.12em] node[draw=none, midway, inner sep = 0pt, minimum size = .2cm] (a) {} (1+2);
		\draw[-, RoyalBlue] (4+1) edge[bend left=60, line width=.12em] (4+2);
		\draw[-, RoyalBlue] (3+1) edge[bend right=60, line width=.12em] (3+3);
		\draw[-, RoyalBlue] (2+4) edge[bend right=60, line width=.12em] node[draw=none, midway, inner sep = 0pt, minimum size = .2cm] (c) {} (2+2);
		\draw[-, RoyalBlue] (a.center) edge[bend right, line width=.12em] (b.center);
		\draw[-, RoyalBlue] (c.center) edge[bend left, line width=.12em] (d.center);
		
%
%
%

		\end{tikzpicture}
		\label{fig:leadeven2}
	\end{subfigure}
		\begin{subfigure}{.33\linewidth}
		\centering
		\begin{tikzpicture}
		\draw[BrickRed, line width=.12em] (0,0) circle (1.2cm);
		\node[fill=BrickRed, label=$i_{1}$, inner sep = 0pt, minimum size=.2cm] (40) at (360/4: 1.2cm) {};
		\node[fill=BrickRed, label=left:$j_{2}$, inner sep = 0pt, minimum size=.2cm] (41) at (2*360/4: 1.2cm) {};
		\node[fill=BrickRed, label=below:$i_{2}$, inner sep = 0pt, minimum size=.2cm] at (3*360/4: 1.2cm) {};
		\node[fill=BrickRed, label=right:$j_{1}$, inner sep = 0pt, minimum size=.2cm] at (4*360/4: 1.2cm) {};
		\foreach \b in {1,2,...,4}{
			\foreach \a in {1,2,...,4}{
			\node[fill=RoyalBlue, inner sep = 0pt, minimum size=.2cm] (\b+\a) at (\b*360/4+\a*360/20: 1.2cm) {};
			}
		}
		
		\draw[-, RoyalBlue] (1+3) edge[bend left=60, line width=.12em] (1+4);
		\draw[-, RoyalBlue] (4+3) edge[bend left=60, line width=.12em] node[draw=none, midway, inner sep = 0pt, minimum size = .2cm] (b) {} (4+4);
		\draw[-, RoyalBlue] (3+2) edge[bend left=60, line width=.12em] node[draw=none, midway, inner sep = 0pt, minimum size = .2cm] (d) {} (3+4);
		\draw[-, RoyalBlue] (2+3) edge[bend left=60, line width=.12em] (2+1);

		\draw[-, RoyalBlue] (1+1) edge[bend left=60, line width=.12em] node[draw=none, midway, inner sep = 0pt, minimum size = .2cm] (a) {} (1+2);
		\draw[-, RoyalBlue] (4+1) edge[bend left=60, line width=.12em] (4+2);
		\draw[-, RoyalBlue] (3+1) edge[bend right=60, line width=.12em] (3+3);
		\draw[-, RoyalBlue] (2+4) edge[bend right=60, line width=.12em] node[draw=none, midway, inner sep = 0pt, minimum size = .2cm] (c) {} (2+2);
		\draw[-, RoyalBlue] (a.center) edge[bend left, line width=.12em] (c.center);
		\draw[-, RoyalBlue] (d.center) edge[bend left, line width=.12em] (b.center);
		
%
%
%

		\end{tikzpicture}
		\label{fig:leadeven3}
	\end{subfigure}
	\caption{The left figure corresponds to the leading order before centering while the two others illustrate leading order graphs after centering. The center figure involves $\mathds{E}W_{11}^4$ while the right one involves $\mathds{E}X_{11}^4$.}
		\label{fig:leadeven}	
	\end{figure}

After centering, the typical graphs may be those which have additional identifications between niches which have a common \emph{red} vertex as in Figure \ref{fig:leadeven}.
We first consider the contribution of one \emph{red} cycle to the moments and then deduce the contribution of all admissible graphs. 
One can see that to maximize the number of possible choices of \emph{blue} indices, we can first perform a perfect matching into each niche as before the centering, then we can choose either the $i$-vertices or the $j$-vertices and add identifications around the corresponding niches. This prevents having a perfect matching inside any niche (which is forbidden by the centering) but still the gives the maximal number of \emph{blue} indices. With such a matching, moments of order 4 arise in the contribution and we obtain:
\begin{multline*}
{E}_{q,1}(k)=\frac{1}{n_1(k!)^{2q}}\psi^{-q}\left(\frac{k}{2}(k!!)\right)^{2q}\left(\sigma_w^{2q(k-1)}\sigma_x^{2kq}(\mathds{E}{W_{11}^4})^q+\sigma_w^{2qk}\sigma_x^{2q(k-1)}(\mathds{E}{X_{11}^4})^q\right)+o\left(\frac{{\theta_3}(f)}{n_1}\right)
\\=\frac{1}{n_1}
\left[
	\theta_3(f)
	\left(
		\frac{\mathds{E}W_{11}^4}{\sigma_w^4}
		+
		\frac{\mathds{E}X_{11}^4}{\sigma_x^4}
	\right)
\right]^q\psi^{-q}+o\left(\frac{\theta_3(f)}{n_1}\right)\\
\end{multline*}
where we defined
\[
\theta_3(f)
=
\left(
	\frac{(\sigma_w\sigma_x)^2}{2}
	\int
	f''(\sigma_w\sigma_xx)\frac{e^{-x^2/2}}{\sqrt{2\pi}}\D x
\right)^2.
\]


Note that the contribution is of order $n_1^{-1}$ and thus is negligible compared with the contribution from odd polynomials. For the number of distinct indices we obtain $n_1^{(k-1)q}$. We could try to instead create cycles between niches as for the odd polynomial case, but one can see that we would need to create two cycles instead of one and would obtain $(k-2)q+2$ distinc indices which is of lower order.  Now if we only create one cycle, we need to perform at least identifications between three vertices in each niche since we would have an odd number of \emph{blue} vertices left and the number of distinct indices becomes at most $(k-4)q+2q+1$ which is also of lower order than Figure \ref{fig:leadeven}.
If, instead of identifying between different niches we would identify \emph{blue} vertices inside the same niche we can only obtain at most $(k-4)q+2q$ distinct indices which is of lower order than Figure \ref{fig:leadeven}.

Now, in the same way, the case of a \emph{simple cycle} (i.e. with length 2 ) is slightly different due to the centering. Indeed, at least one (thus two) vertices has to be connected to some other niche. Note also that any perfect matching where the two niches are connected is of the same order, thus we obtain for the leading order
\[
{E}_1(k)=\frac{\left(\sigma_w\sigma_x\right)^{2k}}{(k!)^2}\left((2k)!!-(k!!)^2\right)
+
\mathcal{O}\left(
	\frac{(2k-2)!!k^2}{(k!)^2n_0}
\right)=\theta_1(f)
+\mathcal{O}\left(
	\frac{(2k-2)!!k^2}{(k!)^2n_0}
\right).
\]
The above formula is self explanatory.
	\begin{figure}[!ht]
		\centering
		\begin{tikzpicture}
			\node[fill=BrickRed, label=$i_{1}$, inner sep = 0pt, minimum size=.2cm] (1) at (0,0) {};
			\node[fill=BrickRed, label=$j_{1}$, inner sep = 0pt, minimum size=.2cm] (2) at (2,0) {};
			\draw[-, BrickRed, line width=.12em] (1) edge[bend left]  node[draw=black,line width=.05em,pos=.15, fill=RoyalBlue, inner sep = 0pt, minimum size = .2cm] (3) {} node[draw=black,line width=.05em,pos=.37, fill=RoyalBlue, inner sep = 0pt, minimum size = .2cm] (4) {} node[draw=black,line width=.05em,pos=.62, fill=RoyalBlue, inner sep = 0pt, minimum size = .2cm] (5) {} node[draw=black,line width=.05em,pos=.85, fill=RoyalBlue, inner sep = 0pt, minimum size = .2cm] (6) {} (2) ;
			\draw[-, BrickRed, line width=.12em] (1) edge[bend right]  node[draw=black,line width=.05em,pos=.15, fill=RoyalBlue, inner sep = 0pt, minimum size = .2cm] (7) {} node[draw=black,line width=.05em,pos=.32, fill=RoyalBlue, inner sep = 0pt, minimum size = .2cm] (8) {} node[draw=black,line width=.05em,pos=.62, fill=RoyalBlue, inner sep = 0pt, minimum size = .2cm] (9) {} node[draw=black,line width=.05em,pos=.85, fill=RoyalBlue, inner sep = 0pt, minimum size = .2cm] (10) {} (2) ;
			\draw[-, RoyalBlue, line width=.12em] (3) edge  (7);
			\draw[-, RoyalBlue, line width=.12em] (5) edge  (10);
			\draw[-, RoyalBlue, line width=.12em] (4) edge[bend left=60]  (6);
			\draw[-, RoyalBlue, line width=.12em] (8) edge[bend right=60]  (9);

		\end{tikzpicture}
		\caption{Contribution in the case $q=1$ for an even monomial.}
		\label{fig:q1even}
	\end{figure}
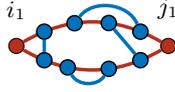

For the general case of admissible graphs with possible identifications, we use the fact that the contribution is just a product over the different cycles. For simplicity, we suppose that we have $\mathds{E} W_{11}^4\sigma_x^4=\mathds{E} X_{11}^4\sigma_w^4$. Since the contribution of the cycles of length greater than 2 are $O(n_1^{-1})$, the terms involving the $4$-th moments of the entries of $X$ and $W$ (which are not $n_1$-dependent) are subleading. Thus, this condition does not impact the overall order of the contribution but gives a simpler formula.  

The leading order of a $q$-moment, corresponding to the total contribution of admissible graphs with $2q$ edges can be written as
\begin{multline*}
{E}_q(k)
=
\frac{1+o(1)}{n_1m^qn_0^{kq}}
\sum_{I_i,I_j=0}^q
\sum_{b=0}^{I_i+I_j+1}\frac{n_1!}{(n-q+I_i)!}\frac{m!}{(m-q+I_j)!}\times
\\
\times\mathcal{A}(q,k,I_i,I_j,b)\theta_1(f)^bn_0^{kb}
	2^{I_i+I_j+1-b}
	\theta_3(f)^{q-b}
	\left(
		\frac{\mathds{E}W_{11}^4}{\sigma_w^4}
	\right)^{q-b}
n_0^{\frac{k-1}{2}(2q-2b)}
\end{multline*}
which gives asymptotically,
\begin{align*}
{E}_q(k)&=
(1+o(1))\sum_{I_i,I_j=0}^q
\sum_{b=0}^{I_i+I_j+1}
	\left(\frac{2}{n_0}\right)^{(I_i+I_j+1)-b}
	\mathcal{A}(q,I_i,I_j,b)
	\theta_1(f)^b\left[
	\theta_3(f)
		\frac{\mathds{E}W_{11}^4}{\sigma_w^4}
\right]^{q-b}
	\phi^{I_j}\psi^{I_i+1-q}
\\
&= \left(1+o(1)\right)
\sum_{\substack{I_i,I_j=0\\I_i+I_j+1=q}}\mathcal{A}(q,I_i,I_j,I_i+I_j+1)\theta_1(f)^{I_i+I_j+1}\phi^{I_j}\psi^{I_i+1-q}
\\
&=
\left(1+o(1)\right)
\sum_{I_i,I_j=0}^q\sum_{b=0}^{I_i+I_j+1}{\mathcal{A}}(q,I_i,I_j,b)\theta_1(f)^b\theta_2(f)^{q-b}\psi^{I_i+1-q}\phi^{I_j}
\end{align*}
where we used in the last equality the fact that $\theta_2(f)=0$ in order to prove the expression \eqref{eq:resultmoment}. Note again that we did not give here all the errors since we have computed them in the previous subsection, the case of even monomials can be done similarly.  
Thus we can see that only the graphs which correspond to a tree of simple cycles contribute to the moments.

We can lead  the analysis of the contribution from non-admissible graphs as in the previous section, as the \emph{non admissible structure} only concerns the \emph{red} graph  while the (centered) polynomial involves only the matching on \emph{blue} vertices. We leave the detail to the reader.
\subsection{Proof of Theorem \ref{theo:pol} when \texorpdfstring{$\bm{f}$}{} is a polynomial:}
We now suppose that we can write
\[
f(x)=\sum_{k=1}^K a_kf_k(x)\quad\text{with}\quad f_k(x)=\frac{x^k-k!!\mathds{1}_{k\text{ even}}}{k!}
\quad\text{and}\quad \sup_{k\in[\![1,K]\!]} \vert a_k\vert\leqslant C^k\quad\text{for some }C. 
\]

In particular, the parameters are in this case
\[
\theta_1(f)=
\sum_{\substack{k_1,k_2=1\\k_1+k_2=:k_0\text{ even}}}^K
\frac{a_{k_1}a_{k_2}(\sigma_w\sigma_x)^{k_0}
\left(
	k_0!!-k_1!!k_2!!\mathds{1}_{k_1\text{ even}}
\right)}{k_1!k_2!},
\theta_2(f)=
\left(
\sum_{\substack{k=1\\k \text{ odd}}}^K\frac{a_k(\sigma_w\sigma_x)^kk(k-1)!!}{k!}
\right)^2.
\]
Note that for any polynomial, by expanding the moment as in \eqref{eq:moment1}, we have to compute the following quantity, for any $k_1,\dots,k_{2q}$ integers,
\begin{multline}\label{eq:momentpolyno}
\frac{1}{n_1}
\mathds{E}\left[
\Tr M^q
\right]
=
\sum_{k_1,\dots,k_{2q}=1}^K
\frac{a_{k_1}\dots a_{k_{2q}}}{n_1m^q\prod_{i=1}^{2q} k_i!}\times
\\
\times\mathds{E}
\sum_{i_1,\dots,i_q}^{n_1}
\sum_{j_1,\dots,j_q}^{m}
\sum_{\substack{\ell_1^1,\dots\ell^1_k\\\dots\\\ell^{2q}_1\dots\ell^{2q}_k}}^{n_0}
f_{k_1}\left(
	\frac{WX}{\sqrt{n_0}}
\right)_{i_1j_1}
f_{k_2}\left(
	\frac{WX}{\sqrt{n_0}}
\right)_{i_2j_1}\dots
f_{k_{2q}}\left(
	\frac{WX}{\sqrt{n_0}}
\right)_{i_1j_q}
\end{multline}
To compute the leading term of this moment, first note that the centering creates disparity between even and odd monomials. Indeed let $q>1$, if we consider one \emph{red} cycle of length $2q$, there are now $2q$ niches of different sizes, namely $k_1,\dots,k_{2q}$. We first bound these moments in order to see that, in each cycle, the niches with an even number of vertices are subleading so that the dominant term in the asymptotic expansion of the moment corresponds to admissible graphs with only odd niches when expanding the polynomial. 
	\begin{figure}[!ht]
		\centering
		\begin{tikzpicture}
		\draw[BrickRed, line width=.12em] (0,0) circle (1.2cm);
		\node[fill=BrickRed, label=$i_{1}$, inner sep = 0pt, minimum size=.2cm] (40) at (360/4: 1.2cm) {};
		\node[fill=BrickRed, label=left:$j_{2}$, inner sep = 0pt, minimum size=.2cm] (41) at (2*360/4: 1.2cm) {};
		\node[fill=BrickRed, label=below:$i_{2}$, inner sep = 0pt, minimum size=.2cm] at (3*360/4: 1.2cm) {};
		\node[fill=BrickRed, label=right:$j_{1}$, inner sep = 0pt, minimum size=.2cm] at (4*360/4: 1.2cm) {};

			\foreach \a in {1,2,...,5}{
			\node[fill=RoyalBlue, inner sep = 0pt, minimum size=.2cm] (1+\a) at (360/4+\a*360/24: 1.2cm) {};
			}
			
			\foreach \a in {1,2}{
			\node[fill=RoyalBlue, inner sep = 0pt, minimum size=.2cm] (2+\a) at (2*360/4+\a*360/12: 1.2cm) {};
			}
			
			\foreach \a in {1,2,3}{
			\node[fill=RoyalBlue, inner sep = 0pt, minimum size=.2cm] (3+\a) at (3*360/4+\a*360/16: 1.2cm) {};
			}
			
			\foreach \a in {1,2,...,4}{
			\node[fill=RoyalBlue, inner sep = 0pt, minimum size=.2cm] (4+\a) at (4*360/4+\a*360/20: 1.2cm) {};
			}
		
		\draw[-, RoyalBlue] (1+5) edge[bend right, line width=.12em] (4+4);
		\draw[-, RoyalBlue] (4+4) edge[bend right, line width=.12em] (3+2);
		\draw[-, RoyalBlue] (3+2) edge[bend right, line width=.12em] (2+2);
		\draw[-, RoyalBlue] (2+2) edge[bend right, line width=.12em] (2+1);
		\draw[-, RoyalBlue] (2+1) edge[bend right, line width=.12em] (1+5);

		\draw[-, RoyalBlue] (1+4) edge[bend right=60, line width=.12em] (1+3);
		\draw[-, RoyalBlue] (1+2) edge[bend right=60, line width=.12em] (1+1);
		\draw[-, RoyalBlue] (4+3) edge[bend right=60, line width=.12em] (4+2);
		\draw[-, RoyalBlue] (4+2) edge[bend right=60, line width=.12em] (4+1);
		\draw[-, RoyalBlue] (3+1) edge[bend right=60, line width=.12em] (3+3);
		\end{tikzpicture}
		\caption{Admissible graph in the case of a polynomial with $(k_1,k_2,k_3,k_4)=(4,3,2,5)$.}
		\label{fig:poly}
	\end{figure}
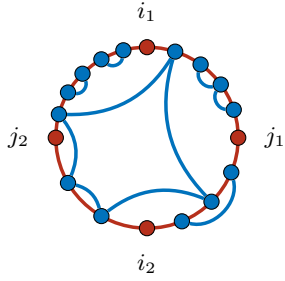
The behavior in a fundamental cycle can be understood as follows: there has to be at least one cycle connecting each niche for the odd or the centered even niches. Now, in each odd niche of length $k_i$, the leading term corresponds to  a perfect matching of the $k_i-1$ remaining vertices from \eqref{eq:contribmatching}. 
The number of pairwise distinct $l$ indices in the niche is then $(k_i-1)/2$, apart from the cycle.
However, in the even niches, since there is already a cycle, there remains an odd number of vertices to be matched.  The leading order is to disgard 2 vertices and then to perform a perfect matching of the $k_i-2$ remaining vertices. The remaining vertices are matched to a blue cycle or to an existing matching. Then, the number of distinct $l$ indices inside one niche is at most $(k_i-2)/2$ (apart from cycles). Denote the number of choices  of indices for \emph{red} and \emph{blue} vertices for a configuration of niches $k_1,\dots,k_{2q}$ by \emph{ $C(k_1,\dots,k_{2q})$}. Then we obtain
\[
\frac{n_0^{-\sum_{i=1}^{2q}\frac{k_i}{2}}}{n_1m^q}C(k_1,\dots,k_{2q})=\frac{n_0^{-\sum_{i=1}^{2q}\frac{k_i}{2}}}{n_1m^q}n_1^qm^qn_0^{1+\sum_{k_i\,\text{odd}}\frac{k_i-1}{2}+\sum_{k_i\,\text{even}}\frac{k_i-2}{2}}(1+o(1))=\frac{\psi^{1-q}}{n_0^{\frac{\#k_i\,\text{even}}{2}}}(1+o(1)).
\]
This contribution can be understood in the following way: apart from the normalization, we have to choose the $q$ $i$-indices, the $q$ $j$-indices, the $\ell$-indices. 
Thus, if we consider the contribution of cycles of size $q>1$ for the polynomial $P=\sum \frac{a_k}{k!}(X^k-k!!\mathds{1}_{k\text{ even}})$, we get the following asymptotic expansion for the moments
\begin{align*}
\eqref{eq:momentpolyno}&=\frac{1+\mathcal{O}\left(\frac{1}{\sqrt{n_0}}\right)}{n_1m^q}\sum_{\substack{k_1,\dots,k_{2q}\\k_i\text{ odd}}}
\left[
	\prod_{i=1}^{2q}\frac{a_{k_i}}{k_i!}
\right]
\frac{1}{n_0^{\sum_{i,k_i\text{ odd}}\frac{k_i}{2}}}n_1^qm^qn_0^{1+\sum_{i,k_i\text{ odd}}\frac{k_i-1}{2}}\prod_{i,k_i\text{ odd}}(\sigma_w\sigma_x)^{k_i}k_i(k_i-1)!!\\
&=\psi^{1-q}\left(
	\sum_{k\text{ odd}}a_{k}(\sigma_w\sigma_x)^{k}k(k-1)!!
\right)^{2q}
\left(
	1+\mathcal{O}\left(
		\frac{1}{\sqrt{n_0}}
	\right)
\right)
=\psi^{1-q}\theta_2^{q}(f)+\mathcal{O}\left(\frac{\theta_2^{q}(f)}{\sqrt{n_0}}\right).
\end{align*}
As we now explain, in the case of a cycle consisting of two edges decorated by $k_1$ and $k_2$ \emph{blue} vertices, there are three different possibilities:
$i)$ if $k_1$ and $k_2$ are odd: the contribution to the moment is $(\sigma_x\sigma_w)^{k_1+k_2}(k_1+k_2)!!.$; $ii)$ if $k_1$ and $k_2$ are even: the contribution is $(\sigma_w\sigma_x)^{k_1+k_2}((k_1+k_2)!!-k_1!!k_2!!); $ iii) while if $k_1$ is even and $k_2$ is odd: the leading term in the asymptotic expansion is of order $n_0^{-1/2}$ due to the discrepancy.
Thus, the 1-moment for a polynomial $f$ is
\begin{multline*}
\sum_{\substack{k_1,k_2=1\\k_1+k_2\text{ even}}}^K
\left(\frac{a_{k_1}a_{k_2}}{k_1!k_2!}(\sigma_w\sigma_x)^{k_1+k_2}
\left(
	(k_1+k_2)!!-k_1!!k_2!!\mathds{1}_{k_1\text{ even}}
\right)+
\mathcal{O}\left(
	\frac{(k_1+k_2)(k_1+k_2-1)!!}{\sqrt{n_0}k_1!k_2!}
\right)
\right)\\
=
\theta_1(f)
	+
	\mathcal{O}\left(\frac{K}{\sqrt{n_0}}
\right)
\end{multline*}
where we used the fact that for any $k_1$ and $k_2$, $(k_1+k_2)!!/(k_1!k_2!)$ is bounded.
While these analysis work in the case of a single cycle, we can do the same generalization to any (non) admissible graphs as before. Thus we get the following $q$-moment in the case of a polynomial
\[
m_q:=\frac{1}{n_1}\mathds{E}\left[
	\mathrm{Tr}M^q
\right]
=
(1+o(1))\sum_{I_i,I_j=0}^q
\sum_{b=0}^{I_i+I_j+1}
\mathcal{A}(q,I_i,I_j,b)\theta_1^b(f)\theta^{q-b}_2(f)\psi^{I_i+1-q}\phi^{I_j}.
\] This finishes the proof of Theorem \ref{theo:pol} when $f$ is a polynomial.
\subsection{Convergence of moments in probability}
In the previous subsection, we have proved convergence of the expected moments of the empirical eigenvalue distribution. We turn to the proof of the convergence in probability of these moments. 
\begin{lemma}\label{lem:variance}
Let $f(x)=\sum_k^K a_k x^k$ be a polynomial activation function and consider the associated matrix $M$ with enpirical eigenvalue distribution $\mu_{n_1}.$ Denote by $m_q$ th  moments $m_q = \frac{1}{n_1}\sum_{i=1}^{n_1}\lambda_i^q=\frac{1}{n_1}\Tr M^q$ and 
$\overline{m}_q=\mathds{E}\left[ m_q\right]$
we then have, for any $\varepsilon>0$,
\begin{equation}\label{eq:momproba}
\mathds{P}\left(
	\vert m_q - \overline{m}_q\vert > \varepsilon
\right)\xrightarrow[n_1\rightarrow\infty]{} 0.
\end{equation}
In addition there exists a constant $C$ such that  
\[
\mathrm{Var}\, m_q = \mathcal{O}\left(
\frac{(q^2K^2+q^4)C^q}{n_1^2}
\right)
\]
\end{lemma}
\begin{proof}
We can write the variance of the moments in the following way
\[
\mathrm{Var}\, m_q
=
\mathds{E}\left[
	\left(
		\frac{1}{n_1}\Tr M^q
	\right)^2
\right]
-
\overline{m}_q^2
=
\frac{1}{n_1^2}\sum_{\mathcal{G}_1,\,\mathcal{G}_2}
\sum_{\bm{\ell_1},\bm{\ell_2}}
\mathds{E}\left[
	M_{\mathcal{G}_1}(\bm{\ell_1})M_{\mathcal{G}_2}(\bm{\ell_2})
\right]
-
\mathds{E}\left[
	M_{\mathcal{G}_1}(\bm{\ell_1})
\right]
\mathds{E}\left[
	M_{\mathcal{G}_2}(\bm{\ell_2})
\right]
\]
with $\mathcal{G}_p=(G_p,\,\mathbf{i_p},\,\mathbf{j_p})$ are labeled graphs with the $i$-labels and $j$-labels given respectively by $\mathbf{i_p}$, $\mathbf{j_p}$. For a given labeled graph $\mathcal{G}=(G,\,\mathbf{i},\,\mathbf{j})$ and a matching $\bm{\ell}$, the notation $M_{\mathcal{G}}(\bm{\ell})$ corresponds to the following product after expansion
\[
M_{\mathcal{G}}(\bm{\ell})=\sum_{k_1,\dots,k_{2q}=1}^K
\frac{a_{k_1}\dots a_{k_{2q}}}{m^qn_0^{\sum k_i/2}}
\prod_{p=1}^{k_1}W_{i_1\ell^1_p}X_{\ell_p^1j_1}
\prod_{p=1}^{k_2}W_{i_2\ell^2_p}X_{\ell_p^2j_1}
\dots
\prod_{p=1}^{k_{2q}}W_{i_1\ell^{2q}_p}X_{\ell_p^{2q}j_q}.
\]
Now, note that the shape of the graph and the possible expansion of the polynomial $f$ does not depend on $n_0$, $n_1$ or $m$. 
By independence, the two graphs $\mathcal{G}_1$ and $\mathcal{G}_2$ have to share an edge otherwise the contribution to the variance is null. In particular, the concatenated graph $\mathcal{G}$ cannot be admissible. Thus we only need to consider graphs $\mathcal{G}_1$, $\mathcal{G}_2$ which share a common edge: either a red one or some $X_{\ell j}$ or $W_{i\ell}$ for some $i$, $j$, and $\ell$. In other words the concatinated graph $\mathcal{G}$ cannot be admissible.  We here assume for ease that $\mathcal{G}_1$ and $\mathcal{G}_2$ have  $2q$ edges. The case where the number of edges is different in each cycle can be similarly handled.\\
To simplify the exposition of the argument further, we suppose that $\mathcal{G}_1$ and $\mathcal{G}_2$ are both a cycle and $f$ is an odd monomial $x^k$. Note that the generalization comes from the fact that admissible graphs are a tree of cycles and non-admissible graphs yield a lower order contribution from \eqref{eq:constraintqk2}. If we suppose that the coincidence between the two graphs comes from an $i$-label and a $\ell$-label, in other words an entry $W_{i\ell}$, we have different possibilities that we now develop.

The first case consists in taking the two \emph{red} cycles and attaching them at a fixed  vertex $i_0$. We then perform a cross-cycle identification  as in Figure \ref{fig:crosscycle} in order to match two entries $W_{i_0\ell_0}$ together from $\mathcal{G}_1$ and $\mathcal{G}_2$. Once these $W$ entries are matched, note that the corresponding $X$ entries have not been matched yet. 
We then need to identify this $l$-vertex with another vertex from an adjacent niche (and then creating a \emph{blue} cycle going over the whole \emph{red} cycle) or to another vertex in the same niche. Finally, it can be seen as simply performing the dominant matching into each graph, identifying two $i$ indices and then identifying two $\emph{blue}$ edges from niches adjacent to $i$. Finally we can compute the contribution of these graphs in the covariance as
\[
\sum_{\bm{\ell_1},\bm{\ell_2}}
\mathrm{Cov}^{(1)}(M_{\mathcal{G}_1}(\bm{\ell_1}),M_{\mathcal{G}_2}(\bm{\ell_2}))
=
\mathcal{O}\left(
	q^2k^2
	\psi^{1-2q}
	\theta_2(f)^{2q}
	\left(
		\frac{\mathds{E}W_{11}^4}{\sigma_w^4}-1
	\right)
\right).
\]
Indeed, in each graph we perform the typical matching corresponding to a \emph{blue} cycle going over every niche and perfect matchings between the remaining indices in each niche. Now the fact that we identify two $W_{i_0\ell_0}$ entries create a moment of order 4 when we compute $\mathds{E}\left[M_{\mathcal{G}_1}M_{\mathcal{G}_2}\right].$ We then have to count the number of possible choices for indices: we have $n_1^{2q-1}$ choices for the $i$ indices as we identify two from $\mathcal{G}_1$ and $\mathcal{G}_2$, $m^{2q}$ for the $j$ indices, $n_0^{2+4q(k-1)/2-1}$ choices for the $\ell$ indices (2 cycles, $4q$ niches and an identification between the two graphs). Taking into account the normalization $m^{-2q}n_0^{-2kq}$, this yields a factor $\psi^{1-2q}$ asymptotically. In the same way, for general polynomial and admissible graphs, for such an identification we would obtain that 
\[
\frac{1}{n_1^2}
\sum_{\mathcal{G}_1,\mathcal{G}_2}
\sum_{\bm{\ell_1},\bm{\ell_2}}
\mathrm{Cov}^{(1)}(M_{\mathcal{G}_1}(\bm{\ell_1}),M_{\mathcal{G}_1}(\bm{\ell_1}))
=
\mathcal{O}\left(
	\frac{q^2k^2}{n_1^2\psi}
	\left(
		1-\frac{\phi}{\psi}
	\right)
	\mathfrak{m}_q^2
	\left(
		\frac{\mathds{E}W_{11}^4}{\sigma_w^4}-1
	\right)
\right)
=
\mathcal{O}\left(
	\frac{q^2k^2C^q}{n_1^2}
\right),
\]
for some $C>0$. Indeed, we get the $q^2k^2$ from the choices for the edge we want to identify between the two graphs, the constant factor in $\phi$ and $\psi$ consists in the choice of choosing a $\{i,\ell\}$ edge or a $\{j,\ell\}$ edge. Then the previous computation in the case of a cycle can be generalized to all graphs as the construction only involves one cycle in each graph. For the second equality we use the fact that $\mathfrak{m}_q\leq C^q$ as proved in the next subsection. 

The second case consists in identifying a pair of \emph{red} vertices in each graph. Such a pair is chosen in \emph{one } fundmaental cycle in both $\mathcal{G}_1$ and $\mathcal{G}_2$. 
Then we identify the pair from one graph to the other pair. This allows the existence of edges belonging to the two graphs $\mathcal{G}_1$ and $\mathcal{G}_2$. 
The whole graph $\mathcal{G}$ created by this construction is non admissible as we have two identifications and two fundamental cycles. We thus need to choose the fundamental cycles in $\mathcal{G}$.
The fundamental cycles we choose for this \emph{red} graph are given by the cycles between the two vertices with edges belonging to both graphs in each cycle. Since we need to choose a pair of vertices in each graph we have $q^4$ choices. In each fundamental cycles, we perform the typical \emph{blue} matching and we have an edge between a niche from $\mathcal{G}_1$ and a niche from $\mathcal{G}_2$ (corresponding to the cycle going over every niche for instance). Thus we have a common $W$ or $X$ entry between the two graphs and the contribution to the covariance does not vanish. Considering the $q^4$ choices for the \emph{red} vertices, we can see that we have
\[
\frac{1}{n_1^2}
\sum_{\mathcal{G}_1,\mathcal{G}_2}
\sum_{\bm{\ell_1},\bm{\ell_2}}
\mathrm{Cov}^{(1)}(M_{\mathcal{G}_1}(\bm{\ell_1}),M_{\mathcal{G}_2}(\bm{\ell_2}))
=
\mathcal{O}\left(
	\frac{q^4C^q}{n_1^2}
\right).
\]
Regarding the number of possible choices for the vertices, the number of $l$-indices is unchanged while that of $i-$ or $j$-indices decreases of $2$ if we compare to the computation of the expected moment.
Finally, we obtain that 
\[
\mathrm{Var}\, m_q = \mathcal{O}\left(
	\frac{(q^2k^2+q^4)C^q}{n_1^2}
\right).
\]
Using Bienaymé-Chebyshev inequality, one easily deduces \eqref{eq:momproba}.
\begin{figure}[!ht]
	\begin{subfigure}[t]{.49\textwidth}
	\centering
	\begin{tikzpicture}
		\draw[BrickRed, line width=.12em, fill=LimeGreen!30] (0,0) circle (1cm);
		\node[fill=BrickRed, inner sep = 0pt, minimum size=.2cm] (40) at (360/4: 1cm) {};
		\node[fill=BrickRed, label=left:\large $\mathcal{G}_1$, inner sep = 0pt, minimum size=.2cm] (41) at (2*360/4: 1cm) {};
		\node[fill=BrickRed, inner sep = 0pt, minimum size=.2cm] at (3*360/4: 1cm) {};
		\foreach \b in {1,2,...,4}{
			\foreach \a in {1,2,3}{
			\node[fill=RoyalBlue, inner sep = 0pt, minimum size=.2cm] (\b+\a) at (\b*360/4+\a*360/16: 1cm) {};
			}
		}
		
		\draw[-, RoyalBlue] (1+1) edge[bend left, line width=.12em] (2+1);
		\draw[-, RoyalBlue] (2+1) edge[bend left, line width=.12em] (3+1);
		\draw[-, Black] (3+1) edge[bend left, line width=.30em] (4+3);
		\draw[-, Orange] (3+1) edge[bend left, line width=.20em] (4+3);		
		\draw[-, RoyalBlue] (4+3) edge[bend left, line width=.12em] (1+1);
		
		\draw[-, RoyalBlue] (1+2) edge[bend left=60, line width=.12em] (1+3);
		\draw[-, RoyalBlue] (2+2) edge[bend left=60, line width=.12em] (2+3);
		\draw[-, RoyalBlue] (3+2) edge[bend left=60, line width=.12em] (3+3);
		\draw[-, RoyalBlue] (4+1) edge[bend left=60, line width=.12em] (4+2);
		
		\begin{scope}[xshift=2cm]
			\draw[BrickRed, line width=.12em, fill=ProcessBlue!30] (0,0) circle (1cm);
			\node[fill=BrickRed, inner sep = 0pt, minimum size=.2cm] (40) at (360/4: 1cm) {};
			\node[fill=BrickRed, label=left:$\bm{i_{0}}$, inner sep = 0pt, minimum size=.2cm] (41) at (2*360/4: 1cm) {};
			\node[fill=BrickRed, inner sep = 0pt, minimum size=.2cm] at (3*360/4: 1cm) {};
			\node[fill=BrickRed, label=right:\large $\mathcal{G}_2$, inner sep = 0pt, minimum size=.2cm] at (4*360/4: 1cm) {};
			\foreach \b in {1,2,...,4}{
				\foreach \a in {1,2,3}{
				\node[fill=RoyalBlue, inner sep = 0pt, minimum size=.2cm] (0\b+0\a) at (\b*360/4+\a*360/16: 1cm) {};
				}
			}
			
			\draw[-, RoyalBlue] (01+03) edge[bend left, line width=.12em] (02+01);
			\draw[-, RoyalBlue] (02+01) edge[bend left, line width=.12em] (03+03);
			\draw[-, RoyalBlue] (03+03) edge[bend left, line width=.12em] (04+03);
			\draw[-, RoyalBlue] (04+03) edge[bend left, line width=.12em] (01+03);
			
			\draw[-, Black] (01+01) edge[bend left=60, line width=.30em] (01+02);
			\draw[-, Orange] (01+01) edge[bend left=60, line width=.20em] (01+02);
			\draw[-, RoyalBlue] (02+02) edge[bend left=60, line width=.12em] (02+03);
			\draw[-, RoyalBlue] (03+01) edge[bend left=60, line width=.12em] (03+02);
			\draw[-, RoyalBlue] (04+01) edge[bend left=60, line width=.12em] (04+02);
		\end{scope}
		
		\draw[-, Black] (4+3) edge[bend left, line width=.30em] (01+01);
		\draw[-, Orange] (4+3) edge[bend left, line width=.20em] (01+01);
		
		\node[draw=none] at (1,-1.65) {};
	\end{tikzpicture}
	\caption{In this figure, the two highlighted cycles correspond to the graph $\mathcal{G}_1$ and $\mathcal{G}_2$ which are attached at ${i_0}.$ We perform a typical \emph{blue} matching in each graph and then add an identification between the two graphs. The highlighted orange edges  correspond to the edges common to the two graphs, which  yields a moment of order 4.}
	\end{subfigure}
	\begin{subfigure}[t]{.5\textwidth}
	\centering
	\begin{tikzpicture}
		\draw[BrickRed, line width=.12em] (0,0) circle (1cm);
		\node[fill=BrickRed, label={[label distance=-5pt]above:$\bm{i_0^\prime}$}, inner sep = 0pt, minimum size=.2cm] (3) at (360/4: 1cm) {};
		\node[fill=BrickRed, inner sep = 0pt, minimum size=.2cm] (41) at (2*360/4: 1cm) {};
		\node[fill=BrickRed, ,label={[label distance=-5pt]below:$\bm{i_0}$}, inner sep = 0pt, minimum size=.2cm] (4) at (3*360/4: 1cm) {};
		\node[fill=BrickRed, inner sep = 0pt, minimum size=.2cm] (42) at (4*360/4: 1cm) {};
		\foreach \b in {1,2,...,4}{
			\foreach \a in {1,2,3}{
			\node[fill=RoyalBlue, inner sep = 0pt, minimum size=.2cm] (\b+\a) at (\b*360/4+\a*360/16: 1cm) {};
			}
		}
		
		\node[fill=BrickRed, inner sep = 0pt, minimum size=.2cm] (1) at (-3,0) {};

		\node[fill=BrickRed, inner sep = 0pt, minimum size=.2cm] (2) at (3,0) {};
		
		\begin{scope}[on background layer]
    		\path[draw=none] [fill=LimeGreen!30] (3.center) to [ bend right] (1.center);    
        	\path[draw=none] [fill=LimeGreen!30] (1.center) to [bend right] (4.center);        			\path[draw=none] [fill=LimeGreen!30] (1.center) to [bend left=55] (42.center);
        	\path[draw=none] [fill=LimeGreen!30] (1.220) to [bend left=55] (42.320);
        	\path[draw=none] [fill=LimeGreen!30] (1.center) to [bend right=55] (42.center);
        	
        	\path[draw=none] [fill=ProcessBlue!30] (2.center) to [ bend right] (3.center);    
        	\path[draw=none] [fill=ProcessBlue!30] (4.center) to [bend right] (2.center);        			\path[draw=none] [fill=ProcessBlue!30] (41.center) to [bend left=55] (2.center);
        	\path[draw=none] [fill=ProcessBlue!30] (41.220) to [bend left=55] (2.320);
        	\path[draw=none] [fill=ProcessBlue!30] (41.center) to [bend right=55] (2.center);

        	\draw[draw=none,line width=.12em, fill=White] (0,0) circle (1cm);

		\end{scope}

		\draw[-, BrickRed, line width=.12em] 
			(3) edge[bend right] 
				node[draw=black, line width=.05em, fill=RoyalBlue, inner sep=0pt, minimum size=.2cm, pos=.25] (311) {} 
				node[draw=black, line width=.05em, fill=RoyalBlue, inner sep=0pt, minimum size=.2cm, pos=.5] (312) {} 
				node[draw=black, line width=.05em, label={[Black]135:\large $\mathcal{G}_1$}, fill=RoyalBlue, inner sep=0pt, minimum size=.2cm, pos=.75] (313) {} 
			(1);
			
		\draw[-, BrickRed, line width=.12em] 
			(1) edge[bend right] 
				node[draw=black, line width=.05em, fill=RoyalBlue, inner sep=0pt, minimum size=.2cm, pos=.25] (141) {} 
				node[draw=black, line width=.05em, fill=RoyalBlue, inner sep=0pt, minimum size=.2cm, pos=.5] (142) {} 
				node[draw=black, line width=.05em, fill=RoyalBlue, inner sep=0pt, minimum size=.2cm, pos=.75] (143) {} 
			(4);
			
		\draw[-, BrickRed, line width=.12em] 
			(4) edge[bend right] 
				node[draw=black, line width=.05em, fill=RoyalBlue, inner sep=0pt, minimum size=.2cm, pos=.25] (421) {} 
				node[draw=black, line width=.05em, fill=RoyalBlue, inner sep=0pt, minimum size=.2cm, pos=.5] (422) {} 
				node[draw=black, line width=.05em, fill=RoyalBlue, inner sep=0pt, minimum size=.2cm, pos=.75] (423) {} 
			(2);
			
		\draw[-, BrickRed, line width=.12em] 
			(2) edge[bend right] 
				node[draw=black, line width=.05em, label={[Black]45:\large $\mathcal{G}_2$}, fill=RoyalBlue, inner sep=0pt, minimum size=.2cm, pos=.25] (231) {} 
				node[draw=black, line width=.05em, fill=RoyalBlue, inner sep=0pt, minimum size=.2cm, pos=.5] (232) {} 
				node[draw=black, line width=.05em, fill=RoyalBlue, inner sep=0pt, minimum size=.2cm, pos=.75] (233) {} 
			(3);

		\draw[-, RoyalBlue] (1+1) edge[bend left=10, line width=.12em] (2+3);
		\draw[-, RoyalBlue] (3+1) edge[bend left=10, line width=.12em] (4+3);
		\draw[-, RoyalBlue] (1+2) edge[bend left=60, line width=.12em] (1+3);
		\draw[-, RoyalBlue] (2+1) edge[bend left=60, line width=.12em] (2+2);
		\draw[-, RoyalBlue] (3+2) edge[bend left=60, line width=.12em] (3+3);
		\draw[-, RoyalBlue] (4+1) edge[bend left=60, line width=.12em] (4+2);
		
		\draw[-, RoyalBlue] (311) edge[bend right, line width=.12em] (143);
		\draw[-, RoyalBlue] (312) edge[bend left=60, line width=.12em] (313);
		\draw[-, RoyalBlue] (141) edge[bend left=60, line width=.12em] (142);
		\draw[-, RoyalBlue] (421) edge[bend right, line width=.12em] (233);
		\draw[-, RoyalBlue] (422) edge[bend left=60, line width=.12em] (423);
		\draw[-, RoyalBlue] (231) edge[bend left=60, line width=.12em] (232);

		\draw[-, Black] (311) edge[bend left=60, line width=.30em] (233);
		\draw[-, Orange] (311) edge[bend left=60, line width=.20em] (233);
		
		\draw[-, Black] (4+3) edge[bend left=60, line width=.30em] (1+1);
		\draw[-, Orange] (4+3) edge[bend left=60, line width=.20em] (1+1);
		
		\draw[-, Black] (2+3) edge[bend left=60, line width=.30em] (3+1);
		\draw[-, Orange] (2+3) edge[bend left=60, line width=.20em] (3+1);
		
		\draw[-, Black] (421) edge[bend left=60, line width=.30em] (143);
		\draw[-, Orange] (421) edge[bend left=60, line width=.20em] (143);
	\end{tikzpicture}
	\caption{In this figure, the two highlighted cycles correspond to the graph $\mathcal{G}_1$ and $\mathcal{G}_2$ which are attached at two vertices ${i_0}$ and $i_0^\prime$. The graph is non-admissible and we choose the fundamental cycles so that neither $\mathcal{G}_1$ or $\mathcal{G}_2$ are fundamental cycles. The typical matching in the chosen cycles create common edges between the two graphs highlighted in orange on the figure. }
	\end{subfigure}
\end{figure}
\end{proof}

\subsection{From bounded to sub-Gaussian random variables}
We have computed the limiting expected moments in the case of bounded random variables. However, note that while high moments of $W$ or $X$ can appear in the error terms, as in \eqref{eq:highmom1}, \eqref{eq:highmoment2} and \eqref{eq:highmom3}, one may use for such sub-Gaussian random variables the following bound 
\[
\mathds{E}\left[\vert X_{11}\vert ^k\right]\leqslant C^kk^{k/\alpha},\quad \mathds{E}\left[\vert W_{11}\vert ^k\right]\leqslant C^kk^{k/\alpha},
\]
for some constant $C$. Thus one may simply replace in all the error terms  $A$ by $k^{1/\alpha}$. Since 
$k$ is of order $\frac{\log n_1}{\log\log n_1}$ all the errors are still $o(1)$.
\subsection{Weak convergence of the empirical spectral measure\label{subsec:36}}
In this section we briefly finish the proof of Theorems \ref{theo:result1} and \ref{theo:result2} for a polynomial activation function. The fact that the sequence of moments  \begin{equation}\label{eq:defmoment}
\mathfrak{m}_q
:=
\sum_{I_i,I_j=0}^q\sum_{b=0}^{I_i+I_j+1}{\mathcal{A}}(q,I_i,I_j,b)\theta_1(f)^b\theta_2(f)^{q-b}\psi^{I_i+1-q}\phi^{I_j}
\end{equation}
uniquely defines a probability measure $\mu$ so that  $\int x^q \D \mu(x) = \mathfrak{m}_q$ follows from Carleman's condition. Indeed, denote by $\Theta(q)$ the number of unlabeled cactus graphs with $q$ vertices.  It has been shown in \cite{ford1956combinatorial} that, regardless of the number of identifications or simple cycles, there exists numerical constants $\delta>0$ and  $\xi>1$ such that $\Theta(q)\sim \frac{3\delta}{4\sqrt{\pi }}\frac{\xi^{q+3/2}}{q^5}.$
Thus there exists a constant $C$ such that $
\mathfrak{m}_q\leqslant C^q.$  This can also been used to show that the measure has compact support.

\subsection{Derivation of the self-consistent equation for the Stieltjes transform}
Consider the Stieltjes transform of the limiting empirical eigenvalue distribution of $M$,
\[
G(z)=\int \frac{\D \mu(x)}{x-z}.
\]
One can also write it as the following generating function of moments, since the following equality makes sense at least on a neighborhood of infinity,
\[
-G(z)=\frac{1}{z}+\sum_{q=1}^\infty \frac{\mathfrak{m}_q}{z^{q+1}}.
\]
Using that
\[
\mathfrak{m}_q
=\psi^{1-q}
\sum_{I_i,I_j=0}^q
\sum_{b=0}^{I_i+I_j+1}
\mathcal{A}(q,I_i,I_j,b)\theta_1^b(f)\theta_2^{q-b}(f)\psi^{I_i}\phi^{I_j},
\]
one can write the Stieltjes transform as 
\[
-G(z)=\frac{1-\psi}{z}+\frac{\psi}{z}H(z)\quad\text{with}\quad H(z)=\sum_{q=0}^\infty\frac{1}{(\psi z)^q}\sum_{I_i,I_j=0}^q
\sum_{b=0}^{I_i+I_j+1}
\mathcal{A}(q,I_i,I_j,b)\theta_1^b(f)\theta_2^{q-b}(f)\psi^{I_i}\phi^{I_j}.
\]
Fix a vertex $v$ and denote $q_0$ the length of one of the fundamental cycles containing $v$. Suppose first that we have $q_0>1$, this cycle contains $2q_0$ edges with $q_0$ vertices labeled with $i$ and $q_0$ vertices labeled with $j$.
On each vertex labeled with $i$, either a graph is attached and we have a $i$-identification on this vertex, or nothing is attached. Thus, considering the formula above, we have that the contributions for identifications for each vertex is 
\[
H_\psi(z):=1-\psi+\psi H(z)\quad\text{for }i\text{-labels and}\quad H_\phi(z):=1-\phi+\phi H(z)\quad\text{for }j\text{-labels}.
\]
Also, one can see in the leading order of the moment that a cycle of length $q_0$ give a contribution of $\left(\frac{\theta_2(f)}{\psi z}\right)^{q_0}.$ Now, if the cycle is of length 1, in the same way, there is a single $i$-labeled vertex and a single $j$-labeled vertex which can give a contribution of $H_\psi$ and $H_\phi$ but the contribution of a simple cycle is not given in terms of $\theta_2(f)$ but by $\frac{\theta_1(f)}{\psi z}$. This is illustrated in Figure \ref{fig:recurs}. Thus, we have the following recursion relation for $H$,
\begin{multline*}
H(z)=1+\frac{H_\phi(z) H_\psi(z)\theta_1}{\psi z}+\sum_{q_0=2}^\infty \left(\frac{H_\phi(z) H_\psi(z) \theta_2}{\psi z}\right)^{q_0}\\
=1+\frac{H_\phi(z)H_\psi(z)(\theta_1-\theta_2)}{\psi z}+\frac{H_\phi(z)H_\psi(z)\theta_2}{\psi z-H_\phi(z)H_\psi(z)\theta_2}.
\end{multline*}
Note that we obtain the final equation from Theorem \ref{theo:result2} by noting that $H_\psi(z)=-zG(z)$ and $H_\phi(z)=-z\tilde{G}(z)$.
	\begin{figure}[!ht]
		\begin{subfigure}[t]{.5\linewidth}
		\centering
		\begin{tikzpicture}
			\draw[BrickRed, line width = .12em, fill=BrickRed!15] (0,0) circle (1.35cm) ;
	  \node[fill=BrickRed, label={[label distance = .5mm]below:$i_{1}$}, inner sep = 0pt, minimum size=.2cm] (1) at (360/6+30: 1.35cm) {};
	  \node[fill=BrickRed, label=315:$j_{3}$, inner sep = 0pt, minimum size=.2cm] (2) at (2*360/6+30: 1.35cm) {};
	  \node[fill=BrickRed, label=45:$i_3$, inner sep = 0pt, minimum size=.2cm] (3) at (3*360/6+30: 1.35cm)  {};
	  \node[fill=BrickRed, label=	above:$j_{2}$, inner sep = 0pt, minimum size=.2cm] (4) at (4*360/6+30: 1.35cm) {};
	  \node[fill=BrickRed, label=135:$i_{2}$, inner sep = 0pt, minimum size=.2cm] (5) at (5*360/6+30: 1.35cm) {};
	  \node[fill=BrickRed, label=225:$j_{1}$, inner sep = 0pt, minimum size=.2cm] (6) at (6*360/6+30: 1.35cm) {};
	  \node[draw=none] (1234) at (0,0) {${\displaystyle{\left(\frac{\theta_2}{\psi z}\right)^{\!q_0}}}$};
	  
	  \begin{scope}[on background layer]
    \path [draw, LimeGreen,fill=LimeGreen!30] (1) to [loop above, looseness=10, min distance=1.7cm, out=60, in=120] (1);
\end{scope}  
	\node[draw=none] (1235) at (360/6+30:1.98cm) {$H_\psi$};
	
	\begin{scope}[on background layer]
	\path [draw, LimeGreen,fill=LimeGreen!30] (3) to [loop above, looseness=10, min distance=1.7cm, out=180, in=240] (3);
\end{scope}  
	\node[draw=none] (1235) at (3*360/6+30:1.98cm) {$H_\psi$};
	
	\begin{scope}[on background layer]
	\path [draw, LimeGreen,fill=LimeGreen!30] (5) to [loop above, looseness=10, min distance=1.7cm, out=0, in=300] (5);
\end{scope}  
	\node[draw=none] (1235) at (5*360/6+30:1.98cm) {$H_\psi$};
	
	\begin{scope}[on background layer]
	\path [draw, ProcessBlue,fill=ProcessBlue!30] (2) to [loop above, looseness=10, min distance=1.7cm, out=120, in=180] (2);
\end{scope}  
	\node[draw=none] (1235) at (2*360/6+30:1.98cm) {$H_\phi$};
	
	\begin{scope}[on background layer]
	\path [draw, ProcessBlue,fill=ProcessBlue!30] (4) to [looseness=10, min distance=1.7cm, out=240, in=300] (4);
\end{scope}  
	\node[draw=none] (1235) at (4*360/6+30:1.98cm) {$H_\phi$};
	
	\begin{scope}[on background layer]
	\path [draw, ProcessBlue,fill=ProcessBlue!30] (6) to [loop above, looseness=10, min distance=1.7cm, out=0, in=60] (6);
\end{scope}  
	\node[draw=none] (1235) at (6*360/6+30:1.98cm) {$H_\phi$};
	
		\end{tikzpicture}
		\caption{Contributions for the recursion formula in the case of a large cycle ($q_0=3$)}
		\end{subfigure}
		\begin{subfigure}[t]{.4\linewidth}
			\centering
			\begin{tikzpicture}			 		\node[fill=BrickRed, label=	below:$i_{1}$, inner sep = 0pt, minimum size=.2cm] (1) at (0,3) {};
			\node[fill=BrickRed, label=	below:$j_{1}$, inner sep = 0pt, minimum size=.2cm] (2) at (3,3) {};
			\begin{scope}[on background layer]
    \path[draw, BrickRed, line width=.12em] [fill=BrickRed!15] (1.center) to [ bend left=40] (2.center);
        \path[draw=none] [fill=BrickRed!15] (1.220) to [bend left] (2.320);
    \path[draw, BrickRed, line width=.12em] [fill=BrickRed!15] (1.center) to [ bend right=40] (2.center);
\end{scope} 
			\node[draw=none] (3) at (1.5,3) {$\displaystyle{\frac{\theta_1}{\psi z}}$};
			\begin{scope}[on background layer]
	\path [draw, ProcessBlue,fill=ProcessBlue!30] (2) to [loop above, looseness=10, min distance=1.7cm, out=330, in=30] (2);
\end{scope}  
	\node[draw=none] (1235) at (3.7,3) {$H_\phi$};
			\begin{scope}[on background layer]
	\path [draw, LimeGreen,fill=LimeGreen!30] (1) to [looseness=10, min distance=1.7cm, out=150, in=210] (1);
\end{scope}  
	\node[draw=none] (1235) at (-.7,3) {$H_\psi$};
	\node[draw=none] (132) at (0,0) {};
			\end{tikzpicture}
			\caption{Contribution for the recursion formula in the case of a simple cycle.}
		\end{subfigure}
		\caption{Illustration of the recursion for the derivation of the self-consistent equation.}
		\label{fig:recurs}
	\end{figure}
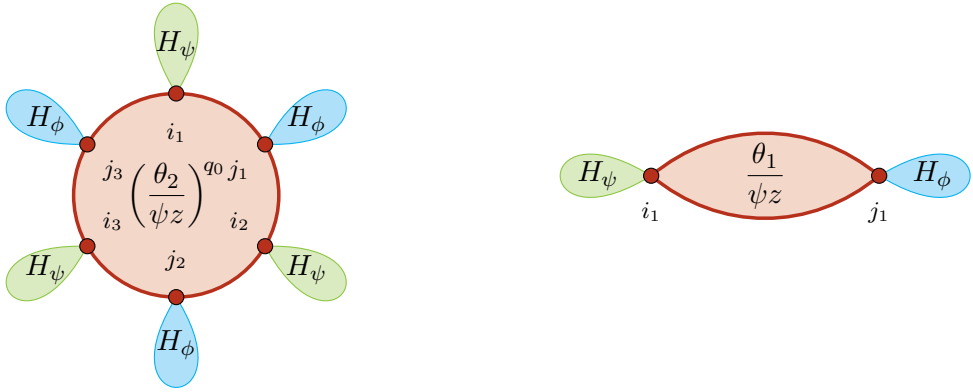
\section{Proof of Theorem \ref{theo:result1} for general activation function}\label{sec:polyapprox}
In this section, we now allow the activation function to belong to a wider class, thus proving Theorem \ref{theo:result1}. For ease, we assume that $\sigma_w=\sigma_x=1$, which can be achieved by scaling.

\begin{proof}[Proof of Theorem \ref{theo:result1}]
We begin by defining the following polynomial which approximates $f$ up to a constant, for $x\in\mathbb{R}$ we define
\begin{equation}\label{eq:defpoly}
P_k(x):=\sum_{j=1}^k f^{(j)}(0)\frac{x^j-j!!}{j!}=\sum_{j=0}^k f^{(j)}(0)\frac{x^j}{j!}-a_{n}
\quad\text{with}\quad
a_n=\sum_{j=0}^k f^{(j)}(0)\frac{j!!}{j!}
\end{equation}
with the convention that $j!!=0$ for $j$ odd and $0!!=1$. This choice ensures that the polynomial is centered with respect to the Gaussian distribution.  Thus, using  Taylor's theorem, we obtain the following approximation for any $A>0$
\begin{equation}\label{eq:approxpoly}
\sup_{x\in[-A,A]}
\left\vert
	(f(x)-a_{k-1})(x)-P_{k-1}(x)
\right\vert
\leqslant
C_f\frac{A^{(1+c_f)k}}{k!}.
\end{equation}
Now, we compare the Hermitized version of the matrix $M$ (up to finite rank modification), and define
\begin{eqnarray}&&\label{eq:defyk}
Y^{(a_k)}=f\left(\frac{WX}{\sqrt{n_0}}\right)-a_k,
\quad
Y_k=P_k\left(\frac{WX}{\sqrt{n_0}}\right),\cr
&&\label{eq:defmatdiff}
\mathcal{E}=\frac{1}{\sqrt{m}}
\begin{pmatrix}
	0 &	Y^{(a_{k-1})}-Y_k\\
	\left(Y^{(a_{k-1})}-Y_k\right)^* & 0
\end{pmatrix}.
\end{eqnarray}
We want to control the spectral radius of the $(m+n_1)\times (m+n_1)$ symmetric matrix $\mathcal{E}$.
Now consider the event, for $\delta_1\in (0,\frac{1}{2})$,
\begin{equation}\label{eq:notyf}
\mathcal{A}_{n_1}(\delta_1)=\bigcap_{1\leqslant i\leqslant n_1}\bigcap_{1\leqslant j\leqslant m}
\left\{
	\left\vert
	\left(
		\frac{WX}{\sqrt{n_0}}
	\right)_{ij}
	\right\vert
	\leqslant (\log n_1)^{1/2+\delta_1}
\right\}.
\end{equation}
On this event, we have, considering the approximation \eqref{eq:approxpoly}, 
\[
\rho(\mathcal{E})\leqslant C_f \sqrt{{m}} \frac{(\log n_1)^{k(1/2+\delta_1)(1+c_f)}}{k!}.
\]
We then choose \begin{equation}\label{eq:choixn}
k
\geqslant
c_0
\frac{\log n_1}{\log \log n_1}
\quad\text{with}\quad
c_0>\frac{1}{2(1-(1+c_f)(\frac{1}{2}+\delta_1))}.
\end{equation} We obtain, by using Stirling formula, that there exists a $\delta_2>0$ such that for any $\varepsilon>0$ we have
\[
\rho(\mathcal{E})=\mathcal{O}\left(
\frac{n_1^\varepsilon}{n_1^{\delta_2}}
\right).
\]
By taking $\varepsilon$ small enough we then see that, on the event $\mathcal{A}_{n_1}(\delta_1)$ and with $k$ as in \eqref{eq:choixn}, $\rho(\mathcal{E})\to 0$ as $n_1\to \infty$. It remains to see that the event $\mathcal{A}_{n_1}(\delta_1)$ occurs with high probability which comes from the assumption  on the entries $W_{ij}$ and $X_{ij}$. Indeed,
\begin{equation}
\mathds{P}\left(
	\mathcal{A}_{n_1}(\delta_1)^c
\right)
=
\mathds{P}\left(
	\exists \:i,j\text{ such that }
	\left\vert	
	\left(
		\frac{WX}{\sqrt{n_0}}
	\right)_{ij}
	\right\vert
	>
	(\log n_1)^{1/2+\delta_1}
\right)
\leqslant Cn_1me^{-\frac{(\log n_1)^{1+2\delta_1}}{2}}
\end{equation}
which goes to zero faster than any polynomial in $n_1$. Now we know the limiting e.e.d. of the matrix $M_{P_k}$ constructed with the centered polynomial $P_k$ as activation function. The above argument yields it is the same for $M_{f-a_k}$ constructed with $f-a_k$ instead. 
Now $Y^{(a_k)}$ is just a rank one deformation of $Y$ and by the rank inequalities (see \cite{bai2010spectral} for instance), $M$ and $M_{f-a_k}$ have the same limiting e.e.d.. This finishes the proof of Theorem \ref{theo:result1}.
\end{proof}

\section{Propagation of eigenvalue distribution through multiple layers}\label{sec:multilayer}
In this section, we study the eigenvalue distribution of a nonlinear matrix model when the data passes through several layers of the neural network. The case of a single layer has been considered in Theorems \ref{theo:result1} and \ref{theo:result2} where we describe the asymptotic e.e.d. in the one layer case. It has been conjectured in \cite{pennington2017nonlinear} that the limiting e.e.d. is stable through the layers in the case where $\theta_2(f)=0$. We give here a positive answer to this conjecture (with the appropriate normalization).
We first develop the combinatorial arguments for an odd monomial of the form \begin{equation}\label{Ass1}f(x)=\frac{x^k}{k!}.\end{equation} for several layers. It can be shown as in Subsection \ref{subsec:even}  that the even monomial are subleading. Thus the leading order for moments is given by  the contribution of odd monomial only. From now on, we assume \eqref{Ass1} holds true. We can write the entries of the two layers data matrix $Y^{(2)}$ as 
\begin{equation}
Y^{(2)}_{ij}
=
\frac{1}{k!}\left(
	\frac{\sigma_x}{\sqrt{\theta_1(f)}}
	\frac{W^{(1)}Y^{(1)}}{\sqrt{n_1}}
\right)^k
=
\frac{\sigma_x^k}{n_1^{k/2}k!\theta_1(f)^{k/2}}
\sum_{\ell_1,\dots,\ell_{k}=1}^{n_1}
\prod_{p=1}^k
W_{i\ell_p}^{(1)}Y^{(1)}_{\ell_p j}.\label{eq:entryy2}
\end{equation}
Then, developing the expected moment of the e.e.d. and using \eqref{eq:entryy2},  we obtain the following

\begin{multline}\label{eq:momentpropag}
\frac{1}{n_2}\mathds{E}\left[\Tr \left(M^{(2)}\right)^q\right]=
\\=
\frac{\sigma_x^{2kq}}{n_2m^qn_1^{kq}(k!)^{2q}\theta_1(f)^{kq}}
\mathds{E}
\sum_{i_1,\dots,i_q}^{n_2}
\sum_{j_1,\dots,j_q}^{m}
\sum_{\substack{\ell_1^1,\dots\ell^1_k\\\dots\\\ell^{2q}_1\dots\ell^{2q}_k}}^{n_1}
\prod_{p=1}^kW_{i_1\ell^1_p}^{(1)}Y^{(1)}_{\ell_p^1j_1}
\prod_{p=1}^kW_{i_2\ell^2_p}^{(1)}Y^{(1)}_{\ell_p^2j_1}
\dots
\prod_{p=1}^kW_{i_1\ell^{2q}_p}^{(1)}Y^{(1)}_{\ell_p^{2q}j_q}.
\end{multline}
We call the terms contributing in a non negligible way \emph{typical}. Now, we can give a graphical representation of these terms as in the previous sections.We will see that the contributing graphs are actually the same admissible graphs from Definition \ref{def:graph}. However, there are less constraints in the choices of the \emph{blue} edges. Indeed, the entries of the matrix $Y^{(1)}$ are not independent: we do not need each entry to be matched with at least another. This constraint however holds for the entries of the matrix $W^{(1)}$. 

\subsection{The simpler case of the simple cycle}
In this subsubsection, we explain the combinatorics in the case where the $i$-labels and $j$-labels are pairwise distinct.
We first perform a matching on the entries of $W^{(1)}$. This matching on the $W^{(1)}$ entries induces one on the entries of $Y^{(1)}$. This matching thus induces another graph between $j$-labeled and $\ell$-labeled vertices. The $i$-labeled vertices do not appear in the graph (as they correspond to entries of $W^{(1)}$). This graph can be constructed from the initial graph by seeing which niches are connected by a \emph{blue} edge. Figure \ref{fig:matchbridge} explains this construction:  $\ell_2$ links the same niche adjacent to $j_2$ while $\ell_1$ links the niches adjacent to $j_1$ and $j_2$. 

\begin{figure}[!ht]
	\centering
	\begin{tikzpicture}
		\draw[BrickRed, line width=.12em] (0,0) circle (1cm);
		\node[fill=BrickRed, label=below:$i_{1}$, inner sep = 0pt, minimum size=.2cm] (40) at (360/4: 1cm) {};
		\node[fill=BrickRed, label=left:$j_{2}$, inner sep = 0pt, minimum size=.2cm] (41) at (2*360/4: 1cm) {};
		\node[fill=BrickRed, label=above:$i_{2}$, inner sep = 0pt, minimum size=.2cm] at (3*360/4: 1cm) {};
		\node[fill=BrickRed, label=right:$j_{1}$, inner sep = 0pt, minimum size=.2cm] at (4*360/4: 1cm) {};
		\foreach \b in {1,2,...,4}{
			\foreach \a in {1,2,3}{
			\node[fill=RoyalBlue, inner sep = 0pt, minimum size=.2cm] (\b+\a) at (\b*360/4+\a*360/16: 1cm) {};
			}
		}
		\draw[-, ForestGreen] (1+1) edge[bend left=80, line width=.12em] node[draw=none, midway, above=-2pt] {$\bm{\ell_1}$} (4+3);
		\draw[-, RoyalBlue] (1+2) edge[bend right=80, line width=.12em] node[draw=none, midway, above=-2pt] {$\bm{\ell_2}$} (1+3);
		\draw[-, RoyalBlue] (2+1) edge[bend right=80, line width=.12em] node[draw=none, midway, below=-2pt] {$\bm{\ell_3}$}  (2+2);
		\draw[-, ForestGreen] (2+3) edge[bend right=80, line width=.12em] node[draw=none, midway, below=-2pt] {$\bm{\ell_4}$}  (3+1);
		\draw[-, RoyalBlue] (3+2) edge[bend right=80, line width=.12em] node[draw=none, midway, below=-2pt] {$\bm{\ell_5}$}  (3+3);
		\draw[-, RoyalBlue] (4+1) edge[bend right=80, line width=.12em] node[draw=none, midway, above=-2pt] {$\bm{\ell_6}$} (4+2);

		\node[draw=none] at (5.5,0) {$\begin{gathered} \text{Corresponding moment:}\\Y_{\ell_1j_1}Y_{\ell_6j_1}^2Y_{\ell_5j_1}^2Y_{\ell_4j_1}Y_{\ell_4j_2}Y_{\ell_3j_2}^2Y_{\ell_2j_2}^2Y_{\ell_1j_2} \end{gathered}$};

		\draw[ForestGreen, line width=.12em] (11.5,0) circle (1cm);
		\node[fill=ForestGreen, label={[ForestGreen]above:$\bm{\ell_{1}}$}, inner sep = 0pt, minimum size=.2cm] (100) at (11.5,1) {};
		\node[fill=BrickRed, label=left:$j_{1}$, inner sep = 0pt, minimum size=.2cm] (200) at (12.5,0) {};
		\node[fill=ForestGreen, label={[ForestGreen]below:$\bm{\ell_{4}}$}, inner sep = 0pt, minimum size=.2cm] at (11.5,-1) {};
		\node[fill=BrickRed, label=right:$j_{2}$, inner sep = 0pt, minimum size=.2cm] (300) at (10.5,0) {};
		
		\node[fill=RoyalBlue, inner sep = 0pt, minimum size=.2cm, label={[RoyalBlue]above:$\bm{\ell_{2}}$}] (11) at (9.75,.5) {};
		\node[fill=RoyalBlue, inner sep = 0pt, minimum size=.2cm, label={[RoyalBlue]below:$\bm{\ell_{3}}$}] (12) at (9.75,-.5) {};
		\node[fill=RoyalBlue, inner sep = 0pt, minimum size=.2cm, label={[RoyalBlue]above:$\bm{\ell_{6}}$}] (13) at (13.25,.5) {};
		\node[fill=RoyalBlue, inner sep = 0pt, minimum size=.2cm, label={[RoyalBlue]below:$\bm{\ell_{5}}$}] (14) at (13.25,-.5) {};
		
		\draw[-, RoyalBlue] (200) edge[bend right=30, line width=.12em] (13);
		\draw[-, RoyalBlue] (200) edge[bend left=30, line width=.12em] (13);
		
		\draw[-, RoyalBlue] (200) edge[bend right=30, line width=.12em] (14);
		\draw[-, RoyalBlue] (200) edge[bend left=30, line width=.12em] (14);
		
		\draw[-, RoyalBlue] (300) edge[bend right=30, line width=.12em] (12);
		\draw[-, RoyalBlue] (300) edge[bend left=30, line width=.12em] (12);
		
		\draw[-, RoyalBlue] (300) edge[bend right=30, line width=.12em] (11);
		\draw[-, RoyalBlue] (300) edge[bend left=30, line width=.12em] (11);
	\end{tikzpicture}
	\vspace{-10ex}
	\caption{Graph obtained after a \emph{blue} matching in the initial graph. The \emph{green} edges, corresponding to bridges between niches, induce a cycle in the final graph. The remaining edges coming from matched pairs inside a niche create simple cycles attached to $j$ labeled indices.}
	\label{fig:matchbridge}
\end{figure}
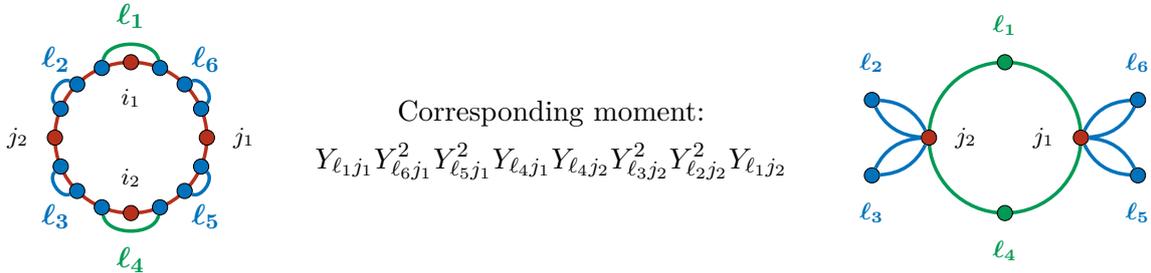


We start with general observations. The largest number of possible distinct $\ell$ indices is $kq$, which is obtained  as follows: 
One matches at least two indices from different adjacent niches of an $i$-label index and perform a perfect matching between the $2k-2$ remaining indices. Such a matching gives $kq$ different $\ell$ indices and matches every $W^{(1)}$ entry with another. This is illustrated in the leftmost graph in Figure \ref{fig:matchbridge}. Note that this type of matching gives $kq$ distinct $\ell$ indices but is actually not necessarily typical (see Figure \ref{fig:matchnonadmi} for an illustration) and is not the sole typical configuration.\\
As in Figure \ref{fig:matchbridge}, we see that the matching on the initial graph induces another admissible graph. Note that it does not consist in one cycle but in a cycle (in green on the figure) where $k-1$ cycles of length 2 are attached to each $j$-labeled vertex.
Also one has to note that it is possible to perform identifications between the $\emph{blue}$ edges and obtain a graph contributing in a non negligible way to the asymptotic expansion (see Figure \ref{fig:bridgeadmi} for an illustration). This behavior is explained in the second step when we develop the entries of $Y^{(1)}$.
Let us briefly indicate, as in Figure \ref{fig:matchnonadmi}, a \emph{blue} matching on the initial cycle which maximizes the number of distinct indices may give rise to a non-admissible induced graph. This comes from the fact that too many edges link two distinct niches. 

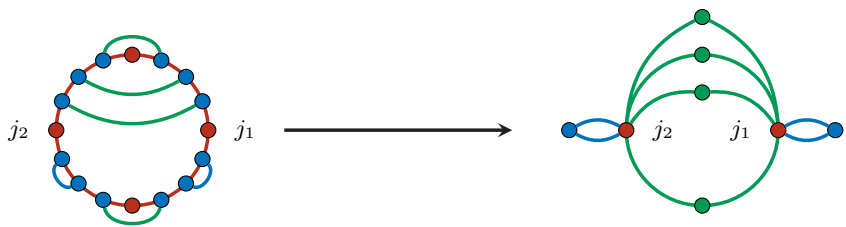
\begin{figure}[!ht]
	\centering
	\begin{tikzpicture}
		\draw[BrickRed, line width=.12em] (0,0) circle (1cm);
		\node[fill=BrickRed, inner sep = 0pt, minimum size=.2cm] (40) at (360/4: 1cm) {};
		\node[fill=BrickRed, label=left:$j_{2}$, inner sep = 0pt, minimum size=.2cm] (41) at (2*360/4: 1cm) {};
		\node[fill=BrickRed, inner sep = 0pt, minimum size=.2cm] at (3*360/4: 1cm) {};
		\node[fill=BrickRed, label=right:$j_{1}$, inner sep = 0pt, minimum size=.2cm] at (4*360/4: 1cm) {};
		\foreach \b in {1,2,...,4}{
			\foreach \a in {1,2,3}{
			\node[fill=RoyalBlue, inner sep = 0pt, minimum size=.2cm] (\b+\a) at (\b*360/4+\a*360/16: 1cm) {};
			}
		}	
		
		\draw[-, ForestGreen] (1+1) edge[bend left=80, line width=.12em] (4+3);
		\draw[-, ForestGreen] (1+2) edge[bend right, line width=.12em] (4+2);
		\draw[-, ForestGreen] (1+3) edge[bend right, line width=.12em] (4+1);
		\draw[-, RoyalBlue] (2+1) edge[bend right=80, line width=.12em]  (2+2);
		\draw[-, ForestGreen] (2+3) edge[bend right=80, line width=.12em]  (3+1);
		\draw[-, RoyalBlue] (3+2) edge[bend right=80, line width=.12em]  (3+3);	
		
		\draw[->, >=stealth, Black, line width=.12em] (2,0) -- (5,0)	;
		
		\draw[ForestGreen, line width=.12em] (7.5,0) circle (1cm);
		\node[fill=ForestGreen, inner sep = 0pt, minimum size=.2cm] (110) at (7.5,1) {};
		\node[fill=ForestGreen, inner sep = 0pt, minimum size=.2cm] (100) at (7.5,.5) {};
		\node[fill=ForestGreen, inner sep = 0pt, minimum size=.2cm] (120) at (7.5,1.5) {};

		\node[fill=BrickRed, label=right:$j_{2}$, inner sep = 0pt, minimum size=.2cm] (200) at (6.5,0) {};
		\node[fill=ForestGreen, inner sep = 0pt, minimum size=.2cm] at (7.5,-1) {};
		\node[fill=BrickRed, label=left:$j_{1}$, inner sep = 0pt, minimum size=.2cm] (210) at (8.5,0) {};
		\node[fill=RoyalBlue, inner sep = 0pt, minimum size=.2cm] (11) at (5.75,0) {};
		\node[fill=RoyalBlue, inner sep = 0pt, minimum size=.2cm] (12) at (9.25,0) {};
		
		\draw[-, RoyalBlue] (200) edge[bend right=30, line width=.12em] (11);
		\draw[-, RoyalBlue] (200) edge[bend left=30, line width=.12em] (11);
		
		\draw[-, RoyalBlue] (210) edge[bend right=30, line width=.12em] (12);
		\draw[-, RoyalBlue] (210) edge[bend left=30, line width=.12em] (12);
		
		\draw[-, ForestGreen] (200) edge[bend left=30, line width=.12em] (100);
		\draw[-, ForestGreen] (200) edge[bend left=30, line width=.12em] (120);
		\draw[-, ForestGreen] (210) edge[bend right=30, line width=.12em] (100);
		\draw[-, ForestGreen] (210) edge[bend right=30, line width=.12em] (120);
	\end{tikzpicture}
	\caption{Non-admissible graph obtained after a \emph{blue} matching which induces a maximum number of distinct indices in the initial cycle. We can see that several \emph{green} bridges between the same niches create a non-admissible graph and is thus subleading via the analysis from the previous section.}
	\label{fig:matchnonadmi}
\end{figure}

The main tool to understand the combinatorial arguments for the multilayer case is the following Lemma. It states that the leading order is actually given by the matchings as in Figure \ref{fig:matchbridge}.
\begin{lemma}\label{lem:leadmatch}
Consider a cycle of length $q> 2$, then the typical matchings on the \emph{blue} vertices consist in the following: \\
i)Two niches adjacent to the same $i$-labeled vertex are linked by a single edge called a bridge.\\ 
ii) Remaining edges inside a niche are matched according to a perfect matching.\\
iii) We can add identifications between \emph{bridges} only.\\
If the cycle is of length $2$ then we perform a perfect matching between the $2k$ \emph{blue} vertices in the cycle.
\end{lemma}

\begin{proof}
The proof is based on the construction of the second graph and the fact that the typical graphs are admissible.
We first show that any other matching gives a non-admissible second layer graph. Firstly, more than one \emph{bridge} between two distinct niches breaks the tree structure and thus yields a non admissible graph. The same reasoning holds for possible identifications between bridges and a matched pair inside a niche. If we identify two matched pairs inside a niche, we can see via the construction of the graph that it creates double edges and we would obtain an entry of $Y^{(1)}$ to the power of 4. 
However, note that in the initial cycle of size $q$, we can add identifications between the $q$ \emph{bridges} and still keep the second graph admissible. This behavior is illustrated in Figure \ref{fig:bridgeadmi} where we perform identifications between bridges and still obtain an admissible graph.

We now need to show that the contribution of the matchings leading to  a non admissible graph is subleading. As in Subsection \ref{subsec:nonadmi}, we have additional identifications between the vertices and we need to choose the fundamental cycles as well as the way one runs through the graph. Suppose we have $I_\ell$ identifications between the $\ell$ vertices. Then if the graph was admissible we would have $I_\ell + q(k-1) + 1$ fundamental cycles in the induced graph on $(j,\ell)$ vertices. Thus, if the graph is non-admissible, we have at most $I_\ell+q(k-1)$ fundamental cycles.

Let $m\leqslant I_\ell+q(k-1)$ be the number of fundamental cycles of the induced graph. We denote by $\mathscr{C}_1,\dots,\mathscr{C}_b,\mathscr{C}_{b+1},\dots,\mathscr{C}_m$ its cycles such that if $\ell(\mathscr{C}_i)$ denotes the length of the cycle $\mathscr{C}_i$ we have: $\ell(\mathscr{C}_1)=\ldots=\ell(\mathscr{C}_b)=2$ and $\ell(\mathscr{C}_{b+1}),\dots,\ell(\mathscr{C}_m)>2$. One then has that $\sum_{i=1}^{m}\ell(\mathscr{C}_i)=2kq.$ Now, the contribution of such graphs (initial and induced), taking into account the normalization and the number of ways to run through the graph, is at most
\[
\frac{(1+o(1))C^{I_l-m+I_j}}{n_1m^qn_o^{kq+k^2q}}
n_1^q
m^q
n_0^{kq-I_\ell}
n_0^{kb}
n_0^{m-b+\frac{k-1}{2}\sum_{i=b+1}^m\ell(\mathscr{C}_i)}
=
\mathcal{O}\left(\left ( 
	\frac{C}{n_0}\right )^{(I_\ell+(k-1)q+1)-m}
\right),
\] 
for some constant $C.$
Indeed one has to choose the $i$- and $j$-labels of vertices in the initial cycle, the $\ell$ vertices in the initial graph with the constraint that there are $I_\ell$ identifications. Then, in the induced graph, there are at most $k$ indices in each cycle of length 2 and $1+(k-1)/2\ell(\mathscr{C}_i)$ indices in the cycle $\mathscr{C}_i$ for $i>b$. Thus the contribution is negligible due the constraint that $m\leqslant I_\ell+q(k-1)$.

Some (negligible) contribution depending on $k$ comes from the possible multiple cycles of length 2 attached together as in Figure \ref{fig:multsingedge}. Here the error is slightly bigger since the induced graph has $2kq$ edges instead of simply $2q$.
Fix a vertex $j_0$, if we match together $2p$ $\ell$-indices together in the niche adjacent to $j_0,$ using \eqref{eq:constraintqk},  the corresponding error is given by $\mathcal{O}(n_0(k(2p)^k/n_0)^p).$ However, up to $2k$ indices can be matched together so that the contribution of non admissible graphs in this case is given by
\[
\sum_{p=2}^k
n_0
\left(
	\frac{k(2p)^k}{n_0}
\right)^p=o(1)
\quad\text{for}\quad
k\leqslant \frac{\log n}{\log \log n}.
\]
It actually decays faster than any polynomial for such $k$. This finishes the proof of the Lemma.
\end{proof}
\begin{figure}[!ht]
	\centering
	\begin{tikzpicture}
		\draw[BrickRed, line width=.12em] (0,0) circle (1cm);
		\node[fill=BrickRed, label=above:$i_{1}$, inner sep = 0pt, minimum size=.2cm] (40) at (360/4: 1cm) {};
		\node[fill=BrickRed, label=left:$j_{2}$, inner sep = 0pt, minimum size=.2cm] (41) at (2*360/4: 1cm) {};
		\node[fill=BrickRed, label=below:$i_{2}$, inner sep = 0pt, minimum size=.2cm] at (3*360/4: 1cm) {};
		\node[fill=BrickRed, label=right:$j_{1}$, inner sep = 0pt, minimum size=.2cm] at (4*360/4: 1cm) {};
		\foreach \b in {1,2,...,4}{
			\foreach \a in {1,2,3}{
			\node[fill=RoyalBlue, inner sep = 0pt, minimum size=.2cm] (\b+\a) at (\b*360/4+\a*360/16: 1cm) {};
			}
		}
		\draw[-, ForestGreen] (1+1) edge[bend right=80, line width=.12em] node[draw=none, midway] (123) {} (4+3);
		\draw[-, RoyalBlue] (1+2) edge[bend right=80, line width=.12em] (1+3);
		\draw[-, RoyalBlue] (2+1) edge[bend right=80, line width=.12em] (2+2);
		\draw[-, ForestGreen] (2+3) edge[bend left=80, line width=.12em] node[draw=none, midway] (456) {} (3+1);
		\draw[-, RoyalBlue] (3+2) edge[bend right=80, line width=.12em] (3+3);
		\draw[-, RoyalBlue] (4+1) edge[bend right=80, line width=.12em] (4+2);
		\draw[-, ForestGreen, line width=.12em] (123.center) -- (456.center);

	    \draw[->, >=stealth, Black, line width=.12em] (2,0) -- (5,0)	;

		\node[fill=ForestGreen, inner sep = 0pt, minimum size=.2cm] (1) at (7.5,0) {};
		\node[fill=BrickRed, inner sep = 0pt, minimum size=.2cm, label=below:$j_2$] (2) at (6.5,0) {};
		\node[fill=RoyalBlue, inner sep = 0pt, minimum size=.2cm] (3) at (5.75,-.5) {};
		\node[fill=RoyalBlue, inner sep = 0pt, minimum size=.2cm] (4) at (5.75,.5) {};
		\node[fill=BrickRed, inner sep = 0pt, minimum size=.2cm,  label=below:$j_1$] (5) at (8.5,0) {};
		\node[fill=RoyalBlue, inner sep = 0pt, minimum size=.2cm] (6) at (9.25,.5) {};
		\node[fill=RoyalBlue, inner sep = 0pt, minimum size=.2cm] (7) at (9.25,-.5) {};

		\draw[-, ForestGreen] (1) edge[bend left=30, line width=.12em] (2);
		\draw[-, ForestGreen] (1) edge[bend right=30, line width=.12em] (2);

		\draw[-, RoyalBlue] (2) edge[bend left=30, line width=.12em] (3);
		\draw[-, RoyalBlue] (2) edge[bend right=30, line width=.12em] (3);

		\draw[-, RoyalBlue] (2) edge[bend left=30, line width=.12em] (4);
		\draw[-, RoyalBlue] (2) edge[bend right=30, line width=.12em] (4);

		\draw[-, ForestGreen] (1) edge[bend left=30, line width=.12em] (5);
		\draw[-, ForestGreen] (1) edge[bend right=30, line width=.12em] (5);

		\draw[-, RoyalBlue] (5) edge[bend left=30, line width=.12em] (6);
		\draw[-, RoyalBlue] (5) edge[bend right=30, line width=.12em] (6);

		\draw[-, RoyalBlue] (5) edge[bend left=30, line width=.12em] (7);
		\draw[-, RoyalBlue] (5) edge[bend right=30, line width=.12em] (7);
	\end{tikzpicture}
	\caption{Admissible graph after a matching with an identification between two \emph{bridges}. While two $\ell$ vertices are identified, the matching is of leading order as one more cycle is in the induced graph.}
	\label{fig:bridgeadmi}
\end{figure}
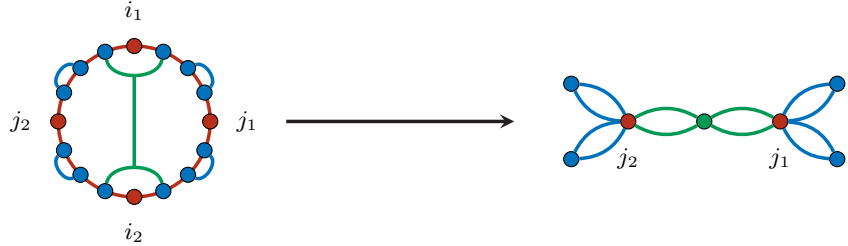
Lemma \ref{lem:leadmatch} has been proved for the two layers case. It can readily be extended to the case of $L\geq 2$ layers: the number of possible distinct $l$-indices is multiplied by $k$ at each layer and the final graph after performing matchings has to be admissible so that it contributes in the limit. The proof is similar to that of Lemma \ref{lem:leadmatch}. The detail is left to the reader.
\subsection{Invariance of the distribution in the case when \texorpdfstring{$\bm{\theta_2(f)}$}{} vanishes.}
In light of the previous combinatorial arguments, it is interesting to consider the special case where $\theta_2(f)=0$. Indeed,
for the one layer case, by Theorem \ref{theo:result1}, the limiting e.e.d. is the Mar\v{c}enko-Pastur distribution with shape $\frac{\phi}{\psi}$, denoted by $\mu_{\phi/\psi}$,  as proved also by the following lemma.
\begin{lemma}\label{lem:momentmarcenko} Let $q$ be a positive integer we have the following equality
\[
\sum_{\substack{I_j,I_i=0\\I_i+I_j+1=q}}^{q-1}
\mathcal{A}(q,I_i,I_j,q)
\psi^{1-q+I_i}\phi^{I_j}\theta_1^q(f)
=
\theta_1^q(f)\sum_{k=0}^{q-1}
\left(
	\frac{\phi}{\psi}
\right)^k\frac{1}{k+1}\binom{q}{k}\binom{q-1}{k}
=
\theta_1^q(f)
\langle x^q, \mu_{\phi/\psi}\rangle.
\]
\end{lemma}

\begin{proof}
Firstly, we can slightly rewrite the left hand side as
\[
\sum_{\substack{I_j,I_i=0\\I_i+I_j+1=q}}^{q-1}
\mathcal{A}(q,I_i,I_j,q)
\psi^{1-q+I_i}\phi^{I_j}\theta_1^q(f)
=
\theta_1^q(f)
\sum_{k=0}^{q-1}
\left(
	\frac{\phi}{\psi}
\right)^k
\mathcal{A}(q,q-k-1,k,q).
\]
Now there only remains to see that 
\begin{equation}\label{eq:narayana}
\mathcal{A}(q,q-k-1,k,q)
=
\frac{1}{k+1}\binom{q}{k}\binom{q-1}{k}.
\end{equation}
This fact comes from another representation of admissible graphs. Consider admissible graphs with $2q$ edges, $q$ cycles of length 2, $k$ $j$-identifications and $q-k-1$ $i$-identifications. Thus we can count this as double trees, in the sense that one of every two vertices are $i$-labeled and the others are $j$-labeled, with the appropriate number of each type of vertex ($q-k$ $j$-labeled vertices and $k+1$ $i$-labeled vertices). This number is known as a Narayana number \cite{chen2008identities} and given by \eqref{eq:narayana}. \end{proof}

This fact then means that if we consider a function $f$ such that $\theta_2(f)=0$, the e.e.d. (up to a change in variance and shape) is ``stable'' after going through one layer of the network. Indeed, if one considers the matrix $\frac{1}{m\sigma_x^2}XX^*$, the asymptotic e.e.d. is given by $\mu_{\phi}$ the Mar\v{c}enko-Pastur distribution with shape parameter $\phi$. Now, after a layer of the network, we see that for $\frac{1}{m\theta_1(f)}YY^*$ it is given by $\mu_{\phi/\psi}$. 

We now consider the case of an arbitrary fut fixed number of layers, mostly interested in the case where $\theta_2(f)=0$.
Let $Y^{(L+1)}$ be as in (\ref{def:yell}), and consider the matrices
\begin{equation}\label{eq:deffell}
M^{(L+1)}=\frac{1}{m\theta_1(f)}Y^{(L+1)}{Y^{(L+1)}}^*.
\end{equation}


\begin{theorem}\label{theo:resultmomentell} Let $ L$ be a given integer. 
Let $f=\sum_{k=1}^K \frac{a_k}{k!}(x^k-k!!\mathds{1}_{k\text{ even}})$ be a polynomial such that \eqref{eq:assumfgauss} holds. The degree of $f$, $K$, can grow with $n_1$ but suppose that 
$
K\leqslant \frac{1}{L-1}
	\frac{\log n_1}{\log \log n_1}.
$
Denote the e.e.d. of $M^{(L)}$ constructed as in \eqref{eq:deffell} by $\mu_{n_L}^{(L)}=\frac{1}{n_L}\sum_{i=1}^{n_L}\delta_{\lambda_i^{(\ell)}}$ and its expected moments by
$
\overline{m}_q^{(L)}:=\mathds{E}\left[
	\scp{\mu_{n_L}^{(L)}}{x^q}
\right] $
There exists a (non explicit) factor $T(q,k,L)$ such that \begin{equation}\label{eq:resultmomentell}
\overline{m}_q^{(L)}=
\left(
	\sum_{k=1}^{q-1}
	\left(
		\frac{\phi}{\prod_{i=0}^{L-1}\psi_i}
	\right)^k
	\frac{1}{k+1}
	\binom{q}{k}
	\binom{q-1}{k}
	+
	\theta_2(f)T(q,k,L)
\right)\left(1+o(1)\right).
\end{equation}
\end{theorem}

\begin{proof}
We again first develop the arguments in the case of a monomial of odd degree $f(x)=x^k$ since the case of an even monomial is completely similar (we only consider graphs with simple cycles).We study and count the admissible graphs along each layer. It is enough to identify in the asymptotic expansion of the moment those terms where no $\theta_2$ arises. Thus one can consider only admissible graphs made of cycles of length 2. 
{For the error terms one has to consider also admissible graphs with longer cycles but where the matching in each niche does not yield an occurence of $\theta_2$.  }

We begin with the case where $q=1$ for two layers. Then the cycle has length 2 as in Figure \ref{fig:q1}.  The dominant term in the asymptotic expansion consists in performing a perfect matching between all edges from Lemma \ref{lem:leadmatch}. The contribution coming from this first construction (in Lemma \ref{lem:leadmatch}) is given by
\[
\frac{\sigma_x^{2k}}{n_2mn_1^k}n_2mn_1^k\left(\sigma_{w}^{2k}(2k)!!\right)=\theta_1(f)(1+o(1)).
\]
This follows from the choices for the $i$ index, the $j$ index and the $\ell$ indices.
Now, this construction on the initial graph induces a second graph as in Figure \ref{fig:simpcycle}.  This induced graph is an admissible graph where all $j$'s are identified to a single vertex and $k$ cycles of length 2 are attached to it (corresponding to the $k$ \emph{blue} edges in the initial cycle).
We use the same reasoning as before and develop the entries $Y^{(1)}$ as a product of entries of $W^{(0)}$ and $X.$ Since the graph is admissible, the dominant term in the asymptotic expansion corresponds to performing a perfect matching in all cycles of length 2 as in Section \ref{sec:momentmethod} (illustrated in Figure \ref{fig:simpcycle}). Thus this adds a contribution of
 \[
 \frac{1}{n_0^{k^2}\theta_1(f)^{k}}n_0^{k^2}\left(\sigma_{w}^{2k}\sigma_x^{2k}(2k)!!\right)^{k}=1+o(1).
 \]
Here, the normalization in $n_0^{-k^2}$ comes from the fact there are $2k$ entries with a normalization of $n_0^{-k/2}$. We then have to choose $n_0^{k^2}$ indices in the second graph. Finally, we obtain for the final contribution for a cycle of length 2 that $E_1(f)=\theta_1\left(f\right)(1+o(1)).$

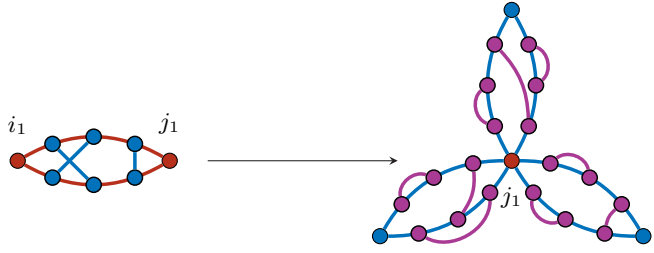
\begin{figure}[!ht]
	\centering
	\begin{tikzpicture}
		\node[fill=BrickRed, label=$i_{1}$, inner sep = 0pt, minimum size=.2cm] (10) at (0,0) {};
		\node[fill=BrickRed, label=$j_{1}$, inner sep = 0pt, minimum size=.2cm] (20) at (2,0) {};
		\draw[-, BrickRed, line width=.12em] (10) edge[bend left]  node[draw=black,line width=.05em,pos=.20, fill=RoyalBlue, inner sep = 0pt, minimum size = .2cm] (30) {} node[draw=black,line width=.05em,midway, fill=RoyalBlue, inner sep = 0pt, minimum size = .2cm] (40) {} node[draw=black,line width=.05em,pos=.80, fill=RoyalBlue, inner sep = 0pt, minimum size = .2cm] (50) {} (20) ;
		\draw[-, BrickRed, line width=.12em] (10) edge[bend right]  node[draw=black,line width=.05em,pos=.20, fill=RoyalBlue, inner sep = 0pt, minimum size = .2cm] (60) {} node[draw=black,line width=.05em,midway, fill=RoyalBlue, inner sep = 0pt, minimum size = .2cm] (70) {} node[draw=black,line width=.05em,pos=.80, fill=RoyalBlue, inner sep = 0pt, minimum size = .2cm] (80) {} (20) ;
		\draw[-, RoyalBlue, line width=.12em] (30) edge  (70);
		\draw[-, RoyalBlue, line width=.12em] (40) edge  (60);
		\draw[-, RoyalBlue, line width=.12em] (50) edge  (80);
		
		\draw[->,>=stealth,Black, label=above:$Y$] (2.5,0) -- (5,0);
		
		\node[fill=BrickRed, inner sep = 0pt, minimum size = .2cm, label=below:$j_1$] (1) at (6.5,0) {};
		\node[fill=RoyalBlue, inner sep = 0pt, minimum size = .2cm] (2) at (6.5,2) {};
		\node[fill=RoyalBlue, inner sep = 0pt, minimum size = .2cm] (3) at (4.768,-1) {};
		\node[fill=RoyalBlue, inner sep = 0pt, minimum size = .2cm,] (4) at (8.232,-1) {};

		\draw[-, RoyalBlue, line width=.12em] (1) edge[bend left]  node[draw=black,line width=.05em,pos=.20, fill=Mulberry, inner sep = 0pt, minimum size = .2cm] (5) {} node[draw=black,line width=.05em,midway, fill=Mulberry, inner sep = 0pt, minimum size = .2cm] (6) {} node[draw=black,line width=.05em,pos=.80, fill=Mulberry, inner sep = 0pt, minimum size = .2cm] (7) {} (2) ;
		\draw[-, RoyalBlue, line width=.12em] (1) edge[bend right]  node[draw=black,line width=.05em,pos=.20, fill=Mulberry, inner sep = 0pt, minimum size = .2cm] (8) {} node[draw=black,line width=.05em,midway, fill=Mulberry, inner sep = 0pt, minimum size = .2cm] (9) {} node[draw=black,line width=.05em,pos=.80, fill=Mulberry, inner sep = 0pt, minimum size = .2cm] (10) {} (2) ;
	
		\draw[-, RoyalBlue, line width=.12em] (1) edge[bend left]  node[draw=black,line width=.05em,pos=.20, fill=Mulberry, inner sep = 0pt, minimum size = .2cm] (11) {} node[draw=black,line width=.05em,midway, fill=Mulberry, inner sep = 0pt, minimum size = .2cm] (12) {} node[draw=black,line width=.05em,pos=.80, fill=Mulberry, inner sep = 0pt, minimum size = .2cm] (13) {} (3) ;
			\draw[-, RoyalBlue, line width=.12em] (1) edge[bend right]  node[draw=black,line width=.05em,pos=.20, fill=Mulberry, inner sep = 0pt, minimum size = .2cm] (14) {} node[draw=black,line width=.05em,midway, fill=Mulberry, inner sep = 0pt, minimum size = .2cm] (15) {} node[draw=black,line width=.05em,pos=.80, fill=Mulberry, inner sep = 0pt, minimum size = .2cm] (16) {} (3) ;
		
		\draw[-, RoyalBlue, line width=.12em] (1) edge[bend left]  node[draw=black,line width=.05em,pos=.20, fill=Mulberry, inner sep = 0pt, minimum size = .2cm] (17) {} node[draw=black,line width=.05em,midway, fill=Mulberry, inner sep = 0pt, minimum size = .2cm] (18) {} node[draw=black,line width=.05em,pos=.80, fill=Mulberry, inner sep = 0pt, minimum size = .2cm] (19) {} (4) ;
		\draw[-, RoyalBlue, line width=.12em] (1) edge[bend right]  node[draw=black,line width=.05em,pos=.20, fill=Mulberry, inner sep = 0pt, minimum size = .2cm] (20) {} node[draw=black,line width=.05em,midway, fill=Mulberry, inner sep = 0pt, minimum size = .2cm] (21) {} node[draw=black,line width=.05em,pos=.80, fill=Mulberry, inner sep = 0pt, minimum size = .2cm] (22) {} (4) ;
		
		\draw[-,Mulberry, line width=.12em] (5) edge[bend left=60,line width=.12em] (6);
		\draw[-,Mulberry, line width=.12em] (7) edge[bend left=20, line width=.12em] (8);
		\draw[-,Mulberry, line width=.12em] (9) edge[bend right=60, line width=.12em] (10);
		
		\draw[-,Mulberry, line width=.12em] (11) edge[bend left=60,line width=.12em] (13);
		\draw[-,Mulberry, line width=.12em] (12) edge[bend right=20, line width=.12em] (14);
		\draw[-,Mulberry, line width=.12em] (15) edge[bend right=60, line width=.12em] (16);
		
		\draw[-,Mulberry, line width=.12em] (17) edge[bend left=60,line width=.12em] (18);
		\draw[-,Mulberry, line width=.12em] (22) edge[bend left=20, line width=.12em] (19);
		\draw[-,Mulberry, line width=.12em] (20) edge[bend right=60, line width=.12em] (21);
	\end{tikzpicture}
	\caption{Construction/matching on the second layer graphs from a matching on the initial graph. The first graph gives a combinatorial factor of $(2k)!!$ while the second graph gives a factor of $(2k)!!^k$.}
	\label{fig:simpcycle}
\end{figure}

For the general case we saw that the first step of the procedure (by the construction explained before) yields a forest of \emph{star admissible graph} where each graph is given by a certain number of cycles of length 2 attached to a unique $j$-labeled vertex. 
Consider now a connected component (of the induced forest) which corresponds to a unique $j$ vertex. The number of cycles of length 2 attached to $j$ is then $k$ times the total number of cycles adjacent to $j$ in the previous steps (since we have $k$ \emph{blue} edges in each cycle of length 2). From this first process we then get the following contribution for this first two steps
\begin{multline*}
\frac{\sigma_x^{2kq+2k^2q}(1+o(1))}{n_{L} m^q \theta_1(f)^{q+kq+k^2q}}
\sum_{\substack{I_i,I_j\\I_i+I_j+1=q}}
\mathcal{A}(q,I_j,I_j,q)
n_{L}^{q-I_i}\times 
\frac{m^{q-I_j}}{n_{L-1}^{kq}}
n_{L-1}^{kq}
\left(
	\sigma_{w}^{2k}(2k)!!
\right)^q \frac{1}{n_{L-2}^{k^2q}}
n_{L-2}^{k^2q}
\left(
	\sigma_{w}^{2k}
	(2k)!!
\right)^{kq}\\
=
(1+o(1))
\sum_{k=0}^{q-1}
\mathcal{A}(q,q-k-1,k,q) 
\left(\frac{n_{L}}{m}\right)^k.
\end{multline*}
Let us explain the above formula: there are $n_{L}^{q-I_i}$ choices needed to label the $i$-labeled vertices and $m^{q-I_j}$ for the $j$-labeled vertices. For the powers of $n_{L-1}$ we take into account the normalization and the corresponding number of $\ell$ indices to choose. Finally in each cycle of length 2 we perform a perfect matching between the two niches: there are $q$ cycles of length 2 in the initial graph and $kq$ such cycles in the forest obtained. See Figure \ref{fig:simplelayer} for an illustration.

Now, we can perform one more step of the procedure, we now have a forest of these \emph{star admissible graphs} where each graph has only one $j$ vertex. To the $j$ vertex are now attached $k$ times more cycles than in the previous step. Thus, for the 3 step procedure, the total number of cycles of length 2 in the forest is given by $k^3q$. We can perform this for each layer the data goes through as the only parameter to be changed is the number of cycles of length 2 attached to each $j$ vertex.

\begin{figure}[!ht]
	\centering
	\begin{tikzpicture}
		\node[fill=BrickRed, inner sep = 0pt, minimum size=.2cm, label=below:$i_1$] (1) at (0,0) {};
		\node[fill=BrickRed, inner sep = 0pt, minimum size=.2cm, label=below:$j_1$] (2) at (1,0) {};
		\node[fill=BrickRed, inner sep = 0pt, minimum size=.2cm, label=below:$i_2$] (3) at (2,0) {};
		\node[fill=BrickRed, inner sep = 0pt, minimum size=.2cm, label=below:$j_2$] (4) at (3,0) {};

		\draw[-, BrickRed, line width=.12em] (1) edge[bend left]  node[draw=black,line width=.05em,pos=.20, fill=RoyalBlue, inner sep = 0pt, minimum size = .15cm] (11) {} node[draw=black,line width=.05em,midway, fill=RoyalBlue, inner sep = 0pt, minimum size = .15cm] (12) {} node[draw=black,line width=.05em,pos=.80, fill=RoyalBlue, inner sep = 0pt, minimum size = .15cm] (13) {} (2) ;
		\draw[-, BrickRed, line width=.12em] (1) edge[bend right]  node[draw=black,line width=.05em,pos=.20, fill=RoyalBlue, inner sep = 0pt, minimum size = .15cm] (14) {} node[draw=black,line width=.05em,midway, fill=RoyalBlue, inner sep = 0pt, minimum size = .15cm] (15) {} node[draw=black,line width=.05em,pos=.80, fill=RoyalBlue, inner sep = 0pt, minimum size = .15cm] (16) {} (2) ;
		
		\draw[-, BrickRed, line width=.12em] (2) edge[bend left]  node[draw=black,line width=.05em,pos=.20, fill=RoyalBlue, inner sep = 0pt, minimum size = .15cm] (21) {} node[draw=black,line width=.05em,midway, fill=RoyalBlue, inner sep = 0pt, minimum size = .15cm] (22) {} node[draw=black,line width=.05em,pos=.80, fill=RoyalBlue, inner sep = 0pt, minimum size = .15cm] (23) {} (3) ;
		\draw[-, BrickRed, line width=.12em] (2) edge[bend right]  node[draw=black,line width=.05em,pos=.20, fill=RoyalBlue, inner sep = 0pt, minimum size = .15cm] (24) {} node[draw=black,line width=.05em,midway, fill=RoyalBlue, inner sep = 0pt, minimum size = .15cm] (25) {} node[draw=black,line width=.05em,pos=.80, fill=RoyalBlue, inner sep = 0pt, minimum size = .15cm] (26) {} (3) ;
		
		\draw[-, BrickRed, line width=.12em] (3) edge[bend left]  node[draw=black,line width=.05em,pos=.20, fill=RoyalBlue, inner sep = 0pt, minimum size = .15cm] (31) {} node[draw=black,line width=.05em,midway, fill=RoyalBlue, inner sep = 0pt, minimum size = .15cm] (32) {} node[draw=black,line width=.05em,pos=.80, fill=RoyalBlue, inner sep = 0pt, minimum size = .15cm] (33) {} (4) ;
		\draw[-, BrickRed, line width=.12em] (3) edge[bend right]  node[draw=black,line width=.05em,pos=.20, fill=RoyalBlue, inner sep = 0pt, minimum size = .15cm] (34) {} node[draw=black,line width=.05em,midway, fill=RoyalBlue, inner sep = 0pt, minimum size = .15cm] (35) {} node[draw=black,line width=.05em,pos=.80, fill=RoyalBlue, inner sep = 0pt, minimum size = .15cm] (36) {} (4) ;

		\draw[->,>=stealth,Black, line width=.12em] (3.5,0) -- (4.5,0);
		
		\begin{scope}[xshift=6cm]
			\node[fill=BrickRed, inner sep = 0pt, minimum size=.2cm, label={[label distance=1em]left:$j_1$}] (01) at (0,0) {};
			\foreach \a in {1,2,...,6}{
				\node[fill=RoyalBlue, inner sep = 0pt, minimum size=.2cm] (\a+01) at (\a*360/6+30: 1cm) {};
			
				\draw[-, RoyalBlue, line width=.12em] (01) edge[bend left=30]  node[draw=black,line width=.05em,pos=.20, fill=Mulberry, inner sep = 0pt, minimum size = .1cm] (\a+1) {} node[draw=black,line width=.05em,midway, fill=Mulberry, inner sep = 0pt, minimum size = .1cm] (\a+2) {} node[draw=black,line width=.05em,pos=.80, fill=Mulberry, inner sep = 0pt, minimum size = .1cm] (\a+3) {} (\a+01) ;
				\draw[-, RoyalBlue, line width=.12em] (01) edge[bend right=30]  node[draw=black,line width=.05em,pos=.20, fill=Mulberry, inner sep = 0pt, minimum size = .1cm] (\a+4) {} node[draw=black,line width=.05em,midway, fill=Mulberry, inner sep = 0pt, minimum size = .1cm] (\a+5) {} node[draw=black,line width=.05em,pos=.80, fill=Mulberry, inner sep = 0pt, minimum size = .1cm] (\a+6) {} (\a+01) ;
		}
		\end{scope}
	
		\begin{scope}[xshift=8.5cm]
			\node[fill=BrickRed, inner sep = 0pt, minimum size=.2cm, label={[label distance=.5em]below:$j_2$}] (02) at (0,0) {};
			\foreach \a in {1,2,3}{
				\node[fill=RoyalBlue, inner sep = 0pt, minimum size=.2cm] (\a+02) at (\a*360/3+90: 1cm) {};
			
				\draw[-, RoyalBlue, line width=.12em] (02) edge[bend left=30]  node[draw=black,line width=.05em,pos=.20, fill=Mulberry, inner sep = 0pt, minimum size = .1cm] (\a+11) {} node[draw=black,line width=.05em,midway, fill=Mulberry, inner sep = 0pt, minimum size = .1cm] (\a+12) {} node[draw=black,line width=.05em,pos=.80, fill=Mulberry, inner sep = 0pt, minimum size = .1cm] (\a+13) {} (\a+02) ;
				\draw[-, RoyalBlue, line width=.12em] (02) edge[bend right=30]  node[draw=black,line width=.05em,pos=.20, fill=Mulberry, inner sep = 0pt, minimum size = .1cm] (\a+14) {} node[draw=black,line width=.05em,midway, fill=Mulberry, inner sep = 0pt, minimum size = .1cm] (\a+15) {} node[draw=black,line width=.05em,pos=.80, fill=Mulberry, inner sep = 0pt, minimum size = .1cm] (\a+16) {} (\a+02) ;
			}
		\end{scope}
		
		\draw[->,>=stealth, Black, line width=.12em] (9.5,0) -- (10.5,0);
		
		\begin{scope}[xshift=12cm]
			\node[fill=BrickRed, inner sep = 0pt, minimum size=.2cm] (03) at (0,0) {};
			
			\foreach \a in {1,2,...,18}{
				\node[fill=Mulberry, inner sep = 0pt, minimum size=.15cm] (\a+03) at (\a*360/18+30: 1cm) {};
				\draw[-,Mulberry, line width=.12em] (03) edge[bend left=10] (\a+03) {};
				\draw[-,Mulberry, line width=.12em] (03) edge[bend right=10] (\a+03) {};
		}
			\begin{scope}[xshift=2.5cm]
				\node[fill=BrickRed, inner sep = 0pt, minimum size=.2cm] (04) at (0,0) {};
			
				\foreach \a in {1,2,...,9}{
					\node[fill=Mulberry, inner sep = 0pt, minimum size=.15cm] (\a+04) at (\a*360/9+11: 1cm) {};
					\draw[-,Mulberry, line width=.12em] (04) edge[bend left=20] (\a+04) {};
					\draw[-,Mulberry, line width=.12em] (04) edge[bend right=20] (\a+04) {};
				}
			\end{scope}
		\end{scope}
		\draw[-, RoyalBlue] (11) edge[line width=.12em] (14);
		\draw[-, RoyalBlue] (12) edge[bend left = 60, line width=.12em] (13);
		\draw[-, RoyalBlue] (15) edge[bend right = 60,line width=.12em] (16);
		
		\draw[-, RoyalBlue] (21) edge[line width=.12em] (24);
		\draw[-, RoyalBlue] (22) edge[line width=.12em] (25);
		\draw[-, RoyalBlue] (23) edge[line width=.12em] (26);
		
		\draw[-, RoyalBlue] (31) edge[line width=.12em] (36);
		\draw[-, RoyalBlue] (34) edge[line width=.12em] (32);
		\draw[-, RoyalBlue] (35) edge[line width=.12em] (33);

		\draw[-, Mulberry] (1+1) edge[line width=.1em] (1+4);
		\draw[-, Mulberry] (1+5) edge[line width=.1em] (1+2);
		\draw[-, Mulberry] (1+6) edge[line width=.1em] (1+3);

		\draw[-, Mulberry] (2+1) edge[line width=.1em] (2+4);
		\draw[-, Mulberry] (2+2) edge[line width=.1em] (2+6);
		\draw[-, Mulberry] (2+3) edge[line width=.1em] (2+5);

		\draw[-, Mulberry] (3+1) edge[line width=.1em] (3+5);
		\draw[-, Mulberry] (3+2) edge[bend left=60, line width=.1em] (3+3);
		\draw[-, Mulberry] (3+4) edge[bend left=60, line width=.1em] (3+6);
		
		\draw[-, Mulberry] (4+1) edge[bend right= 60, line width=.1em] (4+2);
		\draw[-, Mulberry] (4+3) edge[line width=.1em] (4+6);
		\draw[-, Mulberry] (4+4) edge[bend left=60, line width=.1em] (4+5);
		
		\draw[-, Mulberry] (5+1) edge[line width=.1em] (5+6);
		\draw[-, Mulberry] (5+2) edge[line width=.1em] (5+4);
		\draw[-, Mulberry] (5+3) edge[line width=.1em] (5+5);
		
		\draw[-, Mulberry] (6+1) edge[line width=.1em] (6+6);
		\draw[-, Mulberry] (6+2) edge[bend right=60, line width=.1em] (6+3);
		\draw[-, Mulberry] (6+4) edge[bend left=60, line width=.1em] (6+5);
		
		\draw[-, Mulberry] (1+11) edge[line width=.1em] (1+14);
		\draw[-, Mulberry] (1+12) edge[line width=.1em] (1+16);
		\draw[-, Mulberry] (1+13) edge[line width=.1em] (1+15);
		
		\draw[-, Mulberry] (2+11) edge[line width=.1em] (2+16);
		\draw[-, Mulberry] (2+12) edge[line width=.1em] (2+15);
		\draw[-, Mulberry] (2+13) edge[line width=.1em] (2+14);
		
		\draw[-, Mulberry] (3+11) edge[line width=.1em] (3+14);
		\draw[-, Mulberry] (3+12) edge[line width=.1em] (3+15);
		\draw[-, Mulberry] (3+13) edge[line width=.1em] (3+16);
	\end{tikzpicture}
	\caption{Effect on going through several layers for admissible graphs with only cycles of length 2. The first step consists of separating each $j$-labeled vertex into his own graph where it is attached to cycles of length 2. At each layer after the first one, we multiply by $k$ the number of cycles attached. }
	\label{fig:simplelayer}
\end{figure}
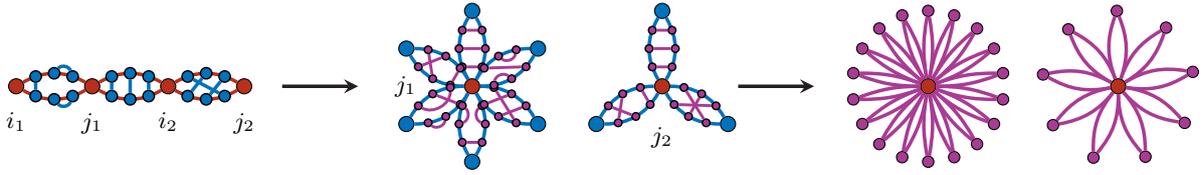   

In the whole, adding the layer $L_0$ multiplies the contribution with no $\theta_2$ by a factor
\[
\frac{1}{n_{L_0}^{k^{L-L_0} q}\theta_1^{k^{L-L_0}q}}n_{L_0}^{k^{L-L_0} q}
\theta_1^{k^{L-L_0}q}(f)
.
\]
Thus the whole contribution can be written in the following way :
\[
(1+o(1))\sum_{k=0}^{q-1}
\left(
	\frac{n_L}{m}
\right)^k
\mathcal{A}(q,q-k-1,k,q).
\]
And we obtain the final result by using that
$\frac{n_L}{m}
\rightarrow \frac{\phi}{\psi_0\psi_1\dots \psi_{L-1}}$
and
$\mathcal{A}(q,q-k-1,k,q)=\frac{1}{k+1}\binom{r}{k}\binom{r-1}{k}.$

Now, in the statement of the theorem we do not explicit the leading contribution of admissible graphs with at least one cycle of length greater than 2. We only need now to get an estimate on the other possible errors and show that they are negligible.
Using Subsection \ref{subsec:36}, the total number of cactus trees with $q$ edges does not exceed $C\xi^q$ for some constant $C$. As we are interested in the case where $\theta_2$ vanishes, it is enough to show that the error terms cannot grow faster than the Marcenko-Pastur moment. Actually using the arguments of Section 3, the whole analysis of errors remains true. The errors  can only come from subleading matchings on the graph at each possible step.  However, the main difference comes from the number of vertices at each step which is $k^{L_0}q$ instead of just $kq$. Note that it still only consists of a power of $k$ which grows slower than any power of $n_1$. 
Again, the leading contribution of the errors comes from possible multiple edges arising in the graph. Say that a given $j$ vertex is first connected to $r$ cycles of length 2 in the initial graph. At the step $L_0$, it is now connected to $k^{L_0-1}r$ cycles of length 2. Thus if at this stage we connect \emph{blue} indices together, say $p$ of them we obtain at the next step a multiple edge of multiplicity $2p$. We have a total of $2k^{L_0}r$ \emph{blue} indices to match at this stage since we have $2k$ vertices per cycle of length 2. Thus, by comparing the contribution of such matchings with the typical matching we obtain, similarly to \eqref{eq:constraintqk},
\[
\sum_{p=2}^{k^{L_0}r}n_0\left(\frac{Ckp^k}{n_0}\right)^p=o(1)\quad\text{for}\quad
k\leqslant \frac{1}{L_0}\frac{\log n_1}{\log \log n_1}.
\]
Now $L_0$ ranges from $1$ to $L-1$ so that we obtain the needed bound if $k\leqslant \frac{1}{L-1}\frac{\log n_1}{\log \log n_1}.$
\end{proof}

We now finish the proof of Theorem \ref{theo:marclayer}.
\begin{proof}[Proof of Theorem \ref{theo:marclayer}]
We have shown that for a polynomial of degree up to $\frac{1}{L-1}\frac{\log n_1}{\log\log(n_1)}$, the expected moments of the e.e.d. are those of the Mar\v{c}enko-Pastur distribution with the appropriate shape parameter. We first see that the variance of the moments is of order $k^L/n_1^2$ in order to show convergence of the actual moments. The principle is similar to that of Lemma \ref{lem:variance} as we count the corresponding graphs such that their covariance is non zero.

We can perform the same expansion as in Lemma \ref{lem:variance} and see that we have for the first layer
\begin{equation}\label{eq:variancepastur}
\mathrm{Var}\, m_q^{(L)}=
\frac{1}{n_1^2}\sum_{\mathcal{G}_1,\,\mathcal{G}_2}
\sum_{\bm{\ell_1},\bm{\ell_2}}
\mathds{E}\left[
	M_{\mathcal{G}_1}^{(L)}(\bm{\ell_1})M_{\mathcal{G}_2}^{(L)}(\bm{\ell_2})
\right]
-
\mathds{E}\left[
	M_{\mathcal{G}_1}^{(L)}(\bm{\ell_1})
\right]
\mathds{E}\left[
	M_{\mathcal{G}_2}^{(L)}(\bm{\ell_2})
\right]
\end{equation}
with 
\[
M_{\mathcal{G}}^{(L)}(\bm{\ell})=\sum_{k_1,\dots,k_{2q}=1}^K
\frac{a_{k_1}\dots a_{k_{2q}}}{m^qn_0^{\sum k_i/2}}
\prod_{p=1}^{k_1}W_{i_1\ell^1_p}^{(L)}Y_{\ell_p^1j_1}^{(L)}
\prod_{p=1}^{k_2}W_{i_2\ell^2_p}^{(L)}Y_{\ell_p^2j_1}^{(L)}
\dots
\prod_{p=1}^{k_{2q}}W_{i_1\ell^{2q}_p}^{(L)}Y_{\ell_p^{2q}j_q}^{(L)}.
\]
Now, in order to have a non vanishing contribution to the variance \eqref{eq:variancepastur}, we need to have additional identifications between the two graphs. Indeed, either at a given layer $L_0$ an entry of $W^{(L_0)}$ is matched between $\mathcal{G}_1$ and $\mathcal{G}_2$ or at the last layer there are identifications between the $X$ entries. 
In the case where there are identifications of $Y^{(L_0)}$ entries we see, by expanding the expansion with respect to the entries of $W^{(L_0-1)}$, that this implies that there are further identifications in the layers beyond $L_0$.
Since at each step we would lose an order $\mathcal{O}(q^2(k)^{2L_0)}/n_0)$ (from the choice of which vertices to identify and the fact that we have one less choice for possible indices),  we see that the leading order comes from identifying $X$ entries in the two last layers.

Thus, since the main contribution to moments are still given by admissible graphs, a similar analysis can be done as in Lemma \ref{lem:variance}: we can, right at the first layer, identify $i$ and $j$ vertices to obtain an identification on the $W^{(L)}$ entries. Or one can choose two $W^{(L_0)}$ entries to be identified at a given layer $L_0$ (or $X$ entries at the last layers $L_0=1$) and thus we obtain
\[
\mathrm{Var}\,m_q^{(L)}=\mathcal{O}\left(\frac{q^4+q^2\sum_{L_0=1}^Lk^{2L_0}+\sum_{L_0=1}^Lk^{4L_0}}{n_0^2}C^q \right)
=
\mathcal{O}\left(
\frac{k^{4L+4}}{n_0^2}
\right),
\]
since $q$ is fixed here. 

Let us now extend the result to a bounded function $f$. As in Section \ref{sec:polyapprox}, we consider a polynomial $P_k$ such that, for some $A>0,$
$
\sup_{x\in[-A,A]}\left\vert (f(x)-a_k)-P_k(x)\right\vert\leqslant C_f\frac{A^{(1+c_f)k}}{(n+1)!}.
$
Now, we can consider $Y^{(L,a_k)}$ the matrix constructed as \eqref{def:yell} with $f-a_k$ as an activation function and $Y^{(L, P_k)}$ the same matrix constructed with $P_k$. Note that we consider the same sampling of $W$ and $X$ for the construction of this model. We describe the case of $L=2$ as we can recursively do the same reasoning for a higher number of layers, for simplicity we also forget the change of variance $\sigma_x/\sqrt{\theta_1(f)}$ at each layer. As we saw in Section \ref{sec:polyapprox}, we simply need to bound
\[
\frac{1}{\sqrt{m}}\max_{1\leqslant i\leqslant n_2}\sum_{j=1}^m
\left\vert
Y^{(2,a_k)}_{ij}-Y^{(2, P_k)}_{ij}
\right\vert
=
\frac{1}{\sqrt{m}}
\max_{1\leqslant i\leqslant n_2}
\sum_{j=1}^m
\left\vert
f\left(
	\frac{W^{(1)}Y^{(1,a_k)}}{\sqrt{n_1}}
\right)_{ij}
-a_k
-
P_k\left(
	\frac{W^{(1)}Y^{(1,P_k)}}{\sqrt{n_1}}
\right)_{ij}
\right\vert.
\]
We split the right hand side into two parts and write
\begin{multline}\label{eq:splitapprox}
\left\vert
Y^{(2,a_k)}_{ij}-Y^{(2, P_k)}_{ij}
\right\vert
\\\leqslant
\left\vert
f\left(
	\frac{W^{(1)}Y^{(1,a_k)}}{\sqrt{n_1}}
\right)_{ij}
-
f\left(
	\frac{W^{(1)}Y^{(1,P_k)}}{\sqrt{n_1}}
\right)_{ij}
\right\vert+\left\vert
f\left(
	\frac{W^{(1)}Y^{(1,P_k)}}{\sqrt{n_1}}
\right)_{ij}
-a_k
-
P_n\left(
	\frac{W^{(1)}Y^{(1,P_k)}}{\sqrt{n_1}}
\right)_{ij}
\right\vert .
\end{multline}
For the first term on the right hand side of the previous equation, we bound it from the polynomial approximation. Indeed, we consider the following event
\[
\mathcal{A}_1(\delta_1)
=
\bigcap_{i=1}^{n_1}
\bigcap_{j=1}^m
\left\{
	\left\vert
		\left(
			\frac{W^{(0)}X}{\sqrt{n_0}}
		\right)_{ij}
	\right\vert
	\leqslant (\log n_1)^{1/2+\delta_1}	
\right\}
\bigcap
\left\{
\left\vert
	W^{(1)}_{ij}
\right\vert
\leqslant (\log n)^{1/\alpha+\delta_1}
\right\}.
\]
This event occurs with overwhelming probability for any $\delta_1>0$ in the sense that its probability decays faster than any polynomial. Now, on this event we can bound
\[
\left\vert
\left(
	\frac{W^{(1)}Y^{(1,a_k)}}{\sqrt{n_1}}
\right)_{ij}
-
\left(
	\frac{W^{(1)}Y^{(1,P_n)}}{\sqrt{n_1}}
\right)_{ij}
\right\vert
\leqslant
C^n\sqrt{n_1}(\log n_1)^{1/\alpha+\delta_1}\frac{(\log n_1)^{(1/2+\delta_1)n}}{n!},
\]
where we expand the entries and use the polynomial approximation. This also decays faster than any polynomial for $n=\mathcal{O}(\frac{\log n_1}{\log\log n_1})$. 
Finally, using the fact that $f$ has a bounded derivative on the event $\mathcal{A}_2(\delta_2)$ defined in \eqref{eq:eventa2}, the first term in \eqref{eq:splitapprox} goes to zero providing that $\mathcal{A}_2$ occurs with high probability.

For the second term in \eqref{eq:splitapprox}, by the previous analysis and as in Section \ref{sec:polyapprox} we only need to prove that the following event occurs with probability tending to one:
\begin{equation}\label{eq:eventa2}
\mathcal{A}_2(\delta_2)
=
\bigcap_{i=1}^{n_2}
\bigcap_{j=1}^m
\left\{
	\frac{1}{\sqrt{n_1}}
	\sum_{\ell_1=1}^{n_1}
	W_{i\ell_1}^{(1)}
	P_n\left(
		\frac{1}{\sqrt{n_0}}
		\sum_{\ell_0=1}^{n_0}
		W_{\ell_1\ell_0}^{(0)}
		X_{\ell_0 j}
	\right)
	\leqslant
	(\log n_1)^{1/2+\delta_1}
\right\}.
\end{equation}
Since we suppose that $f$ is bounded we know that on the event $\mathcal{A}_1(\delta_1)$ (which occurs with very high probability) we have that $\sup_{ij}\vert Y^{(1,P_k)}_{ij}\vert\leqslant C.$ Besides, since $W_{i\ell_1}^{(1)}$ has zero expectation, has a sub-Gaussian tail and is independent of the entries of $W^{(0)}$ and $X$, the random variable $(W^{(1)}Y^{(1)})_{ij}$ is sub-Gaussian as well. So that we obtain that there exists a $C>0$ such that 
\[
\mathds{P}\left(
\sum_{\ell_1=1}^{n_1}
	W_{i\ell_1}^{(1)}
	P_n\left(
		\frac{1}{\sqrt{n_0}}
		\sum_{\ell_0=1}^{n_0}
		W_{\ell_1\ell_0}^{(0)}
		X_{\ell_0 j}
	\right)
	>
	\sqrt{n_1}(\log n_1)^{1/2+\delta_1}
\right)
\leqslant
Ce^{-c(\log n_1)^{1+2\delta_1}}.
\] 
And finally $\mathds{P}(\mathcal{A}_2(\delta_2))\geqslant 1-n_1^{-D}$ for any $D>0$.
\end{proof}


\begin{bibdiv}
\begin{biblist}

\bib{akemann1}{article}{
author={Akemann, G.},
author={Burda, Z.},
title={ Universal microscopic correlation functions for products of independent Ginibre matrices.},
journal={ J. Phys. A Math. Theor.},
volume={ 45},
number = {46},
	pages = {465201},
	year={2012},
	}	
	
\bib{akemann2}{article}{	
author={Akemann, G.},
author={Ipsen, J.R.},
author={Kieburg, M.},
title={Products of rectangular random matrices: singular values and progressive scattering.},
journal={ Phys. Rev. E},
volume={ 88},
pages={ 052118},
year={2013},
}


\bib{bai2010spectral}{book}{
   author={Bai, Z.},
   author={Silverstein, J. W.},
   title={Spectral analysis of large dimensional random matrices},
   series={Springer Series in Statistics},
   edition={2},
   publisher={Springer, New York},
   date={2010},
   pages={xvi+551},
}
	

\bib{benaych2016spectral}{article}{
   author={Benaych-Georges, F.},
   author={Couillet, R.},
   title={Spectral analysis of the Gram matrix of mixture models},
   journal={ESAIM Probab. Stat.},
   volume={20},
   date={2016},
   pages={217--237},
   issn={1292-8100},
}

\bib{cebron2016universal}{article}{
  title={Universal constructions for spaces of traffics},
  author={C{\'e}bron, G.},
  author={Dahlqvist, A.},
  author={Male, C.},
  journal={arXiv preprint},
  year={2016},
}


\bib{chen2008identities}{article}{
   author={Chen, W. Y. C.},
   author={Yan, S. H. F.},
   author={Yang, L. L. M.},
   title={Identities from weighted Motzkin paths},
   journal={Adv. in Appl. Math.},
   volume={41},
   date={2008},
   number={3},
   pages={329--334},
}

\bib{choromanska2015loss}{article}{
	author = {Choromanska, A.},
	author = {Henaff, M.},
	author = {Mathieu, M.},
	author = {Ben Arous, G.},
	author = {LeCun, Y.},
    title  = {The loss surfaces of multilayer networks},
    journal = {Proceedings of the 18th International Conference on Artificial Intelligence and Statistics, {AISTATS} 2015},
  	year = {2015},
}
\bib{Zdeborova}{article}{
title={Machine learning and the physical sciences},
author={ Cirac, C.}, author={  Cranmer,K. }, author={ Daudet, L.},
author={ Schuld, M.}, 
author={ Tishby, N.}, author={ Vogt-Maranto, L.},
author={ Zdeborov{\'a}, L. },
 journal={arXiv preprint arXiv:1903.10563},
  year={2019},
}

\bib{claeys2015correlation}{article}{
  title={Correlation kernels for sums and products of random matrices},
  author={Claeys, T.},
  author={Kuijlaars, A. B. J.},
  author={Wang, D.},
  journal={Random Matrices: Theory and Applications},
  volume={4},
  number={04},
  year={2015},
}

\bib{couillet2016kernel}{article}{
   author={Couillet, R.},
   author={Benaych-Georges, F.},
   title={Kernel spectral clustering of large dimensional data},
   journal={Electron. J. Stat.},
   volume={10},
   date={2016},
   number={1},
   pages={1393--1454},
   issn={1935-7524},
}



\bib{dupic2014spectral}{article}{
  title={Spectral density of products of Wishart dilute random matrices. Part I: the dense case},
  author={Dupic, T.},
  author={Castillo, I. P.},
  journal={arXiv preprint},
  year={2014}
}

\bib{elkaroui2010spectrum}{article}{
   author={El Karoui, N.},
   title={The spectrum of kernel random matrices},
   journal={Ann. Statist.},
   volume={38},
   date={2010},
   number={1},
   pages={1--50},
   issn={0090-5364},
}

\bib{ford1956combinatorial}{article}{
   author={Ford, G. W.},
   author={Uhlenbeck, G. E.},
   title={Combinatorial problems in the theory of graphs. III},
   journal={Proc. Nat. Acad. Sci. U.S.A.},
   volume={42},
   date={1956},
   pages={529--535},
   issn={0027-8424},
}
\bib{Forrester}{article}{
       author = {Forrester, P.J.},
       author = {Liu, D.Z.},
        title = {Raney Distributions and Random Matrix Theory},
      journal = {Journ. Stat. Phys.},
         year = {2015},
       volume = {158},
       number = {5},
        pages = {1051-1082},
}

\bib{furedi1981eigenvalues}{article}{
   author={F\"{u}redi, Z.},
   author={Koml\'{o}s, J.},
   title={The eigenvalues of random symmetric matrices},
   journal={Combinatorica},
   volume={1},
   date={1981},
   number={3},
   pages={233--241},
}
	
\bib{giryes2016deep}{article}{ 
author={Giryes, R.},
author={Sapiro, G.},
author={Bronstein, A. M.}, 
journal={IEEE Transactions on Signal Processing}, 
title={Deep Neural Networks with Random Gaussian Weights: A Universal Classification Strategy?}, 
year={2016}, 
volume={64}, 
number={13}, 
pages={3444-3457}, 
}

\bib{glorot}{article}{
  title = 	 {Understanding the difficulty of training deep feedforward neural networks},
  author = 	 {Glorot, X.},
  author =   {Bengio, Y.},
  booktitle = 	 {Proceedings of the Thirteenth International Conference on Artificial Intelligence and Statistics},
  pages = 	 {249--256},
  year = 	 {2010},
  volume = 	 {9},
  publisher = 	 {PMLR},
}

\bib{hanin2018products}{article}{
  title={Products of Many Large Random Matrices and Gradients in Deep Neural Networks},
  author={Hanin, B.},
  author={Nica, M.},
  journal={arXiv preprint},
  year={2018}
}

\bib{hayou2019selection}{article}{
title={On the Selection of Initialization and Activation Function for Deep Neural Networks},
author={Hayou, S.},
author={Doucet, A.},
author={Rousseau, J.},
year={2019},
journal={Prepublication},
}

\bib{hinton2012deep}{article}{
  title={Deep neural networks for acoustic modeling in speech recognition: The shared views of four research groups},
  author={Hinton, G.},
  author={Deng, L.},
  author={Yu, D.},
  author={Dahl, G. E.},
  author={Mohamed, A.-R.},
  author={Jaitly, N.},
  author={Senior, A.},
  author={Vanhoucke, V.},
  author={Nguyen, P.},
  author={Sainath, T. N.},
  author={Others, },
  journal={IEEE Signal processing magazine},
  volume={29},
  number={6},
  pages={82--97},
  year={2012},
}







\bib{ioffe2015batch}{article}{
 author = {Ioffe, S.},
 author = {Szegedy, C.},
 title = {Batch Normalization: Accelerating Deep Network Training by Reducing Internal Covariate Shift},
 booktitle = {Proceedings of the 32Nd International Conference on International Conference on Machine Learning - Volume 37},
 year = {2015},
 pages = {448--456},
} 

\bib{krizhevsky2012imagenet}{article}{
	title = {ImageNet classification with deep convolutional neural networks},
	author = {Krizhevsky, A.},
	author = {Sutskever, I.},
	author = {Hinton, G. E.},
	journal = {Advances in Neural Information Processing Systems},
	number = {25},
pages = {1097--1105},
year = {2012},
}

\bib{kuijlaars2014singular}{article}{
  title={Singular values of products of Ginibre random matrices, multiple orthogonal polynomials and hard edge scaling limits},
  author={Kuijlaars, A. B. J.},
  author={Zhang, L.},
  journal={Communications in Mathematical Physics},
  volume={332},
  number={2},
  pages={759--781},
  year={2014},
}
\bib{LeCun}{article}{
author={ LeCun, Y.},
author={Bengio, Y.}, 
author={Hinton, G.},
title={Deep learning},
journal={Nature} ,
volume={521}, 
pages={ 436--444 },
year={2015},
}

\bib{LouartCouilletBis}{article}{
author={Louart, C.},
author={Couillet, R.},
title={ Concentration of Measure and Large Random Matrices with an application to Sample Covariance Matrices.},
journal={ arXiv preprint arXiv:1805.08295},
year={2018}
}

\bib{louart2018random}{article}{
   author={Louart, C.},
   author={Liao, Z.},
   author={Couillet, R.},
   title={A random matrix approach to neural networks},
   journal={Ann. Appl. Probab.},
   volume={28},
   date={2018},
   number={2},
   pages={1190--1248},
   issn={1050-5164},
}

\bib{marchenko1967distribution}{article}{
   author={Mar\v{c}enko, V. A.},
   author={Pastur, L. A.},
   title={Distribution of eigenvalues in certain sets of random matrices},
   language={Russian},
   journal={Mat. Sb. (N.S.)},
   volume={72 (114)},
   date={1967},
   pages={507--536},
}
\bib{Montanari}{article}{
author={ Zhou, F. },
author={ Montanari, A.},
title={ The spectral norm of random inner-product kernel matrices.},
journal={ Prob. Theory Rel. Fields},
volume={ 173},
number={1-2}, 
date={2019},
pages={ 27-85},
}
\bib{peche2009universality}{article}{
   author={P\'{e}ch\'{e}, S.},
   title={Universality results for the largest eigenvalues of some sample
   covariance matrix ensembles},
   journal={Probab. Theory Related Fields},
   volume={143},
   date={2009},
   number={3-4},
   pages={481--516},
   issn={0178-8051},
}
\bib{PecheECP}{article}{
   author={P\'{e}ch\'{e}, S.},
   title={A note on the Pennington-Worah distribution},
   journal={Elec. Comm. Probab.},
   volume={24},
   date={2019},
   number={66},
   pages={7 pp.},
}


\bib{pennington2017geometry}{article}{
  title = 	 {Geometry of Neural Network Loss Surfaces via Random Matrix Theory},
  author = 	 {Pennington, J.},
  author = {Bahri, Y.},
  booktitle = 	 {Proceedings of the 34th International Conference on Machine Learning},
  pages = 	 {2798--2806},
  year = 	 {2017},
  volume = 	 {70},
  series = 	 {Proceedings of Machine Learning Research},
  publisher = 	 {PMLR},
}

\bib{pennington2017nonlinear}{article}{
  title={Nonlinear random matrix theory for deep learning},
  author={Pennington, J.},
  author={Worah, P.},
  booktitle={Advances in Neural Information Processing Systems},
  pages={2637--2646},
  year={2017}
}

\bib{Penson2011}{article}{
title={Product of Ginibre matrices: Fuss-Catalan and Raney distributions},
author={Penson,K.},
author={Zyczkowski, K.},
journal={Phys. Rev. E},
volume={83},
number={6},
pages={061118},
year={2011},
}
\bib{rajan2006eigenvalue}{article}{
  title={Eigenvalue spectra of random matrices for neural networks},
  author={Rajan, K.},
  author={Abbott, L. F.},
  journal={Physical review letters},
  volume={97},
  number={18},
  pages={188104},
  year={2006},
  publisher={APS}
}

\bib{schmidhuber2015deep}{article}{
title = {Deep learning in neural networks: An overview},
journal = {Neural Networks},
volume = {61},
pages = {85 - 117},
year = {2015},
issn = {0893-6080},
author = {Schmidhuber, J.},
}

\bib{silverstein1995empirical}{article}{
   author={Silverstein, J. W.},
   author={Bai, Z. D.},
   title={On the empirical distribution of eigenvalues of a class of
   large-dimensional random matrices},
   journal={J. Multivariate Anal.},
   volume={54},
   date={1995},
   number={2},
   pages={175--192},
}

\bib{Soshnikov1999universality}{article}{
   author={Soshnikov, A.},
   title={Universality at the edge of the spectrum in Wigner random
   matrices},
   journal={Comm. Math. Phys.},
   volume={207},
   date={1999},
   number={3},
   pages={697--733},
   issn={0010-3616},
}


\bib{wu2016google}{article}{
  author = {Wu, Y.},
  author = {Schuster, M.},
  author = {Chen, Z.},
  author = {Le, Q. V.},
  author = {Norouzi, M.},
  author = {Macherey, W.},
  author = {Krikun, M.},
  author = {Cao, Y.},
  author = {Gao, Q.},
  author = {Macherey, K.},
  author = {Others, },
  journal = {arXiv preprint},
  title = {Google's neural machine translation system: Bridging the gap between human and machine translation},
  year = {2016}
}

\bib{zhang2012nonlinear}{article}{
  title={Nonlinear system modeling with random matrices: echo state networks revisited},
  author={Zhang, B.},
  author={Miller, D. J.},
  author={Wang, Y.},
  journal={IEEE trans. Neural Netw. Learn. Syst.},
  volume={23},
  number={1},
  pages={175--182},
  year={2012},
  publisher={IEEE}
}
\end{biblist}
\end{bibdiv}
\end{document}